\documentclass[a4paper,12pt]{amsart}

\usepackage{mathtools}
\usepackage{amssymb}
\usepackage{amsthm}
\usepackage{bm}
\usepackage{mathrsfs}
\usepackage{mathcomp}

\usepackage{enumitem}

\newlist{enumarabic}{enumerate}{5}
\setlist[enumarabic]{label=\arabic*}
\newlist{enumarabicd}{enumerate}{5}
\setlist[enumarabicd]{label=\arabic*.}
\newlist{enumarabicp}{enumerate}{5}
\setlist[enumarabicp]{label=(\arabic*)}

\newlist{enumAlph}{enumerate}{5}
\setlist[enumAlph]{label=\Alph*}
\newlist{enumAlphd}{enumerate}{5}
\setlist[enumAlphd]{label=\Alph*.}
\newlist{enumAlphp}{enumerate}{5}
\setlist[enumAlphp]{label=(\Alph*)}

\newlist{enumalph}{enumerate}{5}
\setlist[enumalph]{label=\alph*}
\newlist{enumalphd}{enumerate}{5}
\setlist[enumalphd]{label=\alph*.}
\newlist{enumalphp}{enumerate}{5}
\setlist[enumalphp]{label=(\alph*)}

\newlist{enumRoman}{enumerate}{5}
\setlist[enumRoman]{label=\Roman*}
\newlist{enumRomand}{enumerate}{5}
\setlist[enumRomand]{label=\Roman*.}
\newlist{enumRomanp}{enumerate}{5}
\setlist[enumRomanp]{label=(\Roman*)}

\newlist{enumroman}{enumerate}{5}
\setlist[enumroman]{label=\roman*}
\newlist{enumromand}{enumerate}{5}
\setlist[enumromand]{label=\roman*.}
\newlist{enumromanp}{enumerate}{5}
\setlist[enumromanp]{label=(\roman*)}

\usepackage{graphicx}
\usepackage[dvipsnames,svgnames,table]{xcolor}

\usepackage{float}
\usepackage{array}
\usepackage{booktabs}
\usepackage{tabularx}
\usepackage{longtable}
\usepackage{multirow}

\usepackage{tikz}
\usetikzlibrary{arrows}
\usetikzlibrary{calc}
\usetikzlibrary{intersections}
\usetikzlibrary{patterns}
\usetikzlibrary{quotes}
\usetikzlibrary{shapes}
\usetikzlibrary{through}
\usepackage{tikz-cd}
\usepackage{mdframed}

\usepackage[
  bookmarks=true,%
  bookmarksnumbered=true,%
  setpagesize=false,%
  colorlinks=true,%
  linkcolor=Blue,%
  citecolor=YellowOrange,%
  urlcolor=HotPink,%
]{hyperref}

\usepackage{cleveref}

\theoremstyle{plain}

\newtheorem{theorem}{Theorem}[section]
\newtheorem{proposition}[theorem]{Proposition}

\newtheorem{lemma}[theorem]{Lemma}
\newtheorem{claim}[theorem]{Claim}

\theoremstyle{definition}

\newtheorem{definition}[theorem]{Definition}

\theoremstyle{remark}

\newtheorem{remark}[theorem]{Remark}
\newtheorem{example}[theorem]{Example}

\newtheorem*{acknowledgement}{Acknowledgement}

\crefname{theorem}{Theorem}{Theorem}
\Crefname{theorem}{Theorem}{Theorem}
\crefname{proposition}{Proposition}{Proposition}
\Crefname{proposition}{Proposition}{Proposition}
\crefname{corollary}{Corollary}{Corollary}
\Crefname{corollary}{Corollary}{Corollary}
\crefname{lemma}{Lemma}{Lemma}
\Crefname{lemma}{Lemma}{Lemma}
\crefname{claim}{Claim}{Claim}
\Crefname{claim}{Claim}{Claim}
\crefname{fact}{Fact}{Fact}
\Crefname{fact}{Fact}{Fact}
\crefname{definition}{Definition}{Definition}
\Crefname{definition}{Definition}{Definition}
\crefname{notation}{Notation}{Notation}
\Crefname{notation}{Notation}{Notation}
\crefname{remark}{Remark}{Remark}
\Crefname{remark}{Remark}{Remark}
\crefname{example}{Example}{Example}
\Crefname{example}{Example}{Example}
\crefname{problem}{Problem}{Problem}
\Crefname{problem}{Problem}{Problem}
\crefname{answer}{Answer}{Answer}
\Crefname{answer}{Answer}{Answer}
\crefname{section}{Section}{Section}
\Crefname{section}{Section}{Section}
\crefname{assumption}{Assumption}{Assumption}
\Crefname{assumption}{Assumption}{Assumption}

\newcommand{\parenlr}[1]{\left(#1\right)}
\newcommand{\bracketlr}[1]{\left[#1\right]}

\newcommand{\abracket}[1]{\langle#1\rangle}

\newcommand{\abs}[1]{\lvert#1\rvert}

\newcommand{\norm}[1]{\lVert#1\rVert}

\newcommand{\set}[2]{\{#1\mid#2\mbox{}\}}

\newcommand{\map}[3]{#1\colon#2\to#3}
\newcommand{\restr}[2]{#1|_{#2}}

\newcommand{\id}{\mathrm{id}}

\DeclareMathOperator{\Aut}{Aut}

\DeclareMathOperator{\Image}{Im} 

\DeclareMathOperator{\Ker}{Ker}

\DeclareMathOperator{\rank}{rank}
\DeclareMathOperator{\RePart}{Re} 

\newcommand{\mbfL}{\mathbf{L}}

\newcommand{\mbfS}{\mathbf{S}}

\newcommand{\mbfh}{\mathbf{h}}

\newcommand{\mcalA}{\mathcal{A}}
\newcommand{\mcalB}{\mathcal{B}}
\newcommand{\mcalC}{\mathcal{C}}
\newcommand{\mcalD}{\mathcal{D}}
\newcommand{\mcalE}{\mathcal{E}}
\newcommand{\mcalF}{\mathcal{F}}
\newcommand{\mcalG}{\mathcal{G}}
\newcommand{\mcalH}{\mathcal{H}}

\newcommand{\mcalO}{\mathcal{O}}
\newcommand{\mcalP}{\mathcal{P}}

\newcommand{\mcalS}{\mathcal{S}}

\newcommand{\mcalV}{\mathcal{V}}
\newcommand{\mcalW}{\mathcal{W}}

\newcommand{\mscrF}{\mathscr{F}}
\newcommand{\mscrG}{\mathscr{G}}

\newcommand{\mscrS}{\mathscr{S}}

\newcommand{\mscrV}{\mathscr{V}}
\newcommand{\mscrW}{\mathscr{W}}

\newcommand{\mfrakS}{\mathfrak{S}}

\newcommand{\mfraks}{\mathfrak{s}}
\newcommand{\mfrakt}{\mathfrak{t}}
\newcommand{\mfraku}{\mathfrak{u}}

\newcommand{\mbbP}{\mathbb{P}}

\newcommand{\mbbX}{\mathbb{X}}

\newcommand{\Z}{\mathbb{Z}}

\newcommand{\R}{\mathbb{R}}
\newcommand{\C}{\mathbb{C}}
\newcommand{\HB}{\mathbb{H}}

\usepackage{empheq, comment}
\numberwithin{equation}{section}

\newcommand{\hplus}[1]{\mcalH^+(#1)}
\newcommand{\shplus}[1]{S(\hplus{#1})}

\DeclareMathOperator{\Diff}{Diff}
\DeclareMathOperator{\Iso}{Iso}

\DeclareMathOperator{\Fr}{Fr}

\DeclareMathOperator{\HAut}{HAut}
\DeclareMathOperator{\HIso}{HIso}

\DeclareMathOperator{\univ}{univ}
\DeclareMathOperator{\pt}{pt}
\DeclareMathOperator{\Herm}{Herm}
\DeclareMathOperator{\Gr}{Gr}

\newcommand{\Diffplus}{\Diff^+(X)}
\newcommand{\Diffspin}{\Diff^+(X, \mfraks)}

\newcommand{\BDiffplus}{B\Diffplus}

\newcommand{\EDiffplus}{E\Diffplus}

\newcommand{\Xunivplus}{\mbbX^+_{\univ}}

\newcommand{\tiota}{\tilde\iota}
\newcommand{\tf}{\tilde f}

\newcommand{\tmfrakt}{\tilde\mfrakt}
\newcommand{\tmcalD}{\tilde\mcalD}
\newcommand{\tmcalE}{\tilde\mcalE}

\newcommand{\tmcalV}{\tilde \mcalV}
\newcommand{\tmcalW}{\tilde\mcalW}

\newcommand{\tmcalF}{\tilde \mcalF}

\newcommand{\tg}{\tilde g}

\newcommand{\pprime}{{\prime\prime}}

\begin{document}

\title{A gerbe-like construction in gauge theory}
\author[Mitsuyoshi Adachi]{Mitsuyoshi Adachi}
\address{Graduate School of Mathematical Sciences, the University of Tokyo, 3-8-1 Komaba, Meguro, Tokyo 153-8914, Japan}
\email{adachi-mitsuyoshi302@g.ecc.u-tokyo.ac.jp}

\begin{abstract}
  In 2022 Baraglia and Konno showed the following: for a smooth family of a homotopy $K3$ surface $X \to \mathbb{X} \stackrel{\pi}{\to} B$, if the tangent bundle along the fibers $T_B \mathbb{X}$ admits a spin structure, then $\mathcal{H}^+(\mathbb{X})$ also admits a spin structure, where $\mathcal{H}^+(\mathbb{X})$ is the vector bundle consisting of self-dual harmonic 2-forms. In this paper, we show that $T_B \mathbb{X} \oplus \pi^\ast \mathcal{H}^+(\mathbb{X})$ admits a canonical spin structure. The proof is carried out by canonically constructing a lifting $O(1)$-gerbe for the spin structure on $\mathcal{H}^+(\mathbb{X})$ using the families Seiberg--Witten equations, starting from a lifting $O(1)$-gerbe for the spin structure on $T_B \mathbb{X}$.
\end{abstract}

\maketitle

\tableofcontents

\section{Introduction}\label{sec1:intro}

Gauge theory is a powerful tool for studying the differential topology of 4-manifolds. The Seiberg--Witten theory utilizes non-linear partial differential equations called the Seiberg--Witten equations, which are introduced by Witten\cite{Witten-monopoles-1994}. The Seiberg--Witten invariants, important numerical invariants of 4-manifolds obtained from the Seiberg--Witten theory, assign integers to closed, oriented 4-manifolds equipped with a spin$^c$ structure. The computation of the Seiberg--Witten invariants has revealed that there exist countably infinitely many closed, oriented 4-manifolds that are homeomorphic but not diffeomorphic to a $K3$ surface.

Gauge theory is also applicable to the study of smooth families of 4-manifolds (fiber bundles whose fibers are smooth 4-manifolds $X$ and whose structure group is $\Diffplus$; see \cref{rem1:smooth fam} for the precise formulation) (\cite{Ruberman-an-obstruction1998,Ruberman-polynomial-invarint-of-diffeo-1999,Liu-family-blow-up2000,Ruberman-positive-scalar-curvature-2001,Li--Liu-family-wall-crossing2001,Nakamura-family-SW2003,Szymik-Characteristic-cohomotopy-classes-family-2010,Baraglia-Obstructions-to-smooth-actions-2019,Baraglia--Konno-gluing2020,Kronheimer--Mrowka-Dehn-twist-K3-2020,Baraglia-Constraints-on-families-2021,Kato--Konno--Nakamura-Rigidity-mod2-family-SW-2021,Konno-char-class2021,konno--taniguchi2021groups,Baraglia--Konno-2022,Konno--Nakamura-constraints2023,Konno--Lin-2022-arXiv,Lin-nonsymplectic-loops-2022-arXiv,baraglia-mod-2023-arxiv}). This was initiated by Ruberman\cite{Ruberman-an-obstruction1998,Ruberman-polynomial-invarint-of-diffeo-1999,Ruberman-positive-scalar-curvature-2001}.

One of the objects of study in investigating a smooth family $X \to \mbbX \to B$ of 4-manifolds is the real vector bundle $\hplus{\mbbX} \to B$ associated with $\mbbX$. This is constructed when fixing a continuous family of Riemannian metrics on $\mbbX$. The fiber of $\hplus{\mbbX}$ at each point $b \in B$ is the vector space of self-dual harmonic 2-forms on $\mbbX_b$. The purpose of this paper is to examine the additional structures that can be constructed on $\hplus{\mbbX}$ when $X$ is a homotopy $K3$ surface.

\subsection{Main results}\label{ssec1:main res}

We first state the main results. Here, we present one of the central results, and the remaining results are deferred to \cref{sec2:main thms}.

\begin{theorem}\label{thm1:spin on TX+H^+}
  Let $X$ be a homotopy $K3$ surface and $B$ a CW complex. Consider a smooth family of $X$
  \[
    X \to \mbbX \stackrel{\pi}{\to} B
  \]
  associated to a principal $\Diffplus$-bundle $\mcalE \to B$. Choose a continuous family of smooth Riemannian metrics on $\mbbX$. Fix an orientation on $\hplus{\mbbX}$. Then, we can construct a canonical spin structure on the vector bundle over $\mbbX$
  \[
    T_B \mbbX \oplus \pi^\ast \hplus{\mbbX}
  \]
  where
  \[
    T_B \mbbX \to \mbbX
  \]
  denotes the tangent bundle along the fibers.
\end{theorem}

It is well-known that $\hplus{\mbbX}$ is orientable when $X$ is a homotopy $K3$ surface. In \cref{ssec1:prev research}, we review the proof since it is a toy model of \cref{thm1:spin on TX+H^+}.

The statement of \cref{thm1:spin on TX+H^+} is stronger than merely specifying the isomorphism class of the spin structure. To the author's knowledge, this is the first work that constructs a canonical geometric object through gauge theory.

\cref{thm1:spin on TX+H^+} can be rephrased using the notion of a gerbe. The explanation of gerbes is given in \cref{rem6:gerbe}. The important point is that an $O(1)$-gerbe $\mscrG_{\mcalE}$ appears as an obstruction to lifting the principal $\Diffplus$-bundle
\[
  \mcalE \to B
\]
to a principal $\Diffspin$-bundle. (The definition of $\Diffspin$ is given in \cref{def1:str grp}.) From the viewpoint of $O(1)$-gerbes, \cref{thm1:spin on TX+H^+} can be rephrased as follows:
\begin{itemize}
  \item From the information of $\mscrG_{\mcalE}$, another $O(1)$-gerbe $\mscrG_{\hplus{\mbbX}}$ is canonically constructed. $\mscrG_{\hplus{\mbbX}}$ is the obstruction for $\hplus{\mbbX}$ to admit a spin structure.
  \item $\mscrG_{\mcalE}$ and $\mscrG_{\hplus{\mbbX}}$ are canonically isomorphic as $O(1)$-gerbes.
\end{itemize}

\subsection{Comparison with previous work}\label{ssec1:prev research}

We introduce two previous results that form the background of \cref{thm1:spin on TX+H^+}. First, we present a result essentially due to Morgan--Szab\'{o}\cite{Morgan--Szabo-homotopy-K3-1997}.

\begin{theorem}[{Morgan--Szab\'{o}\cite[Theorem 1.1]{Morgan--Szabo-homotopy-K3-1997}}]\label{ref-thm1:MS}
  Let $X$ be a homotopy $K3$ surface and $X \to \mbbX \to B$ a smooth family of $X$ over a CW complex $B$. Choose a continuous family of smooth Riemannian metrics on $\mbbX$. Then, $\hplus{\mbbX}$ has a canonically determined orientation as a vector bundle.
\end{theorem}

The original statement of Morgan--Szab\'{o}\cite[Theorem 1.1]{Morgan--Szabo-homotopy-K3-1997} is that the Seiberg--Witten invariant of a homotopy $K3$ surface for a spin structure is odd. Thus, the orientation of $\hplus{\mbbX}$ is determined so that the Seiberg--Witten invariant for the spin structure on each fiber is positive.

\cref{ref-thm1:MS} implies that $w_1(\hplus{\mbbX})$ vanishes. For $w_2(\hplus{\mbbX})$, there is a result by Baraglia--Konno\cite{Baraglia--Konno-2022}.

\begin{theorem}[{Baraglia--Konno\cite[Corollary 4.21]{Baraglia--Konno-2022}}]\label{ref-thm1:w_2=0}
  Let $X$ be a homotopy $K3$ surface. Consider a smooth family $X \to \mbbX \to B$ of $X$ over a CW complex $B$. Choose a continuous family of smooth Riemannian metrics on $\mbbX$. If the tangent bundle along the fibers $T_B \mbbX$ admits a spin structure, then $w_2(\hplus{\mbbX})$ vanishes.
\end{theorem}

Baraglia--Konno\cite[Theorem 1.3]{Baraglia--Konno-2022} has shown a stronger result than \cref{ref-thm1:w_2=0}: if $X$ satisfies certain assumptions and $T_B \mbbX$ admits a spin$^c$ structure, then $w_2(\hplus{\mbbX})$ equals the first Chern class of the family of the Dirac operators. Their strategy is to apply algebraic topology techniques to a finite-dimensional approximation of the Seiberg--Witten map. More precisely, they compute the Steenrod squares of the cohomology classes on $B$ called the families Seiberg--Witten invariants, and \cref{ref-thm1:w_2=0} follows as a corollary.

\cref{thm1:spin on TX+H^+} can be regarded as a generalization of Baraglia--Konno\cite[Corollary 4.21]{Baraglia--Konno-2022} in two senses. One is that \cref{thm1:spin on TX+H^+} obtains a canonical geometric object related to $\hplus{\mbbX}$. While the result of Baraglia--Konno\cite[Theorem 1.3]{Baraglia--Konno-2022} suggests that a spin structure exists on $\hplus{\mbbX}$, their algebraic topology approach does not lead to an explicit construction.

The other is that the assertion of \cref{thm1:spin on TX+H^+} holds even when the tangent bundle along the fibers $T_B \mbbX$ does not necessarily admit a spin or spin$^c$ structure. The approach of Baraglia--Konno\cite{Baraglia--Konno-2022} is applicable only when at least a spin$^c$ structure is present. When $X$ is a genuine $K3$ surface, there exists a family for which $T_B \mbbX$ does not admit a spin or spin$^c$ structure, so this difference is meaningful. As a consequence of \cref{thm1:spin on TX+H^+}, we also obtain a result identifying $w_2(\hplus{\mbbX})$ with a different cohomology class, even in cases where $T_B \mbbX$ does not necessarily admit a spin structure (\cref{thm1:alpha = w_2}).

\subsection{Outline of the proof}
We outline the strategy for proving \cref{thm1:spin on TX+H^+}. For simplicity, we consider a family $X \to \mbbX \to B$ of $X$ where $T_B \mbbX$ is equipped with a spin structure. In this case, the goal is to construct a spin structure on $\hplus{\mbbX}$. (This statement is part of \cref{thm1:spin on H^+} and is also an important step in the proof of \cref{thm1:spin on TX+H^+}.) The basic idea is as follows:
\begin{enumarabicp}
  \item When the base space $B$ satisfies certain nice properties, a spin structure on $\hplus{\mbbX}$ has a translation in terms of geometric data on $\shplus{\mbbX}$ (using the fact that $\hplus{\mbbX}$ has rank 3). The data consists of a complex line bundle $L$ over $\shplus{\mbbX}$ and an anti-linear $\Z/4$-action on $L$ with some additional properties. The goal is to construct such data.
  \item From the given spin structure on $T_B \mbbX$, we obtain a $Pin(2)$-equivariant family of the Seiberg--Witten maps.
  \item Noting the appearance of $\shplus{\mbbX}$ in (1), we consider a family of perturbations of the Seiberg--Witten map parameterized by $\shplus{\mbbX}$. This yields a finite-dimensional approximation parameterized by $\shplus{\mbbX}$.
  \item From the finite-dimensional approximation, we construct a family of Clifford bundles parameterized by $\shplus{\mbbX}$ via a procedure based on the construction of $K$-theoretic mapping degree. Due to the $Pin(2)$-equivariance of the Seiberg--Witten map, this family is equipped with an $\Z/4$-action. These data yield the geometric data mentioned in (1).
\end{enumarabicp}
To ensure that this construction produces a canonical spin structure, we need to pay attention to the following points:
\begin{itemize}
  \item In (4), the family of Clifford bundles needs to be constructed canonically from the finite-dimensional approximation. This is necessary for the constructed spin structure to be canonical. This is where our method differs from a naive application of algebraic topology techniques.
  \item There are two issues regarding the choice of the finite-dimensional approximation. One is that the finite-dimensional approximation may not be unique due to the possibility of stabilization. The other is that since we do not assume the compactness of the base space $B$, the finite-dimensional approximation in (3) may not be globally constructible over $B$. For the former issue, we must show that the spin structures constructed from two different finite-dimensional approximations are canonically isomorphic. The latter issue is resolved as follows: even if a finite-dimensional approximation is not globally constructible, it can be constructed locally on sufficiently small open subsets of $B$. Moreover, we can cover $B$ with such open subsets. By gluing the spin structures on $\hplus{\mbbX}$ constructed on each open subset, we finally construct a spin structure on $B$.
\end{itemize}

\subsection{Organization of the paper}\label{ssec1:org}
The organization of this paper is as follows. In \cref{sec2:main thms}, we state the precise statements of the main theorems.
In \cref{sec3:spin}, we characterize spin structures on oriented real vector bundles of rank 3.
In \cref{sec4:spin from fda}, we review the construction of the finite-dimensional approximation of the Seiberg--Witten map. After that, we construct a spin structure on $\hplus{\mbbX}$ when fixing a finite-dimensional approximation.
In \cref{sec5:swap fda}, we prove that when two different finite-dimensional approximations are given for the same base space, there exists a canonical isomorphism between the spin structures on $\hplus{\mbbX}$ constructed from each.
In \cref{sec6:prf main thm}, we prove that even when a finite-dimensional approximation cannot be globally constructed, we can glue the spin structures constructed on each open subset to canonically construct a spin structure on $\hplus{\mbbX}$.

\begin{acknowledgement}
  Part of the paper is based on the author's master thesis. The author would like to express his deep gratitude to his supervisor, Mikio Furuta, for helpful suggestions and continued encouragement during this work. The author would like to express his appreciation to Jin Miyazawa, Nobuo Iida, Hokuto Konno, Shinichiroh Matsuo, Taketo Sano, Nobuhiro Nakamura and Masaki Taniguchi for having discussions and helpful comments. This research was supported by Forefront Physics and Mathematics Program to Drive Transformation (FoPM), a World-leading Innovative Graduate Study (WINGS) Program, the University of Tokyo.
\end{acknowledgement}

\section{Statement of the main theorems}\label{sec2:main thms}
In this section, we state the main theorems of this paper: \cref{thm1-0:spin on TX+H^+}, \cref{thm1:alpha = w_2}, and \cref{thm1:spin on H^+}. \cref{thm1-0:spin on TX+H^+} is a restatement of \cref{thm1:spin on TX+H^+}.

Throughout this section, let $X$ denote a homotopy $K3$ surface. Also, let $\mathfrak{s}$ be a spin structure on $X$. In this paper, we formulate spin structures without using Riemannian metrics on $X$.

First, we define necessary concepts.

\begin{definition}\label{def1:str grp}
  Let $\Diffplus$ be the group of all orientation-preserving self-diffeomorphisms of $X$. Also, let
  \begin{align*}
    \Diffspin & = \set{(f, \tf)}{f \in \Diffplus,                                                                                          \\
              & \text{$\map{\tf}{\mfraks}{\mfraks}$ is a $\widetilde{GL}^+(4, \R)$-equivariant lift of $\map{df}{\Fr^+(TX)}{\Fr^+(TX)}$}}.
  \end{align*}
  We endow $\Diffplus$ and $\Diffspin$ with the $C^\infty$ topology. With these topologies and group structures, $\Diffplus$ and $\Diffspin$ are topological groups.
\end{definition}

\begin{remark}\label{rem1:smooth fam}
  Let $\mcalE \to B$ be a principal $\Diffplus$-bundle. We define the associated fiber bundle $\mbbX$ by
  \[
    \mbbX = \mcalE \times_{\Diffplus} X.
  \]
  Since the structure group of $\mcalE$ is $\Diffplus$, each fiber $\mbbX_b$ of $\mbbX$ at $b \in B$ is equipped with a smooth structure and is diffeomorphic to $X$. The precise meaning of ``a smooth family of $X$'' is a fiber bundle with fiber $X$ associated with some principal $\Diffplus$-bundle.
\end{remark}

\begin{definition}
  Let $X \to \mbbX \to B$ be a family of $X$ associated with a principal $\Diffplus$-bundle $\mcalE \to B$ over a CW complex $B$. We choose a continuous family of smooth Riemannian metrics on $\mbbX$ with respect to the $C^\infty$ topology. Then, we define a rank 3 vector bundle over $B$ by
  \[
    \hplus{\mbbX} = \coprod_{b \in B} \hplus{\mbbX_b}.
  \]
  Here, $\hplus{\mbbX_b}$ denotes the real vector space of all self-dual harmonic 2-forms on $\mbbX_b$.
\end{definition}

\begin{remark}\label{rem1:ori of H^+}
  The construction of $\hplus{\mbbX}$ is possible when $X$ is a general oriented closed 4-manifold. In that case, $\hplus{\mbbX}$ is a vector bundle of rank $b^+(X)$.

  When $X$ is a homotopy $K3$ surface, we can canonically give the orientation of $\hplus{\mbbX}$ so that the sign of the Seiberg--Witten invariant for the spin structure becomes positive. We assume that the orientation of $\hplus{\mbbX}$ is determined in this way.
\end{remark}

First, we restate \cref{thm1:spin on TX+H^+} using the universal $\Diffplus$-bundle.

\begin{theorem}\label{thm1-0:spin on TX+H^+}
  Let $X$ be a homotopy $K3$ surface. Let
  \[
    X \to \Xunivplus \stackrel{\pi}{\to} \BDiffplus
  \]
  be the smooth family of $X$ associated with the universal principal $\Diffplus$-bundle
  \[
    \EDiffplus \to \BDiffplus.
  \]
  Choose a continuous family of smooth Riemannian metrics on $\Xunivplus$ with respect to the $C^\infty$ topology. Then, we can construct a canonical spin structure on the vector bundle
  \[
    T_{\BDiffplus} \Xunivplus \oplus \pi^\ast \hplus{\Xunivplus}
  \]
  over $\Xunivplus$. Here
  \[
    T_{\BDiffplus} \Xunivplus \to \Xunivplus
  \]
  denotes the tangent bundle of $\Xunivplus$ along the fibers.
\end{theorem}

\begin{remark}
  In this paper, we assume that $\BDiffplus$ is a CW complex. That is, $\BDiffplus$ classifies principal $\Diffplus$-bundles over CW complexes.
\end{remark}

The second main theorem is an equality of characteristic classes of $\BDiffplus$ obtained from \cref{thm1:spin on TX+H^+}. One of them is $w_2(\hplus{\Xunivplus})$. We explain the definition of the other characteristic class. First, when $X$ is a homotopy $K3$ surface, its spin structure is unique up to isomorphism, and so there is the following central extension:
\[
  1 \to \Z/2 \to \Diffspin \to \Diffplus \to 1.
\]

\begin{definition}\label{def1:alpha}
  Let $X \to \mbbX \to B$ be a family of $X$ associated with a principal $\Diffplus$-bundle $\mcalE \to B$ over a CW complex $B$. We denote the primary obstruction class for lifting the structure group of $\mcalE$ to $\Diffspin$ by
  \[
    \alpha(\mbbX, \mfraks) \in H^2(B; \Z/2).
  \]
\end{definition}

\begin{remark}
  In fact, the vanishing of $\alpha(\mbbX, \mfraks)$ is equivalent to $\mcalE$ lifting to a principal $\Diffspin$-bundle. Also, $\mcalE$ lifting to a principal $\Diffspin$-bundle is equivalent to $T_B \mbbX$ admitting a spin structure.
\end{remark}

\begin{remark}\label{rem1:char of alpha}
  $\alpha(\mbbX, \mfraks)$ can also be characterized as an element of $H^2(B; \Z/2)$ satisfying
  \[
    \pi_{\mbbX}^\ast \alpha(\mbbX, \mfraks) = w_2(T_B \mbbX).
  \]
  Here, $T_B \mbbX$ is the tangent bundle along the fibers of $X \to \mbbX \stackrel{\pi_{\mbbX}}{\to} B$.
\end{remark}

We obtain the following equality about characteristic classes on $\BDiffplus$.

\begin{theorem}\label{thm1:alpha = w_2}
  Let $X$ be a homotopy $K3$ surface. Let $\mfraks$ be a spin structure on $X$. Let $X \to \Xunivplus \to \BDiffplus$ be the family of $X$ associated with the universal principal $\Diffplus$-bundle $\EDiffplus \to \BDiffplus$. We choose a continuous family of smooth Riemannian metrics on the fibers of $\Xunivplus$ with respect to the $C^\infty$ topology. Then we have
  \[
    \alpha(\Xunivplus, \mfraks) = w_2(\hplus{\Xunivplus}).
  \]
\end{theorem}

\begin{remark}
  Suppose $X$ is a genuine $K3$ surface. The argument in Section 5 of Baraglia--Konno\cite{Baraglia-Konno-Nielsen-K3} shows the existence of a smooth family $X \to \mbbX \to T^2$ of $X$ over $T^2$ satisfying
  \[
    w_2(\hplus{\mbbX}) \neq 0.
  \]
  Combined with \cref{ref-thm1:w_2=0}, this means that when $X$ is a genuine $K3$ surface,
  \[
    \alpha(\Xunivplus, \mfraks) \neq 0.
  \]
  The non-vanishing of $\alpha(\Xunivplus, \mfraks)$ shows
  \[
    \Diffspin \ncong \Diffplus \times \{\pm1\}
  \]
  as topological groups. This is in contrast to the result (introduced in Kronheimer--Mrowka\cite{Kronheimer--Mrowka-Dehn-twist-K3-2020}) that when $X$ is a homotopy $K3$ surface, there exists a homeomorphism
  \[
    \Diffspin \approx \Diffplus \times \{\pm1\}
  \]
  as topological spaces (forgetting group structures). For a general homotopy $K3$ surface, it is not known whether $\alpha(\Xunivplus, \mfraks)$ vanishes or not.
\end{remark}

The last main theorem asserts that when a lift of the structure group of a family $\mbbX$ of $X$ to $\Diffspin$ is given, a spin structure on $\hplus{\mbbX}$ is canonically constructed.

\begin{theorem}\label{thm1:spin on H^+}
  Let $B$ be a topological space that is locally simply-connected and homeomorphic to an open set of some paracompact Hausdorff space. Let $X \to \mbbX \to B$ be a family of $X$ associated with a principal $\Diffplus$-bundle $\mcalE \to B$. We assume that $\mbbX$ is given a continuous family of smooth Riemannian metrics on $\mbbX$ with respect to the $C^\infty$ topology.
  \begin{enumarabicp}
    \item Suppose that a lift $\tmcalE \to B$ of $\mcalE$ to a principal $\Diffspin$-bundle is given. Then, we can canonically construct a spin structure $\mfrakt$ on $\hplus{\mbbX} \to B$.
    \item The operation of constructing $\mfrakt$ from a lift $\tmcalE$ is functorial. (See \cref{rem1:functor}.)
    \item Suppose that a lift $\tmcalE \to B$ of the structure group of $\mcalE$ to $\Diffspin$ is given. Let $\mfrakt$ denote the spin structure on $\hplus{\mbbX}$ constructed in (1). Then, under the correspondence of morphisms in (2), the automorphism $+1$ of $\tmcalE$ corresponds to the automorphism $+1$ of $\mfrakt$, and the automorphism $-1$ of $\mcalE$ corresponds to the automorphism $-1$ of $\mfrakt$. (For the meaning of $\pm 1$, see \cref{rem1:pm 1}.)
  \end{enumarabicp}
\end{theorem}

\begin{remark}\label{rem1:functor}
  We explain the precise meaning of the functoriality in \cref{thm1:spin on H^+}(2).

  First, we define a category $\mcalC$ as follows. The objects of $\mcalC$ are tuples $(B, \mcalE, \mbbX, g, \tmcalE)$ where
  \begin{itemize}
    \item $X \to \mbbX \to B$ is a family of $X$ associated with a principal $\Diffplus$-bundle $\mcalE$ over a topological space $B$ which is locally simply-connected and homeomorphic to an open set of some paracompact Hausdorff space,
    \item $g$ is a continuous family of smooth Riemannian metrics on $\mbbX$ with respect to the $C^\infty$ topology, and
    \item $\tmcalE$ is a lift of the structure group of $\mcalE$ to $\Diffspin$.
  \end{itemize}
  A morphism of $\mcalC$ from $(B, \mcalE, \mbbX, g, \tmcalE)$ to $(B^\prime, \mcalE^\prime, \mbbX^\prime, g^\prime, \tmcalE^\prime)$ consists of
  \begin{itemize}
    \item a continuous map $\map{f}{B}{B^\prime}$,
    \item a $\Diffplus$-equivariant map $\map{\Phi}{\mcalE}{\mcalE^\prime}$ covering $f$ and satisfying $\Phi^\ast g^\prime = g$,
    \item a $\Diffspin$-equivariant lift $\map{\tilde \Phi}{\tmcalE}{\tmcalE^\prime}$ of $\Phi$.
  \end{itemize}

  Next, we define a category $\mcalD$ as follows. The objects of $\mcalD$ are triples $(B, E, \mfrakt)$ where
  \begin{itemize}
    \item $E \to B$ is a real vector bundle of rank 3 with an orientation and a metric over a topological space $B$, which is locally simply-connected and homeomorphic to an open set of some paracompact Hausdorff space, and
    \item $\mfrakt$ is a spin structure on $E$.
  \end{itemize}
  A morphism of $\mcalD$ from $(B, E, \mfrakt)$ to $(B^\prime, E^\prime, \mfrakt^\prime)$ consists of
  \begin{itemize}
    \item a continuous map $\map{f}{B}{B^\prime}$,
    \item an orientation and metric preserving bundle map $\map{\Psi}{E}{E^\prime}$ covering $f$,
    \item a $Spin(3)$-equivariant lift $\map{\tilde \Psi}{\mfrakt}{\mfrakt^\prime}$ of $\Psi$.
  \end{itemize}

  \cref{thm1:spin on H^+}(2) asserts that the correspondence of objects in \cref{thm1:spin on H^+}(1) defines a functor from $\mcalC$ to $\mcalD$ by appropriately defining the correspondence of morphisms.
\end{remark}

\begin{remark}\label{rem1:pm 1}
  The meaning of the automorphisms $\pm 1$ of $\tmcalE$ in \cref{thm1:spin on H^+}(3) is as follows. From the short exact sequence
  \[
    1 \to \{\pm 1\} \to \Diffspin \to \Diffplus \to 1,
  \]
  $\pm 1$ are determined as elements of $\Diffspin$. The right action of $\pm 1$ on $\tmcalE$ defines automorphisms of $\tmcalE$, which we also write as $\pm 1$.

  Similarly, from the short exact sequence
  \[
    1 \to \{\pm 1\} \to Spin(3) \to SO(3) \to 1,
  \]
  the morphisms $\pm 1$ of the category $\mcalD$ (See \cref{rem1:pm 1}) are determined.
\end{remark}

\begin{remark}
  In \cref{thm1-0:spin on TX+H^+}, \cref{thm1:alpha = w_2} and \cref{thm1:spin on H^+}, we assumed that $X$ is a homotopy $K3$ surface. In each case, this assumption can be weakened to the following: $X$ is a connected 4-dimensional spin closed manifold satisfying
  \[
    b_0(X) = 1,\ b_1(X) = 0,\ b_2(X) = 22,\ b^+(X) = 3,\ H^1(X; \Z/2) = 0.
  \]
\end{remark}

\section{Spin structures on rank 3 real vector bundles}\label{sec3:spin}

Let $B$ be a topological space, and let $E \to B$ be a rank 3 real vector bundle with an orientation and a metric. In this section, we provide one characterization of spin structures on $E$. The method lies in an unconventional construction of $Spin(3)$. It can be obtained as a certain automorphism group of some geometric data on $S^2$.

First, we define the geometric data on $S^2$. More generally, we formulate it as geometric data on the unit sphere $S(V)$ of a rank 3 vector space $V$ with an orientation and a metric. In this case, the resulting automorphism group will be isomorphic to $Spin(V)$.

\begin{definition}
  A triple $\mbfL = (V, L, \tiota)$ is said to be a triple for $Spin(3)$-torsor if it satisfies the following conditions:
  \begin{enumarabicp}\label{def3:torsor}
    \item $V$ is a rank 3 real vector space with an orientation and a metric.
    \item $L$ is a complex line bundle over the unit sphere $S(V)$ of $V$, and
    \[
      c_1(L) \in H^2(S(V); \Z)
    \]
    is a positive generator.
    \item Let $\map{\iota}{S(V)}{S(V)}$ be the antipodal map. $\map{\tiota}{L}{L}$ is a fiberwise anti-linear continuous map covering $\iota$ and satisfies $\tiota^2 = -1$.
  \end{enumarabicp}
\end{definition}

A similar concept can be formulated for vector bundles as well.

\begin{definition}\label{def3:triple}
  Let $E \to B$ be a rank 3 real vector bundle with an orientation and a metric. A triple
  \[
    \mbfL = (E, L, \tiota)
  \]
  is said to be a triple for spin structure on $E$ if it satisfies the following conditions:
  \begin{enumarabicp}
    \item $L$ is a complex line bundle over $S(E)$.
    \item Let $\map{\iota}{S(E)}{S(E)}$ be the antipodal map. $\map{\tiota}{L}{L}$ is a fiberwise anti-linear continuous map covering $\iota$ and satisfies $\tiota^2 = -1$.
    \item For each $b \in B$,
    \[
      \mbfL_b = (E_b, L_b, \tiota_b)
    \]
    is a triple for $Spin(3)$-torsor. Here, $L_b$ denotes the restriction of $L$ to $S(E)$, and $\tiota_b$ denotes the restriction of $\tiota$ to $L_b$.
  \end{enumarabicp}
\end{definition}

Given a spin structure on $E$, it is easy to see that a triple for spin structure on $E$ can be constructed from that spin structure. Below, we prove that, under suitable assumptions on $B$, a spin structure on $E$ can be constructed from a triple for spin structure on $E$ (\cref{prop3:spin from triple}).

First, we construct a standard model of a triple for $Spin(3)$-torsor.

\begin{definition}\label{def3:std model}
  A triple $(V_0, L_0, \tiota_0)$ for $Spin(3)$-torsor is defined in the following way and called the standard model of triples for $Spin(3)$-torsor:
  \begin{itemize}
    \item Set $V_0 = \R^3$. Then, by stereographic projection, $S^2$ and $\C P^1$ can be identified. The map corresponding to the antipodal map on $S^2$ from $\C P^2$ to itself is given by
          \[
            \iota\parenlr{\begin{bmatrix}
                z \\
                w
              \end{bmatrix}} = \begin{bmatrix}
              -\bar w \\
              \bar z
            \end{bmatrix}.
          \]
    \item Set $L_0 = (S^3 \times \C) / U(1)$, where $S^3$ is realized as the sphere in $\C^2$, and the $U(1)$-action on $S^3 \times \C$ is defined by
          \[
            e^{i\theta} \cdot \parenlr{\begin{pmatrix}
                z \\
                w
              \end{pmatrix}, \alpha} = \parenlr{\begin{pmatrix}
                e^{i\theta} z \\
                e^{i\theta} w
              \end{pmatrix}, e^{i\theta} \alpha}.
          \]
    \item $\tiota_0$ is defined by
          \[
            \tiota\parenlr{\bracketlr{\begin{pmatrix}
                  z \\ w
                \end{pmatrix}, \alpha}} = \bracketlr{\begin{pmatrix}
                -\bar w \\
                \bar z
              \end{pmatrix}, \bar\alpha}.
          \]
  \end{itemize}
\end{definition}

Next, we define the isomorphism of triples for $Spin(3)$-torsors. As we will see later, the naive automorphism group defined here is not a double cover of $SO(V)$. The double cover of $SO(V)$ is constructed as an appropriate quotient of the naive automorphism group.

\begin{definition}
  An isomorphism between two triples $\mbfL = (V, L, \tiota)$ and $\mbfL^\prime = (V^\prime, L^\prime, \tiota^\prime)$ for $Spin(3)$-torsors is a pair $(f, \tf)$ of an orientation and metric preserving linear isomorphism
  \[
    \map{f}{V}{V^\prime}
  \]
  and a fiberwise complex linear isomorphism
  \[
    \map{\tf}{L}{L^\prime}
  \]
  covering it such that $\tf \tiota = \tiota^\prime \tf$.

  The set of all isomorphisms from $\mbfL$ to $\mbfL^\prime$ is denoted by
  \[
    \Iso(\mbfL, \mbfL^\prime).
  \]
  We introduce the compact open topology on $\Iso(\mbfL, \mbfL^\prime)$. Also, we write
  \[
    \Aut(\mbfL) = \Iso(\mbfL, \mbfL).
  \]
\end{definition}

\begin{lemma}\label{lem3:fiber bdl}
  Let $\mbfL = (V, L, \tiota)$ be a triple for $Spin(3)$-torsor. Then, $\Aut(\mbfL)$ is a fiber bundle over $SO(V)$. Its fiber is homeomorphic to the space $C$ of all continuous maps
  \[
    \map{\varphi}{S(V)}{\C^\times}
  \]
  satisfying
  \[
    \varphi(x) = \overline{\varphi(-x)}.
  \]
  The space $C$ has exactly two connected components.
\end{lemma}

\begin{proof}
  By a simple calculation, we can show that the fiber of $\map{\pi_\mbfL}{\Aut(\mbfL)}{SO(V)}$ at $I \in SO(V)$ is $C$.

  Take any $f \in SO(V)$. If an open neighborhood $U$ of $f$ is contractible, there exists a continuous map
  \[
    \map{\Phi}{U}{\Aut(\mbfL)}
  \]
  satisfying $\pi_\mbfL \circ \Phi = \id_U$. Then, the map $U \times C \to \restr{\Aut(\mbfL)}{U}$ given by
  \[
    (f^\prime, \varphi) \to \varphi \cdot \Phi(f^\prime)
  \]
  is a homeomorphism. This gives a local trivialization.

  Finally, we prove that $C$ has exactly two connected components. Take any $\varphi \in C$.
  Since $S(V)$ is contractible, there exist two continuous functions
  \[
    \map{r}{S(V)}{\R^\times},\ \map{\theta}{S(V)}{\R}
  \]
  such that $\varphi = r e^{i\theta}$. From $\varphi(x) = \overline{\varphi(-x)}$, there exists some $n \in \Z$ such that
  \[
    \theta(x) + \theta(-x) = 2 \pi n.
  \]
  By deforming $r$ to the constant function $1$ and $\theta$ to $\pi n$, $\varphi$ can be deformed to either the constant function $1$ or $-1$.
\end{proof}

By \cref{lem3:fiber bdl}, it is expected that a double cover of $SO(V)$ can be obtained by collapsing each connected component of the fiber of $\Aut(\mbfL) \to SO(V)$ to a point. With this motivation, we make the following definition.

\begin{definition}
  Let $\mbfL$ and $\mbfL^\prime$ be triples for $Spin(3)$-torsor. We introduce an equivalence relation on $\Iso(\mbfL, \mbfL^\prime)$ as follows: $(f, \tf) \in \Iso(\mbfL, \mbfL^\prime)$ and $(f^\prime, \tilde f^\prime)$ are equivalent if $f = f^\prime$ and $\tilde f$ and $\tilde f^\prime$ are isotopic through elements of $\Iso(\mbfL, \mbfL^\prime)$ covering $f$. We denote the quotient of $\HIso(\mbfL, \mbfL^\prime)$ by this equivalence relation as
  \[
    \HIso(\mbfL, \mbfL^\prime).
  \]
  Also, we write
  \[
    \HAut(\mbfL) = \HIso(\mbfL, \mbfL).
  \]
\end{definition}

\begin{lemma}
  Let $\mbfL = (V, L, \tiota)$ be a triple for $Spin(3)$-torsor. Then, $\HAut(\mbfL)$ is a double covering group of $SO(V)$.
\end{lemma}

\begin{proof}
  This follows from \cref{lem3:fiber bdl}.
\end{proof}

Under the above preparations, we prove the announced result.

\begin{proposition}\label{prop3:iso Spin(3)}
  Let $\mbfL_0$ be the standard model of triples for spin structure on $E$. Then, there exists a unique isomorphism
  \[
    Spin(3) \cong \HAut(\mbfL_0)
  \]
  as double covering groups of $SO(3)$.
\end{proposition}

\begin{proof}
  It is enough to prove that there exists an isomorphism
  \[
    SU(2) \cong \HAut(\mbfL_0).
  \]
  For $B \in SU(2)$, define an automorphism $\varphi_B$ of $L_0$ by
  \[
    \varphi_B \parenlr{\bracketlr{\begin{pmatrix}
          z \\
          w
        \end{pmatrix}, \alpha}} = \bracketlr{B \begin{pmatrix}
        z \\
        w
      \end{pmatrix}, \alpha}.
  \]
  This action lifts the $SO(3)$-action on $S^2$. Here, the homomorphism from $SU(2)$ to $SO(3)$ is constructed via the adjoint representation $SU(2) \to SO(\Image \HB)$ and the identification
  \[
    \R^3 \cong \Image \HB;\ (a, b, c) \mapsto ci + bj - ak.
  \]
  It is easy to see that it gives a topological group homomorphism
  \[
    SU(2) \to \Aut(\mbfL_0).
  \]
  Furthermore, composing this with the quotient map $\Aut(\mbfL_0) \to \HAut(\mbfL_0)$, a topological group homomorphism
  \[
    SU(2) \to \HAut(\mbfL_0)
  \]
  is constructed. It is readily shown that this is an isomorphism.
\end{proof}

\begin{remark}
  In the discussion so far, we have assumed that $c_1(L)$ is a positive generator of $H^2(S(V); \Z)$, but similar considerations can be made even when $c_1(L)$ takes the other values. The following parts change:
  \begin{itemize}
    \item In \cref{def3:torsor}, change $\tiota^2 = -1$ to $\tiota^2 = (-1)^{c_1(L)}$.
    \item In \cref{def3:std model}, the $U(1)$-action on $S^3 \times \C$ is
          \[
            e^{i\theta} \cdot \parenlr{\begin{pmatrix}
                z \\
                w
              \end{pmatrix}, \alpha} = \parenlr{\begin{pmatrix}
                e^{i\theta} z \\
                e^{i\theta} w
              \end{pmatrix}, e^{i c_1(L) \theta} \alpha}.
          \]
    \item The conclusion of \cref{prop3:iso Spin(3)} changes as follows: $\HAut(\mbfL_0)$ is isomorphic to $Spin(3)$ when $c_1(L)$ is odd, and isomorphic to $SO(3) \times \{\pm 1\}$ when $c_1(L)$ is even.
  \end{itemize}
\end{remark}

Next, we prove that a spin structure can be constructed from a triple for spin structure when $B$ is a good space.

\begin{proposition}\label{prop3:spin from triple}
  Let $B$ be a topological space that is locally simply-connected and homeomorphic to an open set of some paracompact Hausdorff space. Let $E$ be a rank 3 real vector bundle with an orientation and a metric over $B$. Let $\mbfL$ be a triple for spin structure on $E$. Then,
  \[
    P = \coprod_{b \in B} \HIso(\mbfL_0, \mbfL_b)
  \]
  admits a canonical topology, and the pair
  \[
    \mfrakt = (P, \rho)
  \]
  is a spin structure on $E$. Here $\map{\rho}{P}{\Fr^{SO}(E)}$ is a natural projection.
\end{proposition}

\begin{proof}
  We prove the following claim.
  \begin{claim}\label{clm3:Tietze}
    For each $b \in B$, an open neighborhood $V$ of $b$ can be taken with the following property: there exists an isomorphism from $V \times L_0$ to $\restr{L}{V}$ as complex line bundles, which is compatible with $\tiota_0$ and $\tiota$.
  \end{claim}
  \begin{proof}
    Take an isomorphism from $\{b\} \times L_0$ to $L_b$ as complex line bundles. Since we are discussing local properties, we may assume $E = B \times \R^3$. As the product of a paracompact Hausdorff space and a compact Hausdorff space is paracompact Hausdorff, the sphere bundle $B \times S^2$ of $E$ is homeomorphic to an open subset of a paracompact Hausdorff space. Therefore, by Tietze's extension theorem, there exists an open neighborhood $U$ of $b$ such that the isomorphism can be extended to
    \[
      \map{F}{V \times L_0}{\restr{L}{V}}.
    \]
    Since $B$ is locally simply-connected, $V$ can be taken to be simply-connected. We want to multiply $F$ by an appropriate continuous function
    \[
      \map{\varphi}{S(\restr{E}{V})}{\C^\times}
    \]
    so that $\varphi F$ becomes compatible with $\tiota_0$ and $\tiota$. Let $\map{\psi}{S(\restr{E}{V})}{\C^\times}$ be a function satisfying
    \[
      F(-x)^{-1} \circ \tiota(x) \circ F(x) = \psi(x) \tiota_0(x)
    \]
    for any $x \in S(\restr{E}{V})$. The compatibility can be expressed as
    \[
      \overline{\varphi(x)} \psi(x) = \varphi(-x)
    \]
    for any $x \in S(\restr{E}{V})$. We show the existence of $\varphi$ satisfying this property.

    Since $S(\restr{E}{V})$ is simply-connected, we can choose a square root $\chi$ of $\psi$. From $\tiota_0^2 = -1$ and $\tiota^2 = -1$, for any $x \in S(\restr{E}{V})$, we have
    \[
      \psi(x) \overline{\psi(-x)} = 1,
    \]
    and so $\chi$ satisfies
    \[
      \chi(x) \overline{\chi(-x)} = \pm 1.
    \]
    Since $\chi$ is an even function and defined on the 2-sphere, it can be seen that only
    \[
      \chi(x) \overline{\chi(-x)} = 1
    \]
    is possible. It suffices to set $\varphi = \chi / \abs{\chi}^2$.
  \end{proof}
  Take one of the isomorphisms in \cref{clm3:Tietze}. Then, we have a bijection
  \[
    \restr{P}{V} \cong V \times \HAut(\mbfL_0) \cong V \times Spin(3).
  \]
  We introduce a topology on $\restr{P}{V}$ by this bijection. It is clear that $P$ with this topology gives a spin structure on $E$.
\end{proof}

Finally, we state a proposition that is repeatedly used in \cref{sec4:spin from fda} and \cref{sec5:swap fda}. In these sections, we construct a triple for spin structure for $\hplus{\mbbX}$ by taking some auxiliary data. Therefore, we need to show that the spin structure constructed from it does not depend on them. For most of the auxiliary data, the space of all choices of that data is contractible. In that case, the following proposition shows the independence.

\begin{proposition}\label{prop3:simp-conn}
  Let $C$ be a simply-connected topological space. Let $B$ be a topological space that is locally simply-connected and homeomorphic to an open set of some paracompact Hausdorff space, and let $E$ be a rank $3$ real vector bundle with an orientation and a metric over $B$. Let
  \[
    \mbfL = (C \times E, L, \tiota)
  \]
  be a triple for spin structure on $C \times E$. For each $c \in C$, let $\mfrakt_c$ be the spin structure on $E$ constructed from
  \[
    (E, L_c, \tiota_c)
  \]
  by the method described in \cref{prop3:spin from triple}. Then, for any $c, c^\prime \in C$, an isomorphism of spin structures
  \[
    \map{\varphi_{c^\prime c}}{\mfrakt_c}{\mfrakt_{c^\prime}}
  \]
  is canonically constructed. This isomorphism satisfies the following properties:
  \begin{itemize}
    \item $\varphi_{cc} = \id_{\mfrakt_c}$.
    \item For $c, c^\prime, c^\pprime \in C$, we have $\varphi_{c^\pprime c^\prime} \circ \varphi_{c^\prime c} = \varphi_{c^\pprime c}$.
  \end{itemize}
\end{proposition}
\begin{proof}
  Take a path $\map{\gamma}{[0, 1]}{C}$ connecting $c$ and $c^\prime$. By pulling back $\mbfL$ by $\gamma$, we obtain a triple for spin structure on $[0, 1] \times E$. By \cref{prop3:spin from triple} we obtain a spin structure on $[0, 1] \times E$. By the unique lifting property, an isomorphism
  \[
    \map{\varphi_{c^\prime c}}{\mfrakt_{c^\prime}}{\mfrakt_c}
  \]
  is obtained. Since $C$ is simply-connected, this isomorphism does not depend on the choice of $\gamma$. The equalities
  \[
    \varphi_{cc} = \id_{\mfrakt_c},\ \varphi_{c^\pprime c^\prime} \circ \varphi_{c^\prime c} = \varphi_{c^\pprime c}
  \]
  also holds since $C$ is simply-connected.
\end{proof}

\section{Construction of a spin structure from a finite-dimensional approximation}\label{sec4:spin from fda}
In this section, we will do the following:
\begin{itemize}
  \item Construct a finite-dimensional approximation of the Seiberg--Witten map for a family $X \to \mbbX \to B$ of 4-dimensional closed spin manifolds $X$ satisfying $b_1(X) = 0$.
  \item Let $X$ be a homotopy $K3$ surface. Given a finite-dimensional approximation, construct a spin structure on $\hplus{\mbbX}$.
\end{itemize}

The former is a review of Furuta\cite{Furuta-11/8-inequality2001}\cite{Furuta-preprint}. The finite-dimensional approximation can be formulated for any 4-dimensional closed spin (or spin$^c$) manifold. In this paper, we deal with its construction only in the case $b_1(X) = 0$. The latter is achieved by proving that a triple for a spin structure on $\hplus{\mbbX}$ can be constructed from a finite-dimensional approximation.

There are two things to note about finite-dimensional approximations. First, the choice of a finite-dimensional approximation is not unique. Second, a global finite-dimensional approximation over $B$ may not exist in the first place. (This is because we do not assume compactness of $B$.) What can be proved is that a finite-dimensional approximation of the Seiberg--Witten map restricted to a sufficiently small open subset of $B$ exists. Therefore, to prove that a spin structure is constructed canonically over $B$, we also need to show the following:
\begin{itemize}
  \item Given two different global finite-dimensional approximations over $B$, a canonical isomorphism is constructed between the two spin structures obtained from them.
  \item Even if a global finite-dimensional approximation does not exist over $B$, $B$ can be covered by open sets on which a finite-dimensional approximation can be constructed. A gluing map is canonically constructed between the spin structures constructed on those open sets, and thus, a global spin structure is constructed over $B$.
\end{itemize}
The former will be proved in \cref{sec5:swap fda}. The latter will be proved in \cref{sec6:prf main thm}.

\subsection{Finite-dimensional approximation of the Seiberg--Witten map}\label{ssec4:sw fda}

In this subsection, we review the construction of a finite-dimensional approximation map of the Seiberg--Witten map for a family of 4-dimensional closed spin manifolds $X$ satisfying $b_1(X) = 0$. (We do not have to assume $X$ is a homotopy $K3$ surface here.) In this paper, we consider the following perturbed Seiberg--Witten map: parametrize the usual Seiberg--Witten map by $\shplus{\mbbX}$, and perturb it using the tautological section of $\shplus{\mbbX} \times \hplus{\mbbX}$. The method of construction is as written in Furuta\cite{Furuta-11/8-inequality2001}\cite{Furuta-preprint}. However, since we need to consider the map itself rather than the homotopy class of the finite-dimensional approximation, we carefully recall its construction.

Let $B$ be a topological space, and let $\mcalE \to B$ be a principal $\Diffplus$-bundle. Let
\[
  X \to \mbbX \to B
\]
be the family of $X$ associated with $\mcalE$. We assume $\mbbX$ is equipped with a family of Riemannian metrics that are smooth with respect to the $C^\infty$ topology. We also assume a lift
\[
  \tmcalE \to B
\]
of $\mcalE$ to a principal $\Diffspin$-bundle is given.

First, we explain the family of the Seiberg--Witten equations (with gauge fixing). Let
\begin{align*}
  \mscrV & = \coprod_{b \in B} \Omega^1(\mbbX_b; i\R) \times \Gamma(\mbbX_b; S^+_b),                               \\
  \mscrW & = \coprod_{b \in B} \Omega^+(\mbbX_b; i\R) \times \Omega^0(\mbbX_b; i\R) \times \Gamma(\mbbX_b; S^-_b).
\end{align*}
Here $S^\pm_b$ are the spinor bundles. We view $\mscrV$ and $\mscrW$ as vector bundles over $B$. Let
\[
  \map{D = (d^+ + d^\ast, D^+)}{\mscrV}{\mscrW}.
\]
Here, $\map{D^+}{\Gamma(S^+_b)}{\Gamma(S^-_b)}$ is the Dirac operator determined from the canonical reference connection. Also, define
\[
  \map{Q}{\mscrV}{\mscrW}
\]
by
\[
  Q(a, \phi) = (-q(\phi), 0, a \cdot \phi).
\]
The map $\map{q}{\Gamma(S^+_b)}{\Omega^1(\mbbX_b; i\R)}$ is defined by
\[
  q(\phi) \psi = \abracket{\psi, \phi} \phi - \frac{1}{2} \abs{\phi}^2 \psi.
\]
The term $a \cdot \phi$ is the Clifford multiplication of $\phi$ by $a$. The families Seiberg--Witten map is defined by $D + Q$.

The equation $D + Q = 0$ has reducible solutions, i.e., solutions satisfying $\phi = 0$. On such solutions, $U(1)$ acts trivially. To avoid reducible solutions, we introduce a perturbation parametrized by $\shplus{\mbbX}$. First, consider the pullback vector bundles
\[
  \shplus{\mbbX} \times_B \mscrV,\ \shplus{\mbbX} \times_B \mscrW
\]
of $\mscrV$, $\mscrW$ to $\shplus{\mbbX}$. The operators $D$ and $Q$ lift to the maps between the above bundles. We denote them by $D$, $Q$ as well. Also, let
\[
  \map{\Delta}{\shplus{\mbbX} \times_B \mscrV}{\shplus{\mbbX} \times_B \mscrW}
\]
be the pullback of the tautological section of $\shplus{\mbbX} \times_B \hplus{\mbbX} \to \shplus{\mbbX}$. (Note that $\hplus{\mbbX}$ is a subbundle of $\mscrW$.) We define the perturbed Seiberg--Witten map as follows.

\begin{definition}\label{def4:ptbd SW map}
  In the above situation, let
  \[
    \map{F = D + Q - \Delta}{\shplus{\mbbX} \times_B \mscrV}{\shplus{\mbbX} \times_B \mscrW}.
  \]
\end{definition}

The goal of this subsection is to describe a finite-dimensional approximation of $F$. As already mentioned, a finite-dimensional approximation may not exist globally over $B$. Therefore, we formulate auxiliary data (including an open subset of $B$) sufficient to construct it. After that, we verify that $B$ can be covered by such data.

To formulate the auxiliary data, we define the necessary concepts. We define a family of norms on the Hilbert bundles $\mscrV$, $\mscrW$ by
\begin{align}\label{eq4:norm of Hilb}
  \norm{v}_\mscrV^2 = \norm{(D^\ast D)^2 v}_{L^2}^2 + \norm{v}_{L^2}^2,\ \norm{w}_\mscrW^2 = \norm{(D D^\ast)^{3/2} w}_{L^2}^2 + \norm{w}_{L^2}^2.
\end{align}

For $\lambda > 0$, let $\mscrV^\lambda_b$ be the subspace of $\mscrV_b$ spanned by the eigenvectors corresponding to eigenvalues of $D^\ast D$ less than or equal to $\lambda$, and set
\[
  \mscrV^\lambda = \coprod_{b \in B} \mscrV^\lambda_b.
\]
A set $\mscrW^\lambda$ is defined similarly using $D D^\ast$. Both $\mscrV^\lambda_b$ and $\mscrW^\lambda_b$ are $Pin(2)$-equivariant finite-dimensional vector spaces. If $\lambda$ is not an eigenvalue of $D^\ast D$ at any point of a subset $U$ of $B$, these are $Pin(2)$-equivariant vector bundles over $U$.

\begin{definition}\label{def4:aux data}
  Let $\mcalA$ be the set of data $(R, U, \varepsilon, \lambda)$ satisfying the following properties:
  \begin{itemize}
    \item $R$, $\varepsilon$, $\lambda$ are positive real numbers.
    \item $U$ is an open subset of $B$.
    \item If $\norm{\tilde v}_{\shplus{\mbbX} \times_B \mscrV} \geq R$, then $F(\tilde v) \neq 0$.
    \item If $R \leq \norm{\tilde v}_{\restr{(\shplus{\mbbX} \times_B \mscrV)}{U}} \leq \sqrt{2} R$, then
          \[
            \norm{F(\tilde v)}_{\shplus{\mbbX} \times_B \mscrW} \geq \varepsilon.
          \]
    \item The number $\lambda$ is not an eigenvalue of $D^\ast D$ over $U$.
  \end{itemize}
\end{definition}

In Furuta\cite{Furuta-11/8-inequality2001}\cite{Furuta-preprint}, a finite-dimensional approximation is constructed using an element $(R, U, \varepsilon, \lambda)$ of $\mcalA$ and the orthogonal projection from $\mscrW$ to $\mscrW^\lambda$. (The condition that $\lambda$ is sufficiently large is necessary.) In this paper, we introduce a concept that abstracts only the necessary properties of the orthogonal projection.

\begin{definition}\label{def4:proj}
  Let $(R, U, \varepsilon, \lambda) \in \mcalA$. We define $\mcalP(R, U, \varepsilon, \lambda)$ as the space of continuous families
  \[
    \map{p}{\restr{(\shplus{\mbbX} \times_B \mscrW)}{U}}{\restr{(\shplus{\mbbX} \times_B \mscrW)}{U}}
  \]
  of bounded operators with respect to the norm of $\mscrW$ satisfying the following conditions:
  \begin{itemize}
    \item The map $p$ commutes with the $Pin(2)$-action on $\restr{(\shplus{\mbbX} \times_B \mscrW)}{U}$.
    \item The image of $p$ is contained in $\restr{(\shplus{\mbbX} \times_B \mscrW^\lambda)}{U}$.
    \item $\sup_{b \in U} \norm{p_b} < \infty$.
    \item If $R \leq \norm{\tilde v}_{\restr{(\shplus{\mbbX} \times_B \mscrV)}{U}} \leq \sqrt{2} R$, then
          \[
            \norm{(1 - p) Q (\tilde v)}_{\shplus{\mbbX} \times_B \mscrW} < \varepsilon
          \]
          holds.
  \end{itemize}
  The topology on $\mcalP(R, U, \varepsilon, \lambda)$ is introduced by the $\sup$-norm over $U$ of the operator norms.

  Also, let $\tilde \mcalA$ be the set of elements $(R, U, \varepsilon, \lambda)$ of $\mcalA$ such that $\mcalP(R, U, \varepsilon, \lambda)$ is nonempty.
\end{definition}

Before constructing the finite-dimensional approximation, we examine the properties of the set of auxiliary data. The following lemma shows that the set of all open subsets $U$ of $B$ where a finite-dimensional approximation is possible covers $B$. This is necessary to construct a spin structure on $\hplus{\mbbX}$ over the whole of $B$ by gluing. The proof is similar to Furuta\cite[Lemma 3.2, Lemma 3.3]{Furuta-11/8-inequality2001}.

\begin{lemma}\label{lem4:cov of B}
  \begin{enumarabicp}
    \item Fix $b \in B$. There exists $R > 0$ such that if $R \leq \norm{\tilde v}_{\shplus{\mbbX} \times_B \mscrV}$ we have $F(\tilde v) \neq 0$.
    \item Fix $b \in B$ and $R > 0$ satisfying the property in (1). Take any $R^\prime \geq R$. Then there exist an open neighborhood $U$ of $b$ and positive real numbers $\varepsilon$, $\lambda$ such that
    \[
      (R^\pprime, U, \varepsilon, \lambda) \in \tilde \mcalA
    \]
    holds for any $R^\pprime \in [R, R^\prime]$. In particular, the set
    \[
      \set{U}{\text{There exist $R$, $\varepsilon$, $\lambda$ such that $(R, U, \varepsilon, \lambda) \in \tilde \mcalA$}}
    \]
    is an open cover of $B$.
    \item The set $\mcalP(R, U, \varepsilon, \lambda)$ is contractible for any $(R, U, \varepsilon, \lambda) \in \tilde \mcalA$.
  \end{enumarabicp}
\end{lemma}

\begin{proof}
  The statement of (1) is a consequence of the compactness of the moduli space of the Seiberg--Witten equations. (2) follows from the same argument as in Furuta\cite[Lemma 3.2, Lemma3.3]{Furuta-11/8-inequality2001}. The contractibility in (3) is proved by a linear homotopy.
\end{proof}

We construct the finite-dimensional approximation.

First we explain the notation. The spaces $\mscrV$ and $\mscrW$ can be divided into the parts $\mscrV_\HB$, $\mscrW_\HB$ on which $U(1)$ acts with weight 1, and the parts $\mscrV_\R$, $\mscrW_\R$ on which the $U(1)$-action is trivial. We denote the $\mscrV_\HB$ component of $v \in \mscrV$ by $v_\HB$ and the $\mscrV_\R$ component by $v_\R$. We use the same notation for $\shplus{\mbbX} \times_B \mscrV$.

Take an $R > 0$. Let $B^\prime_R(\mscrV^\lambda)$ be the topological disk bundle over $U$ consisting of all $v \in \mscrV^\lambda$ satisfying $\norm{v_\HB} \leq R$ and $\norm{v_\R} \leq R$. Also, let $\partial B^\prime_R(\mscrV^\lambda)$ be the sphere bundle over $U$ consisting of all points $v \in B^\prime_R(\mscrV^\lambda)$ satisfying $\norm{v_\HB} = R$ or $\norm{v_\R} = R$.

\begin{definition}\label{def4:fda}
  Take $(R, U, \varepsilon, \lambda) \in \tilde \mcalA$. Choose $p \in \mcalP(R, U, \varepsilon, \lambda)$. Then, define
  \[
    \map{D^\lambda}{\restr{(\shplus{\mbbX} \times_B \mscrV^\lambda)}{U}}{\restr{(\shplus{\mbbX} \times_B \mscrW^\lambda)}{U}}
  \]
  as the restriction of $D$ to $\restr{(\shplus{\mbbX} \times_B \mscrV^\lambda)}{U}$, and define
  \[
    \map{Q^\lambda}{\restr{(\shplus{\mbbX} \times_B \mscrV^\lambda)}{U}}{\restr{(\shplus{\mbbX} \times_B \mscrW^\lambda)}{U}}
  \]
  by $Q^\lambda = p \circ Q$. Finally, define
  \[
    \map{F^\lambda = D^\lambda + Q^\lambda - \Delta}{\restr{(\shplus{\mbbX} \times_B B^\prime_R(\mscrV^\lambda))}{U}}{\restr{(\shplus{\mbbX} \times_B \mscrW^\lambda)}{U}}.
  \]
  We call $F^\lambda$ a finite-dimensional approximation of $F$.
\end{definition}

We end this subsection by stating a proposition about the relationship between $\mscrV^\lambda_\R$ and $\mscrW^\lambda_\R$. The proof is straightforward.

\begin{proposition}\label{prop4:split of vb}
  Assume that $\lambda$ is not an eigenvalue of $D^\ast D$ over an open subset $U$ of $B$. Then there is a natural isomorphism of finite-dimensional $Pin(2)$-equivariant vector bundles over $U$
  \[
    \hplus{\restr{\mbbX}{U}} \oplus \restr{\mscrV^\lambda_\R}{U} \cong \restr{\mscrW^\lambda_\R}{U}.
  \]
  Here $\hplus{\mbbX} \to \mscrW^\lambda_\R$ is the inclusion, and $\mscrV^\lambda_\R \to \mscrW^\lambda_\R$ is defined by $D^\lambda$.
\end{proposition}

\subsection{Abstraction of the properties of finite-dimensional approximations}\label{ssec4:prpty fda}

In this subsection, we formulate a concept that abstracts the properties of the finite-dimensional approximations. First, we perform an abstraction with families of 4-dimensional closed spin manifolds $X$ satisfying $b_1(X) = 0$ in mind. Next, we formulate a concept with additional conditions satisfied in the case where $X$ is a homotopy $K3$ surface.

\begin{definition}\label{def4:fda model}
  A tuple
  \[
    \mscrF = (U, E, V, W, i, D, F)
  \]
  is said to be a model of FDA for families of spin closed 4-manifolds with $b_1 = 0$ if it satisfies the following properties.
  \begin{enumarabicp}
    \item $U$ is a topological space.
    \item $E \to U$ is a real vector bundle equipped with a metric.
    \item $V$, $W$ are finite-dimensional vector bundles over $U$ given in the form
    \[
      V = V_\HB \oplus V_\R,\ W = W_\HB \oplus W_\R.
    \]
    Here, $V_\HB$, $W_\HB$ are quaternionic vector bundles, and $V_\R$, $W_\R$ are real vector bundles. We consider the quaternionic action as right multiplication. We define the right $Pin(2)$-action on $V$ and $W$ as follows: on the quaternionic part, we act via the natural inclusion $Pin(2) \to \HB$, and on the real part, we act via $Pin(2) \to \{\pm 1\}$.
    \item $V$ and $W$ are equipped with $Pin(2)$-invariant metrics.
    \item $\map{i}{E}{W_\R}$ is an injective $Pin(2)$-equivariant fiberwise linear bundle map. Here, we define the $Pin(2)$-action on $E$ via $Pin(2) \to \{\pm 1\}$.
    \item $\map{D}{V}{W}$ is the sum of a $Pin(2)$-equivariant fiberwise linear bundle map
    \[
      \map{D_\R}{V_\R}{W_\R}
    \]
    and a $Pin(2)$-equivariant fiberwise linear bundle map
    \[
      \map{D_\HB}{V_\HB}{W_\HB}.
    \]
    \item The map
    \[
      \map{i + D_\R}{E \oplus V_\R}{W_\R}
    \]
    is a metric-preserving isomorphism.
    \item $\map{F}{S(E) \times_U B^\prime(V)}{S(E) \times_U W}$ is a fiberwise smooth map preserving the fibers of $S(E)$. (It varies continuously with respect to the $C^\infty$ topology on the fibers.) Here, $B^\prime(V)$ is the topological disk bundle consisting of all $\tilde v = (v_\HB, v_\R) \in V$ satisfying $\norm{v_\HB} \leq 1,\ \norm{v_\R} \leq 1$.
    \item $F$ does not vanish for elements of the boundary
    \[
      S(E) \times_U \partial B^\prime(V)
    \]
    and for elements of $S(E) \times_U B^\prime(V)$ whose $V_\HB$ component is 0.
  \end{enumarabicp}
\end{definition}

\begin{example}\label{eg4:fda from sw}
  Using the notation of \cref{def4:fda}, a model of FDA for families of spin closed 4-manifolds with $b_1 = 0$ is defined from the finite-dimensional approximation of the Seiberg--Witten map for a family in the following way:
  \begin{align*}
    U = U,\ E = \restr{\hplus{\mbbX}}{U},\ V = \restr{\mscrV}{U},\ W = \restr{\mscrW}{U}, \\
    i \colon \shplus{\restr{\mbbX}{U}} \xhookrightarrow{} \restr{\mscrW^\lambda_\R}{U},\ D = D^\lambda \circ m_R,\ F = F^\lambda \circ m_R.
  \end{align*}
  Here,
  \[
    \map{m_R}{\mscrV}{\mscrV}
  \]
  denotes the $R$-times map. Note that $D^\lambda \circ m_R$ is not necessarily a metric-preserving map between $\restr{\mscrV^\lambda_\R}{U}$ and $\restr{\mscrW^\lambda_\R}{U}$ defined from \cref{eq4:norm of Hilb} in general. Therefore, when we view these as a model of FDA, we keep the metric on $\restr{\mscrV^\lambda_\R}{U}$ unchanged and redefine the metric on $\restr{\mscrW^\lambda_\R}{U}$ so that
  \[
    \map{i + D^\lambda \circ m_R}{\restr{\hplus{\mbbX}}{U} \oplus \restr{\mscrV^\lambda_\R}{U}}{\restr{\mscrW^\lambda_\R}{U}}
  \]
  becomes an isomorphism preserving the metric.
\end{example}

\begin{definition}\label{def4:iso of fda}
  Let $\mscrF = (U, E, V, W, i, D, F)$ be a model of FDA for families of spin closed 4-manifolds with $b_1 = 0$. An automorphism of $\mscrF$ is a pair of $Pin(2)$-equivariant metric-preserving isomorphisms of vector bundles
  \[
    \map{f}{V}{V},\ \map{g}{W}{W}
  \]
  satisfying
  \[
    g \circ i = i,\ D \circ f = g \circ D,\ F \circ f = g \circ F.
  \]
\end{definition}

\begin{example}\label{eg4:pm 1 for fda}
  Let $\mscrF = (U, E, V, W, i, D, F)$ be a model of FDA for families of spin closed 4-manifolds with $b_1 = 0$. Then, acting $\{\pm 1\}$ on $V$ and $W$ gives an automorphism of $\mscrF$.
\end{example}

When $X$ is a homotopy $K3$ surface, the finite-dimensional approximation satisfies three additional properties. We explain those properties.
Let
\[
  \mscrF = (U, E, V, W, i, D, F)
\]
be a model of FDA for families of closed spin 4-manifolds with $b_1 = 0$. The first additional condition is that the rank of $E$ is 3. The second is that, by what was stated in \cref{ref-thm1:MS}, a canonical orientation can be given to $E$. The third condition is that, for each $b \in B$, the so-called families Seiberg--Witten invariant
\[
  c(\mscrF_b) \in H^2(S(E_b); \Z)
\]
is a positive generator. Here, we review the construction of $c(\mscrF_b)$.  Let
\[
  \tmcalV = S(E) \times S(V_\HB) \times \R^+ \times V_\R, \ \tmcalW = S(E) \times S(V_\HB) \times W_\HB \times W_\R
\]
where $\R^+$ denotes $\R$ with the trivial $Pin(2)$-action. These are vector bundles over $S(E) \times S(V_\HB)$, and $Pin(2)$ acts freely on them. Also, let
\[
  B^\prime(\tmcalV) = S(E) \times S(V_\HB) \times B(\R^+) \times B(V_\R).
\]
Define $\map{\tmcalF^\prime}{B^\prime(\tmcalV)}{\tmcalW}$ by
\[
  \tmcalF^\prime(e, v_\HB, t, v_\R) = F\parenlr{e, v_\HB, \frac{t + 1}{2}v_\R}.
\]
(To be precise, since $S(V_\HB)$ is not included in the codomain of $F$, the above notation does not give a map to $\tmcalW$. We simply concatenate the $S(V_\HB)$ component from the domain side for the $S(V_\HB)$ component.) This is a smooth $Pin(2)$-equivariant map.

We quotient this map by the $U(1)$-action. Let
\begin{align*}
  \mcalV & = \tmcalV / U(1) = S(E) \times S(V_\HB)/U(1) \times \R^+ \times V_\R,    \\
  \mcalW & = \tmcalW / U(1) = S(E) \times (S(V_\HB) \times W_\HB)/U(1) \times W_\R.
\end{align*}
Also, let
\[
  B^\prime(\mcalV) = S(E) \times S(V_\HB)/U(1) \times B(\R^+) \times B(V_\R)
\]
These spaces are equipped with a $\Z/2$-action defined from the action of $j \in Pin(2)$. Quotienting $\tmcalF$ by the $U(1)$-action defines a smooth $\Z/2$-equivariant map
\begin{align}\label{eq4:deformed F}
  \map{\mcalF^\prime}{B^\prime(\mcalV)}{\mcalW}.
\end{align}
Denote the restriction of this map to the fiber over $b \in B$ by
\[
  \map{\mcalF_b^\prime}{B^\prime(\mcalV_b)}{\mcalW_b}
\]
To  construct $c(\mscrF_b)$, we fix an auxiliary orientation of $V_{\R, b}$. Using the isomorphism
\[
  \map{i + D}{E_b \oplus V_{\R. b}}{W_{\R, b}}
\]
and the orientation of $E_b$, an orientation is also induced on $W_{\R, b}$. This gives an orientation on $\mcalV_b$ as a smooth manifold and an orientation on $\mcalW_b$ as a vector bundle over $S(E_b) \times S(V_{\HB, b}) / U(1)$. Let
\[
  x = c_1(\mcalO(1)) \in H^2(\mcalV_b; \Z).
\]
Here, $\mcalO(1)$ denotes the pullback to $\mcalV_b$ of the canonical line bundle over $S(V_\HB)/U(1)$. Also, let
\[
  \tau_{\mcalW_b} \in H^{\rank \mscrW}(\mcalW_b, \mcalW_b \setminus (S(E_b) \times S(V_{\HB, b})) / U(1); \Z)
\]
be the Thom class of the vector bundle $\mcalW_b \to (S(E_b) \times S(V_{\HB, b}) / U(1))$. Under these preparations, let
\begin{align}\label{eq4:fam sw inv}
  c(\mscrF_b) = \int_{B^\prime(\mscrV_b) \to S(E_b)} x \cdot \mcalF_b^{\prime \ast} \tau_{\mcalW_b} \in H^2(S(E_b); \Z).
\end{align}
The cohomology class $c(\mscrF_b)$ is independent of the choice of the orientation on $V_{\R, b}$.

\begin{definition}\label{def4:fda model hK3}
  A model of FDA for families of h$K3$ is a model
  \[
    \mscrF = (U, E, V, W, i, D, F)
  \]
  of FDA for families of spin closed 4-manifolds with $b_1 = 0$ satisfying the following properties:
  \begin{enumarabicp}
    \item The rank of $E$ is 3.
    \item An orientation is given on $E$.
    \item For any $b \in U$, $c(\mscrF_b)$ is a positive generator.
  \end{enumarabicp}
\end{definition}

\begin{proposition}
  Let $\mscrF = (U, E, V, W, i, D, F)$ be a model of FDA for families of spin closed 4-manifolds with $b_1 = 0$ constructed by the method of \cref{eg4:fda from sw} for $X$ a homotopy $K3$ surface. Then $\mscrF$ is a model of FDA for families of h$K3$.
\end{proposition}

\begin{proof}
  \cref{def4:fda model hK3}(2) is a consequence of \cref{ref-thm1:MS}. \cref{def4:fda model hK3}(3) is a consequence of Baraglia--Konno\cite[Theorem 1.1, Theorem 1.8]{Baraglia--Konno-2022}. (The calculation excluding the sign is due to Li--Liu\cite[Theorem 4.10]{Li--Liu-family-wall-crossing2001}.)
\end{proof}

The rest of this section is devoted to the proof of the following theorem.

\begin{theorem}\label{thm4:spin from fda}
  Let $\mscrF = (U, E, V, W, D, i, F)$ be a model of FDA for families of h$K3$. Further, assume that $U$ is locally simply-connected and homeomorphic to an open set of some paracompact Hausdorff space. Then the following statements hold:
  \begin{enumarabicp}
    \item A spin structure $\mfrakt_\mscrF$ on $E$ is canonically constructed from $\mscrF$.
    \item The self-isomorphisms $\pm 1$ of $\mscrF$ (See \cref{eg4:pm 1 for fda}.) induce self-isomorphisms of $\mfrakt_\mscrF$ covering the identity on $E$. Under this correspondence, the self-isomorphism $+1$ of $\mscrF$ goes to the self-isomorphism $+1$ of $\mfrakt_\mscrF$, and $-1$ goes to $-1$. (For the meaning of $\pm 1$ as isomorphisms of a spin structure, see \cref{rem1:pm 1}.)
  \end{enumarabicp}
\end{theorem}

\subsection{Construction of a spin structure from a model of FDA}\label{ssec4:spin from fda}

In this subsection, we fix a model
\[
  \mscrF = (U, E, V, W, D, i, F)
\]
of FDA for families of h$K3$.

The goal of this subsection is to construct a spin structure on $E$ from $\mscrF$ when $U$ is a point. The case where $U$ is not a point needs a slight modification, and it is described in \cref{ssec4:base not pt}. The strategy is to construct a triple for spin structure on $E$ from $\mscrF$.

When $U$ is a point, $E$ is a real vector space of rank 3 with an orientation and a metric. Also, $V$, $W$ are direct sums of real vector spaces and quaternionic vector spaces. Therefore,
\[
  \shplus{\mbbX} \times_U V = S(E) \times V_\HB \times V_\R,\ \shplus{\mbbX} \times_U W = S(E) \times W_\HB \times W_\R.
\]

\subsubsection{Modification of $F$}\label{sssec4:deform F}

We use the deformation of $F$ in \cref{ssec4:prpty fda} to construct $c(\mscrF_b)$ here as well. Let
\[
  \map{\mcalF^\prime}{B^\prime(\mcalV)}{\mcalW}
\]
be the map in \cref{eq4:deformed F}. Later, we consider a family of Dirac operators on $\mcalV$, and for that purpose, it is desirable for the domain of $\mcalF^\prime$ to be $\mcalV$ rather than $B^\prime(\mcalV)$. By appropriately modifying $\mcalF^\prime$, we make the domain to be the whole of $\mcalV$. Take a smooth function
\begin{align}\label{eq4:rho}
  \map{\rho}{[0, \infty)}{[0,1]}
\end{align}
satisfying
\begin{itemize}
  \item $\rho(t) = t$ in some neighborhood of $t = 0$,
  \item $\rho(t) = 1$ for $t \geq 1$.
\end{itemize}
By adjusting the length of vectors using $\rho$, smooth maps
\[
  \R^+ \to B(\R^+),\ V_\R \to B(V_\R)
\]
are constructed, and we get a smooth map from $\mcalV$ to $B^\prime(\mcalV)$. Composing this with $B^\prime(\mcalV)$, we obtain
\begin{align}\label{eq4:normalized F}.
  \map{\mcalF}{\mcalV}{\mcalW}
\end{align}

We prepare the notation used in the following sections. First, let
\begin{align*}
  \mcalW_\HB & = S(E) \times (S(V_\HB) \times W_\HB)/U(1), \\
  \mcalW_\R  & = S(E) \times S(V_\HB)/U(1) \times W_\R.
\end{align*}
$\mcalW_\HB$ and $\mcalW_\R$ have $\Z/2$-invariant metrics as vector bundles over $S(E) \times S(V_\HB)/U(1)$.

$\R^+ \times V_\R$ has a Riemannian metric as a manifold from the metric as a vector space. Also, we fix a $\Z/2$-invariant Riemannian metric on $S(V_\HB)/U(1)$. These data induce a $\Z/2$-invariant Riemannian metric on the fiberwise tangent bundle $T_{S(E)} \mcalV$ of $\mcalV$.

\subsubsection{Construction of a family of Fredholm operators}\label{sssec4:Fred op}
To construct a triple for spin structure on $E$, we need to construct a complex line bundle over $S(E)$, among others. The idea is to construct it as the determinant line bundle of a certain family of Fredholm operators over $S(E)$. The family of Fredholm operators is obtained as the family of Dirac operators associated with the family of Clifford bundles along the fibers of $\mcalV \to S(E)$. The construction of the family of Clifford bundles is basically what appears in the construction of the $K$-theoretic mapping degree of $\mcalF$ in \cref{eq4:normalized F}.

Here, we note three things. First, since the ultimate goal of this section is the canonical construction of a spin structure on $E$, it is important to confirm that a canonical construction can be chosen for the family of Clifford bundles. Next, since the fibers of $\mcalV \to S(E)$ are non-compact, we need to construct, in addition to the Clifford bundle, a family $h$ of degree 0 Hermitian maps with compact support. Finally, as a reflection of the fact that a $\Z/2$-action was present on $\mcalF$, we will confirm that a degree-preserving action (which is actually a $\Z/4$-action) is also present on the family of Clifford bundles.

We begin the construction of the family of Clifford bundles on $\mcalV \to S(E)$. The strategy is to obtain it as the external tensor product of the following three Clifford bundles:
\begin{itemize}
  \item The family $(S_\HB, c_\HB, \nabla_\HB)$ of Clifford bundles along the fibers of
        \[
          S(E) \times S(V_\HB)/U(1) \to S(E).
        \]
  \item The family $(S_\R, c_\R, \nabla_\R)$ of Clifford bundles along the fibers of
        \[
          S(E) \times (\R^+ \times V_\R) \to S(E).
        \]
  \item The family $(S_E, c_E, \nabla_E)$ of Clifford bundles determined from a $\Z/2$-graded vector bundle $S_E$ with a metric over $S(E)$.
\end{itemize}
The first two families of Clifford bundles are constructed as the pullback of Clifford bundles on $S(V_\HB)/U(1)$ and $\R^+ \times V_\R$, respectively. For the third one, we use the fact that a vector bundle over $S(E)$ can be regarded as a family of Clifford bundles along the fiber bundle $S(E) \to S(E)$ whose fibers are points. (There is a unique Clifford bundle structure on a $\Z/2$-graded vector space with a metric over a point.)

We first describe the construction of the family $(S_\HB, c_\HB, \nabla_\HB)$ of Clifford bundles over $S(E) \times S(V_\HB)/U(1)$. As announced, this is constructed as the pullback of a Clifford bundle on $S(V_\HB)/U(1)$. To save notation, we also denote it by $(S_\HB, c_\HB, \nabla_\HB)$. Specifically, it is defined as follows. First, let
\[
  \mcalW^\prime_\HB = (S(V_\HB) \times W_\HB) / U(1)
\]
and let $T_{V_\HB}$ be the tangent bundle of $S(V_\HB)/U(1)$. As a vector bundle over $S(V_\HB)/U(1)$, let
\[
  S_\HB = (\underline{\C} \oplus \mcalO(1)) \otimes \Lambda^\ast_\C \mcalW^\prime_\HB \otimes \Lambda^\ast_\C T_{V_\HB}.
\]
Here, the tensor product is taken as $\Z/2$-graded vector bundles.
The degree of
\[
  \underline{\C} \oplus \mcalO(1)
\]
is 0 for $\underline{\C}$ and 1 for $\mcalO(1)$. The Clifford multiplication $c_\HB(v_\HB)$ by $v_\HB \in T_{V_\HB}$ is defined by
\[
  \varepsilon \otimes \varepsilon \otimes (v_\HB^\wedge - v_\HB^\lrcorner)
\]
Here $\varepsilon$ is the involution representing the $\Z/2$-grading of each tensor product component. The connection $\nabla_\HB$ on $S_\HB$ can be anything as long as it has certain properties. (The canonical connection is one such connection.) The formulation of the needed property involves the $\Z/4$-action $\tau_\HB$ on $S_\HB$, so it is deferred to \cref{def4:connection on S_HB}.

Next, we describe the construction of the family $(S_\R, c_\R, \nabla_\R)$ of Clifford bundles over $S(E) \times (\R^+ \times V_\R)$. As before, this is also constructed as the pullback of a Clifford bundle on $\R^+ \times V_\R$. We also use the notation $(S_\R, c_\R, \nabla_\R)$ for this Clifford bundle. First, as a vector bundle over $\R^+ \times V_\R$, let
\begin{align}\label{eq4:S_R}
  S_\R = (\R^+ \times V_\R) \times (\Lambda^\ast_\R \R^+ \otimes_\R \C) \otimes (\Lambda^\ast_\R V_\R \otimes_\R \C).
\end{align}
The Clifford multiplication $c_\R(v_+, v_\R)$ by $(v_+, v_\R) \in T(\R^+ \times V_\R)$ is defined by
\[
  (v_+^\wedge - v_+^\lrcorner) \otimes 1 + \varepsilon \otimes (v_\R^\wedge - v_\R^\lrcorner).
\]
We take the Levi--Civita connection as the connection $\nabla_\R$ on $S_\R$.

Finally, we construct the Clifford bundle $(S_E, c_E, \nabla_E)$ along the family of fiber bundles $S(E) \to S(E)$ whose fibers are points. Since the fibers are points, by constructing a $\Z/2$-graded vector bundle $S_E$ over $S(E)$, $c_E$ and $\nabla_E$ are uniquely determined. So we only describe the definition of $S_E$. As a vector bundle over $S(E)$, let
\begin{align}\label{eq4:S_E}
  S_E = \underline{\C} \oplus TS(E)
\end{align}
where $TS(E)$ is the tangent bundle of $S(E)$. Note that since $TS(E)$ is a real vector bundle of rank 2 with an orientation and a metric, $TS(E)$ is a Hermitian line bundle.

Using the above three families of Clifford bundles, we define the family $(\mbfS, c, \nabla)$ of Clifford bundles along the fibers of
\[
  \mcalV = S(E) \times S(V_\HB)/U(1) \times \R^+ \times V_\R \to S(E)
\]
by
\[
  (\mbfS, c, \nabla) = (S_E, c_E, \nabla_E) \boxtimes (S_\HB, c_\HB, \nabla_\HB) \boxtimes (S_\R, c_\R, \nabla_\R).
\]
Specifically,
\begin{align*}
  \mbfS  & = S_E \boxtimes S_\HB \boxtimes S_\R,                                                                                       \\
  c      & = c_E \boxtimes 1 \boxtimes 1 + \varepsilon \boxtimes c_\HB \boxtimes 1 + \varepsilon \boxtimes \varepsilon \boxtimes c_\R, \\
  \nabla & = \nabla_E \boxtimes 1 \boxtimes 1 + 1 \boxtimes \nabla_\HB \boxtimes 1 + 1 \boxtimes 1 \boxtimes \nabla_\R.
\end{align*}

As noted at the beginning of this subsection, we will construct two more maps on $\mbfS$. One is a section $\mbfh$ of the vector bundle
\[
  \Herm(\mbfS) \to \mcalV
\]
obtained by bundling all the Hermitian maps on each fiber of $\mbfS \to \mcalV$, whose support is compact. The other is a fiberwise anti-linear map $\tau$ on $\mbfS$ that lifts the $\Z/2$-action on $\mcalV$ and satisfies $\tau^4 = 1$.

We construct $\mbfh$. To begin with, we check the structure of $\mcalW$. It is the direct sum of
\begin{align*}
  \mcalW_\HB & = S(E) \times (S(V_\HB) \times W_\HB)/U(1), \\
  \mcalW_\R  & = S(E) \times S(V_\HB)/U(1) \times W_\R.
\end{align*}
The space $\mcalW_\R$ further has a more detailed direct sum decomposition. Ignoring the $S(V_\HB)/U(1)$ part for now, we consider
\[
  \mcalW_\R^\prime = S(E) \times W_\R.
\]
By \cref{def4:fda model}(7) of a model of FDA, there is a metric-preserving isomorphism
\[
  W_\R \cong E \oplus V_\R.
\]
Moreover, on $S(E) \times E$, there is a tautological section. The vector bundle consisting of all vectors orthogonal to this section is isomorphic to $TS(E)$, so we have a metric-preserving isomorphism
\[
  \mcalW_\R^\prime \cong (S(E) \times E) \oplus (S(E) \times V_\R) \cong (S(E) \times \R^+) \oplus TS(E) \oplus (S(E) \times V_\R).
\]
Summarizing the above, an element $w \in \mcalW$ corresponds one-to-one with a tuple consisting of elements of
\[
  (S(V_\HB) \times W_\HB)/U(1),\ \R^+,\ TS(E),\ V_\R.
\]
We denote the tuple corresponding to $w \in \mcalW$ as
\[
  w_\HB,\ w_+,\ w_{TS(E)},\ w_{V_\R}.
\]

We return to the construction of $\mbfh$. It is expressed in the form
\begin{align}\label{eq4:herm map}
  \mbfh = h_E \otimes 1 \otimes 1 + \varepsilon \otimes h_\HB \otimes 1 + \varepsilon \otimes \varepsilon \otimes h_\R
\end{align}
using the pullbacks to $\mcalV$ of the sections
\[
  h_E,\ h_\HB,\ h_\R
\]
of $\Herm(S_E)$, $\Herm(S_\HB)$, $\Herm(S_\R)$, respectively. For $v \in \mcalV$, these three sections are defined by
\begin{align}\label{eq4:h}
  \begin{split}
    h_E   & = i(\mcalF(v)_{TS(E)}^\wedge - \mcalF(v)_{TS(E)}^\lrcorner),                                                                           \\
    h_\HB & = \varepsilon \otimes i(\mcalF(v)_\HB^\wedge - \mcalF(v)_\HB^\lrcorner) \otimes 1,                                                     \\
    h_\R  & = (\mcalF(v)_+^\wedge + \mcalF(v)_+^\lrcorner) \otimes 1 + \varepsilon \otimes (\mcalF(v)_{V_\R}^\wedge + \mcalF(v)_{V_\R}^\lrcorner).
  \end{split}
\end{align}

Next, we construct $\tau$. It is expressed as
\[
  \tau = \tau_E \otimes \tau_\HB \otimes \tau_\R
\]
using fiberwise anti-linear maps $\tau_E,\ \tau_\HB,\ \tau_\R$ defined on $S_E$, $S_\HB$, $S_\R$, respectively. These three maps are defined as follows. First, $\tau_E$ is defined by
\begin{align*}
  \tau_E(e, (\alpha, e^\prime)) = (-e, (\bar\alpha, -e^\prime)), &  & e \in S(E),\ \alpha \in \underline \C,\ e^\prime \in T_e S(E), \\
\end{align*}
Next, $\tau_\HB$ is constructed as the tensor product of maps defined on
\[
  \underline{\C} \oplus \mcalO(1),\ \Lambda^\ast_\C \mcalW_\HB,\ \Lambda^\ast_\C T_{V_\HB}
\]
respectively. On $S(V_\HB) \times \C$, we define the map
\[
  (v_\HB, \alpha) \mapsto (v_\HB \cdot j, \bar \alpha)
\]
This induces new maps on $\underline{\C}$, $\mcalO(1)$ which are obtained as quotients of $S(V_\HB) \times \C$ by different $U(1)$-actions, and these are the desired maps. Moreover, on
\[
  \Lambda^\ast_\C \mcalW_\HB,\ \Lambda^\ast_\C T_{V_\HB}
\]
the desired maps are the ones induced by right-multiplying $j$ on
\[
  \mcalW_\HB,\ T_{V_\HB}.
\]
Finally, we define $\tau_\R$ as the tensor product of the map induced on the exterior algebra by $-1$ on $\R^+ \times V_\R$, and the complex conjugation map on $\C$. This completes the definition of $\tau$.

At this point, we specify the needed property for $\nabla_\HB$.

\begin{definition}\label{def4:connection on S_HB}
  Let $\mcalC_\HB$ be the set of unitary connections on $S_\HB$ satisfying the following properties:
  \begin{itemize}
    \item The triple $(S_\HB, c_\HB, \nabla_\HB)$ is a family of Clifford bundles.
    \item Let $\tilde \tau_\HB$ denote the action induced by $\tau_\HB$ on
          \[
            \coprod_{e \in S(E)} \Omega^\ast(S_{\HB, e}).
          \]
          Then we have
          \[
            \tilde \tau_\HB \circ \nabla_\HB = \nabla_\HB \circ \tilde \tau_\HB.
          \]
  \end{itemize}
\end{definition}

\begin{lemma}
  $\mcalC_\HB$ is nonempty and contractible.
\end{lemma}

\begin{proof}
  We first show that $\mcalC_\HB$ is nonempty. Since $S_\HB$ is a Hermitian holomorphic vector bundle over $S(V_\HB)/U(1)$, we can take a family of canonical connections. This is an element of $\mcalC_\HB$. The contractibility of $\mcalC_\HB$ is proved by a linear homotopy.
\end{proof}

The triple $(\mbfS, c, \nabla)$, $\mbfh$ and $\tau$ satisfy the following compatibility conditions.

\begin{proposition}\label{prop4:cmptbl h-c-t}
  \begin{enumarabicp}
    \item The maps $\mbfh$ and $c$ anti-commute.
    \item The following commutative diagrams hold.
    \[
      \begin{tikzcd}[column sep=large]
        T_{S(E)} \mcalV \times_{\mcalV} \mbfS \arrow[r, "c"] \arrow[d, "(\cdot j)\times \tau"] & \mbfS \arrow[d, "\tau"] \\
        T_{S(E)} \mcalV \times_{\mcalV} \mbfS \arrow[r, "c"] & \mbfS
      \end{tikzcd}
      \qquad
      \begin{tikzcd}[column sep=large]
        \mcalV \times_{\mcalV} \mbfS \arrow[r, "\mbfh"] \arrow[d, "(\cdot j)\times \tau"] & \mbfS \arrow[d, "\tau"] \\
        \mcalV \times_{\mcalV} \mbfS \arrow[r, "\mbfh"] & \mbfS
      \end{tikzcd}
    \]
    \item Let $\tilde \tau$ denote the action of $\tau$ on
    \[
      \coprod_{e \in S(E)} \Omega^\ast(\mbfS_e).
    \]
    Then we have
    \[
      \tilde \tau \circ \nabla = \nabla \circ \tilde \tau.
    \]
  \end{enumarabicp}
\end{proposition}

\begin{definition}\label{def4:Dirac type op}
  Let $\mcalD$ denote the family of Dirac operators constructed from the family of Clifford bundles $(\mbfS, c, \nabla)$.
\end{definition}

The compatibility of $\mcalD$ with $\mbfh$, $\tau$ is described as follows.

\begin{proposition}\label{prop4:cmptbl D-c-t}
  Let $\tilde \tau$ denote the action of $\tau$ on
  \[
    \coprod_{e \in S(E)} \Gamma(\mbfS_e).
  \]
  Then we have
  \[
    \tilde \tau \circ \mcalD = \mcalD \circ \tilde \tau,\ \tilde \tau \circ \mbfh = \mbfh \circ \tilde \tau.
  \]
  Furthermore,
  \[
    \map{\mcalD \mbfh + \mbfh \mcalD}{\coprod_{e \in S(E)} \Gamma(\mbfS_e)}{\coprod_{e \in S(E)} \Gamma(\mbfS_e)}
  \]
  is a family of order 0 operators.
\end{proposition}

\begin{remark}\label{rem4:ord of tensor}
  In the definitions of $c$ and $\mbfh$, the $\Z/2$-grading operator $\varepsilon$ appears asymmetrically, so at first glance, these objects seem to depend on the order of taking the tensor product. However, as explained below, we can take an action of the symmetric group $\mfrakS_3$ on $\mbfS$ that is compatible with how $c$, $\mbfh$ change, and this action preserves $\tau$ and $\nabla$. From this, we see that the order of taking the tensor product can be freely interchanged as needed. This fact will be used in \cref{ssec5:real vb} and \cref{ssec5:quat vb}.

  For simplicity, we describe how to construct the $\mfrakS_2$-action when $\mbfS$ is the tensor product of two vector bundles $S_0$ and $S_1$, i.e.,
  \[
    \mbfS = S_0 \otimes S_1.
  \]
  The construction for the case of three tensor factors is similar. Initially, when tensoring in the original order, we have
  \begin{align*}
    \mbfh & = h_0 \otimes 1 + \varepsilon_0 \otimes h_1,\ c = c_0 \otimes 1 + \varepsilon_0 \otimes c_1.
  \end{align*}
  If we tensor $S_1$ first, we have
  \begin{align*}
    \mbfh^\prime & = h_0 \otimes \varepsilon_1 + 1 \otimes h_1,\ c^\prime = c_0 \otimes \varepsilon_1 + 1 \otimes c_1.
  \end{align*}
  (We identify $S_1 \otimes S_0$ with $\mbfS$ by a simple permutation of factors.) We define $\map{G}{\mbfS}{\mbfS}$ by
  \begin{itemize}
    \item $-1$ on $S_0^- \otimes S_1^-$,
    \item $+1$ on the other direct summands.
  \end{itemize}
  It is verified that $G$ takes $\mbfh$ to $\mbfh^\prime$ and $c$ to $c^\prime$. Moreover, $\tau$ and $\nabla$ are invariant under this action.
\end{remark}

\subsubsection{Construction of a triple for spin structure}\label{sssec4:const of spin}

Using the $\mcalD$ and $\mbfh$ constructed in \cref{sssec4:Fred op}, we construct a family of Fredholm operators on $\mbfS$. These operators are constructed as
\[
  \mcalD + t\mbfh
\]
for sufficiently large $t > 0$. The following proposition describes how large $t$ needs to be.

\begin{proposition}\label{prop4:t-delta pt}
  Define a compact subset $K$ of $\mcalV$ by
  \[
    K = S(E) \times S(V_\HB)/ U(1) \times B(V_\R) \times B(\R^+).
  \]
  The set $T$ of all $t > 0$ with the following property is non-empty and contractible: there exists $\bar \lambda > 0$ such that for any $s \geq t$, outside of $K$ we have
  \[
    s(\mcalD \mbfh + \mbfh \mcalD) + s^2 \mbfh^2 > \bar \lambda.
  \]
  Furthermore, when $t \in T$, the family of operators
  \begin{align}\label{eq4:Fred op}
    \mcalD + t\mbfh
  \end{align}
  from the $L^2_1$ space to the $L^2$ space, parameterized by $S(E)$, is a family of Fredholm operators.
\end{proposition}

\begin{proof}
  For the fact that $T$ is non-empty and $\mcalD + t\mbfh$ is Fredholm when $t \in T$, refer to Furuta\cite[Assumption 3.22, 3.25, Corollary 5.28]{Furuta2007Index-Theorem}. The contractibility is trivial.
\end{proof}

\begin{definition}\label{def4:det line bdl pt}
  Let $\mscrF = (\pt, E, V, W, i, D, F)$ be a model of FDA for h$K3$. Choose $t \in T$. Then, a complex line bundle
  \[
    \det (\mcalD + t\mbfh)
  \]
  is constructed over $S(E)$. Furthermore, $\tau$ induces
  \[
    \map{\tiota_t}{\det (\mcalD + t\mbfh)}{\det (\mcalD + t\mbfh)}
  \]
  which is an anti-linear map covering $\iota$.
\end{definition}

\begin{proposition}\label{thm4:det D+tiota pt}
  Let $\mscrF= (\pt, E, V, W, i, D, F)$ be a model of FDA for families of h$K3$. Then $(E, \det (\mcalD + t\mbfh), \tiota_t)$ is a triple for spin structure on $E$.
\end{proposition}

\begin{proof}
  We show that $c_1(\det (\mcalD + t\mbfh))$ is the positive generator of $H^2(S(E); \Z)$. The class $c_1(\det (\mcalD + t\mbfh))$ is equal to the first Chern class of the index bundle of $\mcalD + t\mbfh$. By the families index theorem, it is equal to the degree 2 term of
  \[
    \int_{\mcalV \to S(E)} \frac{1 - e^{2\omega}}{2\omega} \parenlr{\frac{1 - e^x}{x}}^{\dim_\C W_\HB} \parenlr{\frac{x}{e^{x/2} - e^{-x/2}}}^{\dim_\C V_\HB} (1 - e^x) \mcalF^\ast \tau_{\mcalW}.
  \]
  Here, $\omega$ is the pullback of the positive generator of $H^2(S(E); \Z)$, $x$ is the pullback of $c_1(\mcalO(1))$, and
  \[
    \tau_\mcalW \in H^{\rank \mcalW}(\mcalW, \mcalW \setminus (S(E) \times S(V_\HB)) / U(1); \Z)
  \]
  is the Thom class of $\mcalW$. Its degree 2 term is equal to
  \[
    \int_{\mcalV \to S(E)} x \cdot \mcalF^\ast \tau_{\mcalW} = c(\mscrF)
  \]
  and by the assumption stated in \cref{def4:fda model hK3}(2), this is the positive generator of $H^2(S(E); \Z)$.
  It remains to show that $\tiota^2 = -1$. Since $\tau^4 = 1$, we have $\tiota^4 = 1$. Therefore, $\tiota^2$ is either $+1$ or $-1$. When $c_1(L) = 1$, an anti-linear map over $L$ covering the antipodal map on $S(E)$ cannot square to $+1$.
\end{proof}

\begin{remark}\label{rem4:split det line}
  We revise the structure of the proof of \cref{thm4:det D+tiota pt}. Since $S_0 = \underline \C \oplus \mcalO(1)$, we have a direct sum decomposition
  \[
    \mbfS = \mbfS_\C \oplus \mbfS_{\mcalO(1)}.
  \]
  For a suitable choice of $\nabla_\HB$, $\mcalD$ and $\mbfh$ can be written as direct sums of certain operators
  \[
    \mcalD_\C,\ \mcalD_{\mcalO(1)},\ \mbfh_\C,\ \mbfh_{\mcalO(1)}
  \]
  defined on $\mbfS_\C$ and $\mbfS_{\mcalO(1)}$ respectively. We have
  \[
    \det (\mcalD + t\mbfh) \cong \det (\mcalD_\C + t\mbfh_\C) \otimes \det (\mcalD_{\mcalO(1)} + t\mbfh_{\mcalO(1)}).
  \]
  The map $\tiota$ on $\det \mcalD$ is also expressible as the tensor product of certain maps $\tiota_\C$, $\tiota_{\mcalO(1)}$ on $\det (\mcalD_\C + t\mbfh_\C)$, $\det (\mcalD_{\mcalO(1)} + t\mbfh_{\mcalO(1)})$. From \cref{thm4:det D+tiota pt}, we see that:
  \begin{enumarabicp}
    \item On $\Ker \mcalD_\C$, the self-isomorphism induced by the $-1 \in U(1)$-action on $\mscrF$ is $1$.
    \item On $\Ker \mcalD_{\mcalO(1)}$, this self-isomorphism is $-1$.
    \item The virtual rank of $\Ker \mcalD_\C$, $\Ker \mcalD_{\mcalO(1)}$ are both odd.
  \end{enumarabicp}
  The last assertion follows from the fact that if the virtual rank of $\Ker \mcalD_{\mcalO(1)}$ were even, then $\tiota^2$ would be $+1$.

  The virtual rank of $\Ker \mcalD_\C$ coincides with the Seiberg--Witten invariant of the homotopy K3 surface $X$. Thus (3) provides an alternative proof of the result of Morgan--Szab\'o\cite[Theorem 1.1]{Morgan--Szabo-homotopy-K3-1997} that the Seiberg--Witten invariant of a homotopy K3 surface is odd.
\end{remark}

Using \cref{thm4:det D+tiota pt} and \cref{prop3:spin from triple}, for a model of FDA
\[
  \mscrF = (\pt, E, V, W, i, D, F)
\]
for h$K3$, and for each choice of $\nabla_\HB \in \mcalC_\HB$ and $t \in T$, we obtain a spin structure on $E$. We denote it by
\[
  \mfrakt_{\mscrF, \nabla_\HB, t}.
\]

The following proposition shows that the spin structures are canonically identified for different choices of $t$.

\begin{proposition}\label{prop4:indep t pt}
  Let $\mscrF = (\pt, E, V, W, i, D, F)$ be a model of FDA for h$K3$. Fix $\nabla_\HB \in \mcalC_\HB$. For $s, t \in T$, there exists a canonical isomorphism
  \[
    \map{\Phi_{ts}}{\mfrakt_{\mscrF, \nabla_\HB, s}}{\mfrakt_{\mscrF, \nabla_\HB, t}}.
  \]
  Moreover, if we take another $u \in T$, then
  \[
    \Phi_{ut} \circ \Phi_{ts} = \Phi_{us}.
  \]
\end{proposition}

\begin{proof}
  The key point is that $T$ is contractible. By parametrizing the construction of \cref{def4:det line bdl pt} over $T$, a triple
  \[
    (T \times E, L, \tiota)
  \]
  for spin structure on $T \times E$ is constructed. \cref{prop3:simp-conn} shows the existence of canonical isomorphisms and their compatibility.
\end{proof}

Similarly, the following can be shown.

\begin{proposition}\label{prop4:indep nabla}
  Let $\mscrF = (\pt, E, V, W, i, D, F)$ be a model of FDA for h$K3$. Then, the spin structure $\mfrakt_{\mscrF, \nabla_\HB, t}$ does not depend on the choice of $\nabla_\HB$.
\end{proposition}

\subsection{Construction when the base space is not a point}\label{ssec4:base not pt}

In \cref{ssec4:spin from fda}, we proceeded with the discussion, assuming that the base space $U$ of the vector bundle $E$ is a single point. Even when $U$ is not a point, we can construct a spin structure on $E$ from a model
\[
  (U, E, V, W, i, D, F)
\]
of FDA for families of h$K3$. The strategy is the same as in the one-point case, which is to construct a family of Fredholm operators with a slight modification. First, we define $\mcalC_\HB^U$ as the set of all families of connections whose restriction to each fiber belongs to $\mcalC_\HB$. Second, the resulting family of operators $\mcalD + t \mbfh$ may not be Fredholm over the entire $U$, so we need the following modification.

We begin with some preparation of notation. For $b \in U$, let $K_b$ be a subset of $\mcalV_b$ given by
\[
  K_b = S(E_b) \times S(V_{\HB, b}) / U(1) \times B(V_{\R, b}) \times B(\R^+).
\]
For an open subset $U^\prime$ of $U$, we set
\[
  K_{U^\prime} = \coprod_{b \in U^\prime} K_b.
\]

Fix $\nabla_\HB \in \mcalC_\HB^U$. For each $t > 0$, we define $U_t$ as the set of all points satisfying the following uniform estimate for some $\bar \lambda > 0$: for any $s \geq t$, we have
\begin{align}\label{eq4:large t}
  s(\mcalD \mbfh + \mbfh \mcalD) + s^2\mbfh^2 > \bar \lambda
\end{align}
on $\restr{\mcalV}{U_t} \setminus K_{U_t}$.

\begin{proposition}\label{prop4:t-delta fam}
  Let $\mscrF = (U, E, V, W, i, D, F)$ be a model of FDA for h$K3$. For each $t > 0$, $U_t$ is an open set of $U$. Furthermore,
  \[
    U = \bigcup_{t > 0} U_t.
  \]
\end{proposition}

\cref{prop4:t-delta fam} follows from the same argument as Furuta\cite[Assumption 3.22, 3.25, Corollary 5.28]{Furuta2007Index-Theorem}.

With the above setup, we proceed similarly to \cref{def4:det line bdl pt} and \cref{thm4:det D+tiota pt}. First of all, a complex line bundle $\restr{\det (\mcalD + t\mbfh)}{U_t}$ is constructed over $S(\restr{E}{U_t})$. Furthermore, $\tau$ induces
\[
  \map{\tiota_t}{\restr{\det (\mcalD + t\mbfh)}{U_t}}{\restr{\det (\mcalD + t\mbfh)}{U_t}}
\]
which is an anti-linear map covering $\iota$. Then the triple
\[
  (\restr{E}{U_t}, \restr{\det (\mcalD + t\mbfh)}{U_t}, \tiota_t)
\]
is a triple for spin structure on $\restr{E}{U_t}$. From this we obtain a spin structure on $\restr{E}{U_t}$ denoted by $\mfrakt_{\mscrF, \nabla_\HB, t}$.

The same argument as in \cref{prop4:indep t pt} shows the following.

\begin{proposition}\label{thm4:indep t fam}
  Let $U$ be a topological space that is locally simply-connected and homeomorphic to an open set of some paracompact Hausdorff space. Let
  \[
    \mscrF = (U, E, V, W, i, D, F)
  \]
  be a model of FDA for h$K3$. Fix $\nabla_\HB \in \mcalC_\HB^U$. Let $s < t$ be positive real numbers. Then, there exists a canonical isomorphism
  \[
    \map{\Phi_{ts}}{\mfrakt_{\mscrF, \nabla_\HB, s}}{\restr{\mfrakt_{\mscrF, \nabla_\HB, t}}{U_s}}.
  \]
  Moreover, if we take another $u > t$, we have $\Phi_{ut} \circ \Phi_{ts} = \Phi_{us}$.
\end{proposition}

By \cref{thm4:indep t fam}, the spin structures $\mfrakt_{\mscrF, \nabla_\HB, t}$ on $\restr{E}{U_t}$ can be glued together to define a spin structure on $E$. This spin structure depends on the various choices of the auxiliary data. The following lemma shows that it actually depends on none of them.

\begin{lemma}\label{lem4:indep rho}
  Let $U$ be a topological space that is locally simply-connected and homeomorphic to an open set of some paracompact Hausdorff space. Let
  \[
    \mscrF = (U, E, V, W, i, D, F)
  \]
  be a model of FDA for h$K3$. Then the spin structure on $E$ constructed from \cref{thm4:indep t fam} is independent of the choices of $\nabla_\HB$, $\rho$ in \cref{eq4:rho}, and the $\Z/2$-invariant Riemannian metric on $S(V_\HB)/U(1)$.
\end{lemma}

\begin{proof}
  This is an application of \cref{prop3:simp-conn}. The key point is that the space of all data specified above is contractible.
\end{proof}

With these preparations, we can now prove \cref{thm4:spin from fda}.

\begin{proof}[Proof of \cref{thm4:spin from fda}]
  So far, we have constructed a spin structure $\mfrakt_\mscrF$ on the bundle $E$, which is the statement of (1).

  We prove the correspondence of self-isomorphisms in (2). A self-isomorphism of $\mscrF$ induces self-diffeomorphisms $\tilde f$, $\tilde g$ of $\mcalV$ and $\mcalW$ respectively, satisfying the conditions
  \[
    \mcalF \circ \tf = \tg,\ D \circ \tf = D,\ \mbfh \circ \tf = \mbfh.
  \]
  Therefore, $\tilde f$ and $\tilde g$ induce self-isomorphisms of $\mfrakt_\mscrF$.

  To prove the correspondence of self-isomorphisms, it suffices to consider the case when $U$ is a point. Using the notation of \cref{def4:det line bdl pt}, it is clear that $+1$ corresponds to $+1$. We prove that $-1$ corresponds to $-1$. Since $j \in Pin(2)$ satisfies
  \[
    j^2 = -1,
  \]
  the self-isomorphism of $\det (\mcalD + t\mbfh)$ induced by the self-isomorphism $-1$ of $\mscrF$ is equal to the square of the self-isomorphism $\tiota_t$ of $L$ induced by $j$, which is
  \[
    \map{\tiota_t^2 = -1}{\det (\mcalD + t\mbfh)}{\det (\mcalD + t\mbfh)}.
  \]
  This is equal to $-1$ as a self-isomorphism of the spin structure $\mfrakt_\mscrF$.
\end{proof}

In \cref{def4:aux data} and \cref{def4:proj}, we discussed the set $\tilde \mcalA$ of auxiliary data needed for the construction of finite-dimensional approximations of the family of the Seiberg--Witten maps, as well as the set
\[
  \mcalP(R, U, \varepsilon, \lambda)
\]
of projection-like maps determined by specifying an element $(R, U, \varepsilon, \lambda) \in \tilde \mcalA$. By fixing $(R, U, \varepsilon, \lambda) \in \tilde \mcalA$ and $p \in \mcalP(R, U, \varepsilon, \lambda)$, we have obtained a model of FDA for families of h$K3$ via the method of \cref{eg4:fda from sw}. As the final task of this section, we prove that the spin structure on $\hplus{\restr{\mbbX}{U}}$ constructed from these data is independent of the choice of the projection-like map $p \in \mcalP(R, U, \varepsilon, \lambda)$. (The independence of the choice of $(R, U, \varepsilon, \lambda)$ is given by \cref{thm5:can iso}.)

Let us recall the setup. Let $X$ be a homotopy K3 surface, and let $X \to \mbbX \to B$ be the family of $X$ associated with the principal $\Diffplus$-bundle $\mcalE \to B$. We assume that the base space $B$ is a topological space that is locally simply-connected and homeomorphic to an open set of some paracompact Hausdorff space. We assume that $\mbbX$ is endowed with a continuous family of smooth Riemannian metrics on the fibers. We also assume that a lift $\tmcalE$ of $\mcalE$ to a principal $\Diffspin$-bundle is given.

\begin{proposition}\label{thm4:indep proj}
  Fix $(R, U, \varepsilon, \lambda) \in \tilde \mcalA$ (see \cref{def4:proj}). Choose $p \in \mcalP(R, U, \varepsilon, \lambda)$, and denote by $\mscrF_p$ the model of FDA for families of h$K3$ constructed from this data. Then for different choices of $p$, the spin structures $\mfrakt_{\mscrF_p}$ on $\hplus{\restr{\mbbX}{U}}$ constructed from them are canonically isomorphic. Moreover, this isomorphism is compatible for any three choices of $p$.
\end{proposition}

\begin{proof}
  This is essentially an application of \cref{prop3:simp-conn}. The only thing we need to care about is that a triple for spin structure may not be globally defined on $U$. We first construct an isomorphism between $\mfrakt_{\mscrF_p}$'s on small open sets by using \cref{prop3:simp-conn}, and we glue them together to obtain a global isomorphism.
\end{proof}

\section{Relationship between stabilization of FDA and induced spin structure}\label{sec5:swap fda}

Recall the setting of \cref{ssec4:sw fda}. That is, let $B$ be a topological space that is locally simply-connected and homeomorphic to an open set of some paracompact Hausdorff space, $\mcalE \to B$ be a principal $\Diffplus$-bundle, and $X \to \mbbX \to B$ be a family of $X$ associated with $\mcalE$. Assume that $\mbbX$ is equipped with a continuous family of smooth Riemannian metrics on each fiber. Assume that a lift $\tmcalE$ of the principal $\Diffplus$-bundle $\mcalE$ to a principal $\Diffspin$-bundle is given. Let $(R_1, U_1, \varepsilon_1, \lambda_1),\ (R_2, U_2, \varepsilon_2, \lambda_2) \in \tilde \mcalA$. (See \cref{def4:proj}.) From \cref{thm4:indep proj}, spin structures
\[
  \mfrakt_1,\ \mfrakt_2
\]
on $\hplus{\restr{\mbbX}{U_1}}$, $\hplus{\restr{\mbbX}{U_2}}$ are constructed. The goal is to prove the following:

\begin{theorem}\label{thm5:can iso}
  Let $(R_1, U_1, \varepsilon_1, \lambda_1),\ (R_2, U_2, \varepsilon_2, \lambda_2) \in \tilde \mcalA$. For each of these pairs, a canonical isomorphism of spin structures
  \[
    \map{\Phi_{21}}{\restr{\mfrakt_1}{U_1 \cap U_2}}{\restr{\mfrakt_2}{U_1 \cap U_2}}
  \]
  can be constructed, satisfying the following property: when $(R_3, U_3, \varepsilon_3, \lambda_3)$ is a third element of $\tilde \mcalA$, over $U_1 \cap U_2 \cap U_3$ we have
  \[
    \Phi_{32} \circ \Phi_{21} = \Phi_{31}.
  \]
\end{theorem}

\cref{thm5:can iso} will be proved in the following steps. Let $U$ be a topological space that is locally simply-connected and homeomorphic to an open set of some paracompact Hausdorff space. Let $E \to U$ be a rank 3 vector bundle over $U$ with a given orientation and metric.
\begin{itemize}
  \item In general, when a model $\mscrF$ of FDA for families of h$K3$ is given, we prove the following claim: when $V_\R^{\prime}$ is a real vector bundle over $U$, by taking the direct sum of $\id_{V_\R^{\prime}}$ with $\mscrF$, a new model $\mscrF^{\prime}$ of FDA for families of h$K3$ can be obtained. We prove that there is a canonical isomorphism between the spin structures $\mfrakt_\mscrF$ and $\mfrakt_{\mscrF^{\prime}}$ constructed from each.
  \item We show that there is a similar canonical isomorphism when taking the direct sum with a quaternionic vector bundle $V_\HB^{\prime}$ over $U$.
  \item From the construction, when taking the direct sum of both $V_\R^{\prime}$ and $V_\HB^{\prime}$, the order of the direct sum does not affect the resulting isomorphism of spin structures.
  \item We apply the above statements to the models of FDA for families of h$K3$ constructed from $(R_1, U_1, \varepsilon_1, \lambda_1)$ and $(R_2, U_2, \varepsilon_2, \lambda_2)$.
\end{itemize}

\subsection{The case of taking the direct sum with a real vector bundle}\label{ssec5:real vb}

Let $U$ be a topological space that is locally simply-connected and homeomorphic to an open set of some paracompact Hausdorff space. Let $E \to U$ be a rank 3 vector bundle over $U$ with an orientation and a metric. In this section, we fix a model $\mcalF = (U, E, V, W, D, i, F)$ of FDA for families of h$K3$ and consider the case of taking the direct sum with a $Pin(2)$-equivariant real vector bundle $V_\R^{\prime}$ with a metric. The $Pin(2)$-action is introduced as follows: $U(1)$ acts trivially, and $j$ acts by multiplication by $-1$.

Let
\[
  D^{\prime} = D \oplus \id_{V_\R^{\prime}}.
\]
Let
\[
  \map{F^{\prime}}{S(E) \times_U B(V_\HB) \times_U B(V_\R \oplus V_\R^{\prime})}{S(E) \times_U (W \oplus V_\R^{\prime})}
\]
be the map obtained by juxtaposing the identity on $V_\R^{\prime}$ and $F$. Then,
\[
  \mscrF^{\prime} = (E, V \oplus V_\R^{\prime}, W \oplus V_\R^{\prime}, D^{\prime}, i, F^{\prime})
\]
is also a model of FDA for families of h$K3$. The goal of this subsection is to prove the following theorem:

\begin{theorem}\label{thm5:stab for R}
  Let $U$ be a topological space that is locally simply-connected and homeomorphic to an open set of some paracompact Hausdorff space. Let $E \to U$ be a rank 3 vector bundle over $U$ with an orientation and a metric. Let $\mscrF = (U, E, V, W, D, i, F)$ be a model of FDA for families of h$K3$, and let $V_\R^{\prime}$ be a $Pin(2)$-equivariant real vector bundle with a metric. Let $\mscrF^{\prime}$ be the model of FDA for families of h$K3$ obtained by taking the direct sum of $V_\R^{\prime}$ with $\mscrF$. Let $\mfrakt_\mscrF$, $\mfrakt_{\mscrF^\prime}$ be spin structures on $\restr{E}{U}$ constructed from $\mscrF$ and $\mscrF^\prime$. Then a canonical isomorphism
  \[
    \map{\Phi_{\mscrF^{\prime} \mscrF}}{\mfrakt_{\mscrF}}{\mfrakt_{\mscrF^{\prime}}}
  \]
  can be constructed. This isomorphism satisfies
  \[
    \Phi_{\mscrF^{\pprime}\mscrF^{\prime}} \circ \Phi_{\mscrF^{\prime}\mscrF} = \Phi_{\mscrF^{\pprime}\mscrF}
  \]
  for $\mscrF^{\pprime}$ obtained by taking the direct sum of a $Pin(2)$-equivariant real vector bundle $V_\R^{\pprime}$ with a metric to $\mscrF^{\prime}$.
\end{theorem}

The key to the proof is that by taking the direct sum with $V_\R^{\prime}$, the difference of families of differential operators described in \cref{sssec4:const of spin} is described by something similar to a super-symmetric harmonic oscillator.

Here, we explain the proof of \cref{thm5:stab for R} limited to the case where the base space of the vector bundle $E$ is a point. The general case can be modified according to the changes explained in \cref{ssec4:base not pt}.

Define the Clifford bundle
\[
  (S_\R^{\prime}, c_\R^{\prime}, \nabla_\R^{\prime})
\]
on $V_\R^{\prime}$ as follows. First, let
\[
  S_\R^{\prime} = V_\R^{\prime} \times (\Lambda_\R^{\ast} V_\R^{\prime} \otimes \C),
\]
and for $v \in V_\R^{\prime}$, $\tilde v \in T_v V_\R^{\prime} \cong V_\R^{\prime}$, let
\[
  c_{V_\R^{\prime}}(\tilde v) = \tilde v^{\wedge} - \tilde v^{\lrcorner}.
\]
Also, let $\nabla_\R^{\prime}$ be the Levi--Civita connection. $\nabla_\R^{\prime}$ is equal to the trivial connection. Through the isomorphism
\[
  V_\R^{\prime} \times \Lambda_\R^{\ast} V_\R^\prime \otimes_\R \C \cong \Omega^{\ast}(V_\R^{\prime}; \C),
\]
the Dirac type operator $\mcalD_\R^{\prime}$ defined by this Clifford bundle coincides with $d + d^{\ast}$. The family of Hermitian maps
\[
  h_{V_\R^{\prime}} \in \Gamma(\Herm(S_\R^{\prime}))
\]
is defined by the following equation: for $v \in V_\R^{\prime}$,
\[
  \bar h_{V_\R^{\prime}}(v) = v^{\wedge} + v^{\lrcorner}.
\]

Moreover, define $\tau_{V_\R^{\prime}}$ by the following equation: for $v \in V_\R^{\prime}$, $\omega \in \Lambda_\R^{\ast} V_\R^{\prime}$, $z \in \C$,
\[
  \tau_{V_\R^{\prime}}(v, \omega \otimes z) = (-v, (-1)^{\deg \omega} \omega \otimes \bar z).
\]

The operators constructed from $\mscrF^{\prime}$ are close to the direct sum of the operators constructed from $\mscrF$ with the above data. However, due to the effect of $\rho$ taken in \cref{eq4:rho}, it is necessary to consider $h_{V_\R^{\prime}}$ defined by the following equation instead of $\bar h_{V_\R^{\prime}}$:
\[
  h_{V_\R^{\prime}}(v) = \bar h_{V_\R^{\prime}}\parenlr{\rho(\abs{v})\frac{v}{\abs{v}}}.
\]
$h_{V_\R^{\prime}}$ is equal to $\bar h_{V_\R^{\prime}}$ near the origin and is constant in the radial direction when $\abs{v} \geq 1$.

\begin{proposition}\label{prop5:ker pssho}
  When $t_0 > 0$ is sufficiently large, there is an isomorphism of vector bundles that commutes with $\tau_{V_\R^{\prime}}$ and preserves the $\Z/2$-grading
  \[
    \coprod_{t \geq t_0} \Ker (\mcalD_{V_\R^{\prime}} + t \bar h_{V_\R^{\prime}}) \cong \coprod_{t \geq t_0} \Ker (\mcalD_{V_\R^{\prime}} + t h_{V_\R^{\prime}}).
  \]
  This isomorphism is unique up to homotopy through degree-preserving isomorphisms that commute with $\tau_{V_\R^{\prime}}$.
\end{proposition}

The proof of \cref{prop5:ker pssho} can be done using the technique of the Witten deformation. For the Witten deformation, the readers are referred to Zhang\cite{Zhang-Witten-deformation-2001}, Furuta\cite{Furuta2007Index-Theorem} and Miyazawa\cite{miyazawa2021localization}.

\begin{remark}\label{rem5:triv pssho}
  As a $\tau_{V_\R^\prime}$-invariant element in the kernel of the super-symmetric harmonic oscillator $(\mcalD_{V_\R^\prime} + t \bar h_{V_\R^\prime})$,
  \[
    e^{-\frac{1}{2}t\abs{v}^2}
  \]
  can be taken. This gives a trivialization of the orientation of the vector bundle on the left-hand side of \cref{prop5:ker pssho}. Therefore, for $t \geq t_0$, an isomorphism
  \[
    \Ker (\mcalD_{V_\R^\prime} + t h_{V_\R^\prime}) \cong \C
  \]
  can be uniquely determined up to positive real scalar multiplication, such that $\tau_{V_\R^\prime}$ corresponds to the complex conjugation on $\C$. Moreover, this isomorphism is compatible with the isomorphism
  \[
    \Ker (\mcalD_{V_\R^\prime \oplus V_\R^\pprime} + t h_{V_\R^\prime\oplus V_\R^\pprime}) \cong \Ker (\mcalD_{V_\R^\prime} + t h_{V_\R^\prime}) \otimes \Ker (\mcalD_{V_\R^\pprime} + t h_{V_\R^\pprime})
  \]
  determined by another $Pin(2)$-equivariant real vector space $V_\R^\pprime$ with a metric.
\end{remark}

\begin{proposition}\label{prop5:pssho from +R}
  Let $(\mbfS, c, \nabla)$ be the vector bundle constructed from $\mscrF$ in \cref{ssec4:spin from fda}, $h$ be the family of Hermitian maps on it, $c$ be the Clifford action, and $\tau$ be the fiberwise anti-linear map covering the antipodal map $\iota$ on $S(E)$. Similarly, let $(\mbfS^\prime, c^\prime, \nabla^\prime)$, $h^\prime$, $\tau^\prime$ be the data constructed from $\mscrF^\prime$. Then, there is a canonical isomorphism
  \[
    (\mbfS^\prime, c^\prime, \nabla^\prime) \cong (\mbfS, c, \nabla) \boxtimes (\mbfS_\R^\prime, c_\R^\prime, \nabla_\R^\prime).
  \]
  Moreover, a homotopy between
  \[
    h^\prime,\ h \otimes 1 + \varepsilon \otimes h_{V_\R^\prime}
  \]
  that commutes with $\tau^\prime$ can be specified. This homotopy satisfies the following condition: when $V_\R^\pprime$ is further added as a direct sum, there are two different homotopies between
  \[
    h^\pprime,\ h \otimes 1 \otimes 1 + \varepsilon \otimes h_{V_\R^\prime} \otimes 1 + \varepsilon \otimes \varepsilon \otimes h_{V_\R^\pprime},
  \]
  depending on whether $V_\R^\prime \oplus V_\R^\pprime$ is added as a direct sum at once or one by one. These two homotopies themselves are homotopic.
\end{proposition}

\begin{proof}
  From what was mentioned in \cref{rem4:ord of tensor}, the order of the tensor product can be chosen freely. The first isomorphism is obtained by considering the exterior algebra of $\R^+ \times V_\R \times V_\R^\prime$ being tensored at the end.

  Let $\bm{v} = (v, v^\prime) \in V \times V_\R^\prime$. Using $\rho$ taken in \cref{eq4:rho},
  \begin{align*}
    h^\prime(\bm{v})                                            & = (h \otimes 1)\parenlr{\rho(\abs{\bm{v}}) \frac{v}{\abs{\bm{v}}}} + (\varepsilon \otimes \bar h_{V_\R^\prime})\parenlr{\rho(\bm{v}) \frac{v^\prime}{\abs{\bm{v}}}}, \\
    (h \otimes 1 + \varepsilon \otimes h_{V_\R^\prime})(\bm{v}) & = (h \otimes 1)\parenlr{\rho(\abs{v}) \frac{v}{\abs{v}}} + (\varepsilon \otimes \bar h_{V_\R^\prime})\parenlr{\rho(\abs{v^\prime})\frac{v^\prime}{\abs{v^\prime}}}
  \end{align*}
  can be written.
  Let
  \[
    a_s(\bm{v}) = (1 - s) \bm{v} + s v,\ b_s(v) = (1 - s) \bm{v} + s v^\prime.
  \]
  The homotopy can be defined by
  \[
    H(s, \bm{v}) = (h \otimes 1)\parenlr{\rho(a_s(\bm{v})) \frac{v}{a_s(\bm{v})}} + (\varepsilon \otimes \bar h_{V_\R^\prime})\parenlr{\rho(b_s(\bm{v}))\frac{v^\prime}{b_s(\bm{v})}}.
  \]
  The last claim can be proved by explicitly writing down the two homotopies.
\end{proof}

\begin{proof}[Proof of \cref{thm5:stab for R}]
  It suffices to construct a canonical isomorphism up to homotopy between the triples for spin structure on $E$. From \cref{prop5:pssho from +R}, it suffices to specify an identification between the triple for spin structure on $E$ made from the one obtained by replacing $h^\prime$ after the direct sum with $h \otimes 1 + \varepsilon \otimes h_{V_\R^\prime}$ and the triple for spin structure on $E$ made from the one before the direct sum. The Dirac type operator $\mcalD^\prime$ constructed from $\mscrF^\prime$ in \cref{def4:Dirac type op} is the tensor product of the Dirac type operator $\mcalD$ made from $\mscrF$ and the Dirac type operator $\mcalD_{V_\R^\prime}$ on $V_\R^\prime$. Therefore, the kernel of
  \[
    \mcalD^\prime + t (h \otimes 1 + \varepsilon \otimes h_{V_\R^\prime})
  \]
  is isomorphic to the tensor product of the kernels of
  \[
    \mcalD+ t h,\ \mcalD_{V_\R^\prime} + t h_{V_\R^\prime},
  \]
  respectively. When $t > 0$ is taken sufficiently large, the kernel of the deformed super-symmetric harmonic oscillator is given a trivialization by the method described in \cref{rem5:triv pssho}. From this, an identification between the triples for spin structure on $E$ is given.

  The compatibility when another $V_\R^\pprime$ is added as a direct sum follows from the compatibility of the trivialization of the kernel of the deformed super-symmetric harmonic oscillator explained in \cref{rem5:triv pssho} and the existence of a homotopy between the homotopies stated in \cref{prop5:pssho from +R}.
\end{proof}

\subsection{The case of taking the direct sum with a quaternionic vector bundle}\label{ssec5:quat vb}

Let $U$ be a topological space that is locally simply-connected and homeomorphic to an open set of some paracompact Hausdorff space. Let $E \to U$ be a rank 3 vector bundle over $U$ with an orientation and a metric. Let $\mscrF = (E, V, W, D, i, F)$ be a model of FDA for families of h$K3$. Consider the case of taking the direct sum with a $Pin(2)$-equivariant quaternionic vector bundle $V_\HB^\prime$ with a metric.
Let
\[
  D^\prime = D \oplus \id_{V_\HB^\prime}.
\]
Let
\[
  \map{F^\prime}{S(E) \times_U B(V_\HB \oplus V_\HB^\prime) \times_U B(V_\R)}{S(E) \times_U (W \oplus V_\HB^\prime)}
\]
be the map obtained by juxtaposing $F$ and the identity on $V_\HB^\prime$. Then,
\[
  \mscrF^\prime = (E, V \oplus V_\HB^\prime, W \oplus V_\HB^\prime, D^\prime, i, F^\prime)
\]
is also a model of FDA for families of h$K3$. In this subsection, we prove the following theorem.

\begin{theorem}\label{thm5:stab for H}
  Let $U$ be a topological space that is locally simply-connected and homeomorphic to an open set of some paracompact Hausdorff space. Let $E \to U$ be a rank 3 vector bundle over $U$ with an orientation and a metric. Let $\mscrF = (E, V, W, D, i, F)$ be a model of FDA for families of h$K3$, and let $V_\HB^\prime$ be a $Pin(2)$-equivariant quaternionic vector bundle with a $Pin(2)$-equivariant metric. Then, there is a canonical isomorphism
  \[
    \map{\Phi_{\mscrF^\prime \mscrF}}{\mfrakt_\mscrF}{\mfrakt_{\mscrF^\prime}}
  \]
  between the spin structures $\mfrakt_\mscrF$, $\mfrakt_{\mscrF^\prime}$ of $E$ constructed from $\mscrF$, $\mscrF^\prime$. This isomorphism satisfies
  \[
    \Phi_{\mscrF^\pprime\mscrF^\prime} \circ \Phi_{\mscrF^\prime\mscrF} = \Phi_{\mscrF^\pprime\mscrF}
  \]
  for $\mscrF^\pprime$ obtained by taking the direct sum of a vector bundle $V_\HB^\pprime$ with $\mscrF^\prime$.
\end{theorem}

When taking the direct sum with a quaternionic vector bundle $V_\HB^\prime$, the argument becomes more complicated compared to the case of a real vector bundle in the previous section. Both \cref{thm5:stab for R} and \cref{thm5:stab for H} essentially correspond to performing the pushforward in $K$-theory at the level of representatives. In the proof of \cref{thm5:stab for R}, it was enough to discuss the Thom isomorphism for trivial bundles, but in the proof of \cref{thm5:stab for H}, it is necessary to carry out an argument corresponding to the Thom isomorphism and the excision theorem for non-trivial vector bundles. In other words, the approach in this section is as follows:

\begin{enumarabicp}
  \item Families of Fredholm operators parametrized by $S(E)$ are constructed from $\mscrF$, $\mscrF^\prime$ by the method explained in \cref{ssec4:spin from fda}. The vector bundles on which these operators are defined are denoted by $\mbfS$, $\mbfS^\prime$, and their base spaces are denoted by $\mcalV$, $\mcalV^\prime$, respectively. There is a natural embedding from $\mcalV$ to $\mcalV^\prime$.
  \item Using the normal bundle $\mscrV$ of this embedding, another family of Fredholm operators is constructed. The spin structure on $E$ constructed from it is denoted by $\mfrakt_\mscrV$.
  \item A canonical isomorphism between $\mfrakt_\mscrF$ and $\mfrakt_\mscrV$ will be constructed.
  \item The normal bundle $\mscrV$ is identified with the tubular neighborhood of the embedding. A canonical isomorphism between $\mfrakt_\mscrV$ and $\mfrakt_{\mscrF^\prime}$ will be constructed.
\end{enumarabicp}
The statement (3) corresponds to the Thom isomorphism, and (4) corresponds to the excision theorem.

We will construct a Clifford bundle on a normal bundle $\mscrV$. This is constructed only after specifying some auxiliary data, which is described in detail below. In the actual argument, two Clifford bundles are constructed on $\mscrV$. These can be identified after further specifying the auxiliary data. One of them is convenient for proving (3), and the other for proving (4).

We explain the proof of \cref{thm5:stab for H} limited to the case where the base space of the vector bundle $E$ is a point. The general case can be modified according to the changes explained in \cref{ssec4:base not pt}.

First, we organize the notation. Let
\begin{align*}
  \mcalV        & = S(E) \times S(V_\HB)/U(1) \times \R^+ \times V_\R,                     \\
  \mcalV^\prime & = S(E) \times S(V_\HB \oplus V_\HB^\prime)/U(1) \times \R^+ \times V_\R.
\end{align*}
What we truly want to consider is the natural embedding
\[
  \mcalV \to \mcalV^\prime
\]
between these, but in many parts of the following discussion, it becomes essential to consider the part involving the quaternionic vector spaces
\[
  \mbbP = S(V_\HB)/U(1),\ \mbbP^\prime = S(V_\HB \oplus V_\HB^\prime)/U(1),\ \map{f}{\mbbP}{\mbbP^\prime}.
\]
Let
\[
  \mscrV_\HB = (S(V_\HB) \times V_\HB^\prime)/U(1).
\]
Note that $\mscrV_\HB$ is the normal bundle of $f$, and
\[
  \mbbP,\ \mbbP^\prime,\ \mscrV_\HB
\]
are equipped with $\Z/2$-actions originating from $j \in Pin(2)$.
Let $B_r(\mscrV_\HB)$ denote the disk bundle of $\mscrV_\HB$ with radius $r$. Define
\begin{align}\label{eq5:embedding}
  \map{\phi}{B_{1/2}(\mscrV_\HB)}{\mbbP^\prime}
\end{align}
by
\[
  \phi([v_\HB, v_\HB^\prime]) = \bracketlr{\frac{v_\HB}{\sqrt{v_\HB^2 + v_\HB^{^\prime 2}}}, \frac{v_\HB^\prime}{\sqrt{v_\HB^2 + v_\HB^{^\prime 2}}}}.
\]
Note that $\phi$ is a holomorphic embedding. In the following, $B_{1/2}(\mscrV_\HB)$ and $\phi(B_{1/2}(\mscrV_\HB))$ are identified through $\phi$.
Let
\[
  \map{\pi}{\mscrV_\HB}{\mbbP}
\]
be the natural projection. Also, we denote all of the vector bundles obtained by dividing
\[
  S(V_\HB) \times W_\HB,\ S(V_\HB \oplus V_\HB^\prime) \times W_\HB,\ S(V_\HB) \times V_\HB^\prime \times W_\HB
\] by the diagonal $U(1)$-action by the same symbol $\mcalW_\HB$.

Let $\mcalV_\HB^\prime$ denote the vector bundles obtained by replacing $W_\HB$ with $V_\HB^\prime$. We set
\begin{align}\label{eq5:spin bdl on H comp}
  \begin{split}
    S_\mbbP               & = (\underline{\C} \oplus \mcalO(1)) \otimes \Lambda_\C^\ast T\mbbP \otimes \Lambda_\C^\ast \mcalW_\HB,                                                                    \\
    S_{\mbbP^\prime}      & = (\underline{\C} \oplus \mcalO(1)) \otimes \Lambda_\C^\ast T\mbbP^\prime \otimes \Lambda_\C^\ast \mcalW_\HB \otimes \Lambda_\C^\ast \mcalV_\HB^\prime,                   \\
    S_{\mscrV_\HB}        & = (\underline{\C} \oplus \mcalO(1)) \otimes \Lambda_\C^\ast T\mscrV_\HB \otimes \Lambda_\C^\ast \mcalW_\HB \otimes \Lambda_\C^\ast \mcalV_\HB^\prime,                     \\
    S_{\mscrV_\HB}^\prime & = (\underline{\C} \oplus \mcalO(1)) \otimes \pi^\ast \Lambda_\C^\ast T\mbbP \otimes \Lambda_\C^\ast \mcalW_\HB \otimes (\Lambda_\R^\ast \mcalV_\HB^\prime \otimes_\R \C).
  \end{split}
\end{align}
These vector bundles have $\Z/4$-actions
\[
  \tau_\mbbP,\ \tau_{\mbbP^\prime},\ \tau_{\mscrV_\HB},\ \tau_{\mscrV_\HB}^\prime
\]
originating from $j \in Pin(2)$.

Note that $S_{\mscrV_\HB}$ and $S_{\mscrV_\HB}^\prime$ are both vector bundles over $\mscrV_\HB$. These can be almost identified in the presence of each choice of the auxiliary data mentioned below. The representation $S_{\mscrV_\HB}$ is convenient for carrying out the argument corresponding to the excision theorem, and $S_{\mscrV_\HB}^\prime$ for the Thom isomorphism.

The construction of the canonical isomorphism between $\mfrakt_\mscrF$ and $\mfrakt_{\mscrF^\prime}$ is reduced to the following propositions.

\begin{proposition}\label{prop5:spin mscrV}
  By appropriately defining Clifford bundle structures on each of $S_{\mscrV_\HB}$ and $S_{\mscrV_\HB^\prime}$, spin structures $\mfrakt_\mscrV$ and $\mfrakt_\mscrV^\prime$ on $E$ is constructed canonically.
\end{proposition}

\begin{proposition}\label{prop5:can iso for vari spin}
  The following isomorphisms are constructed canonically:
  \begin{enumarabicp}
    \item $\mfrakt_\mscrF \cong \mfrakt_\mscrV^\prime$.
    \item $\mfrakt_\mscrV^\prime \cong \mfrakt_\mscrV$.
    \item $\mfrakt_\mscrV \cong \mfrakt_{\mscrF^\prime}$.
  \end{enumarabicp}
\end{proposition}

Regarding \cref{prop5:can iso for vari spin}, (1) is proved by the same argument as the proof of the Thom isomorphism. The statement (2) is proved by ``almost identifying $S_{\mscrV_\HB}$ and $S_{\mscrV_\HB}^\prime$'', and (3) is proved by the same argument as the proof of the excision theorem.

Before considering the auxiliary data, let us confirm the following points.
\begin{itemize}
  \item Once a Riemannian metric on $\mbbP$ is determined, the metrics on $S_\mbbP$ and $S_{\mscrV_\HB}^\prime$ as complex vector bundles are determined. Once a Riemannian metric on $\mbbP^\prime$ is determined, the metric on $S_{\mbbP^\prime}$ as a complex vector bundle is determined. Once a Riemannian metric on $\mscrV_\HB$ is determined, the metric on $S_{\mscrV_\HB}$ as a complex vector bundle is determined. In each case, if the Riemannian metric is $\Z/2$-invariant, the real part of the complex vector bundle metric is $\Z/4$-invariant.
  \item For each choice of Riemannian metrics on $\mbbP$, $\mbbP^\prime$, and $\mscrV_\HB$, the Clifford actions on $S_\mbbP$, $S_\mbbP^\prime$, and $S_{\mscrV_\HB}$ are determined. To define the Clifford action on $S_{\mscrV_\HB}^\prime$, in addition to the Riemannian metric on $\mscrV_\HB$, it is necessary to take a splitting of the exact sequence
        \[
          0 \to T_\mbbP \mscrV_\HB \to T\mscrV_\HB \to \pi^\ast T\mbbP \to 0.
        \]
  \item The pullback of $S_{\mbbP^\prime}$ by the embedding $B_{1/2}(\mscrV_\HB) \to \mbbP^\prime$ is naturally isomorphic to $S_{\mscrV_\HB}$. If the Riemannian metrics on $\mscrV_\HB$ and $\mbbP^\prime$ are such that $B_{1/4}(\mscrV_\HB) \to \mbbP^\prime$ is an isometry, then this isomorphism of complex vector bundles also preserves the metrics on $B_{1/4}(\mscrV_\HB)$.
  \item $\mscrV_\HB$ has a metric as a vector bundle over $\mbbP$, but it does not have a canonical Riemannian metric (even if a Riemannian metric on $\mbbP$ is specified).
\end{itemize}
The choices of Riemannian metrics on $\mbbP$, $\mbbP^\prime$, and $\mscrV_\HB$ will be specified as auxiliary data below. Suppose that Riemannian metrics on these manifolds are specified. The metrics on each complex vector bundle in \cref{eq5:spin bdl on H comp} are considered to be those induced from the Riemannian metrics as described here. When the choices of Riemannian metrics vary, the metrics on these vector bundles are also considered to vary.

Keeping the above in mind, we start to describe the auxiliary data. The first auxiliary data is a $\Z/2$-invariant Riemannian metric on $\mbbP$. In the following discussion, we fix one such metric.

The second auxiliary data is a splitting of the $\Z/2$-equivariant vector bundle exact sequence
\begin{align}\label{eq5:ex seq TmscrV}
  0 \to T_\mbbP \mscrV_\HB \to T\mscrV_\HB \to \pi^\ast T\mbbP \to 0,
\end{align}
where $T_\mbbP \mscrV_\HB$ denotes the vertical tangent bundle of $\mscrV_\HB$. For technical reasons related to the construction of Fredholm operators later, we view
\[
  \mscrV_\HB = B(\mscrV_\HB) \cup (S(\mscrV_\HB) \times [1, \infty))
\]
and only consider splittings with cylindrical ends.

\begin{definition}
  Let $\mscrS$ denote the space of all $\Z/2$-equivariant splittings of the exact sequence \cref{eq5:ex seq TmscrV} that are translation-invariant over $S(\mscrV_\HB) \times [1, \infty)$.
\end{definition}

By definition, $\mscrS$ is contractible.

\begin{lemma}\label{lem5:str from splitting}
  Fix an element $s \in \mscrS$.
  \begin{enumarabicp}
    \item In the presence of a specified $\Z/2$-invariant metric on $T_\mbbP \mscrV_\HB$, a $\Z/2$-invariant Riemannian metric $g_s^{\mscrV_\HB}$ on $\mscrV_\HB$ is determined from $s$.
    \item Using $g_s^{\mscrV_\HB}$, a Clifford action $c_\mscrV$ of $T\mscrV_\HB$ on $S_{\mscrV_\HB}$ is defined. Also, using the splitting $s$, a Clifford action $c_\mscrV^\prime$ of $T\mscrV_\HB$ on $S_{\mscrV_\HB}^\prime$ is defined.
  \end{enumarabicp}
\end{lemma}

\begin{proof}
  We prove (1). The splitting $s$ gives an isomorphism
  \[
    T\mscrV_\HB \cong T_\mbbP \mscrV_\HB \oplus \pi^\ast T\mbbP.
  \]
  By using the metrics on the right side, the metric on $T \mscrV_\HB$ is determined.

  We prove (2). Using the Clifford action defined on $\Lambda_\C^\ast T\mscrV_\HB$, we can define a Clifford action on $S_{\mscrV_\HB}$. We confirm that a Clifford action on $S_{\mscrV_\HB}^\prime$ is determined. Using $s \in \mscrS$, an element of $T\mscrV_\HB$ can be decomposed into the sum of an element of $\pi^\ast T\mbbP$ and an element of $T_\mbbP \mscrV_\HB$. The former acts on $\pi^\ast \Lambda_\C^\ast T\mbbP$ and the latter acts on $\Lambda_\R^\ast \mcalV_\HB^\prime \otimes_\R \C$.
\end{proof}

We specify how to take a metric on $T_\mbbP \mscrV_\HB$. In \cref{sec8:deform ssho}, we explain the construction of a pseudo-super-symmetric harmonic oscillator on $V_\HB^\prime$ with cylindrical end. This Riemannian metric coincides with the standard Riemannian metric on $B_{1/2}(V_\HB^\prime)$ and is defined to be isometric to
\[
  S(V_\HB^\prime) \times [1, \infty)
\]
for $r > 1$. Using the Riemannian metric used there, we can give a metric on $T_\mbbP \mscrV_\HB$.

A volume form on $\mscrV_\HB$ is determined from a Riemannian metric on $\mscrV_\HB$. It does not depend on $s$ because it is constructed via a splitting.

The third auxiliary data is a Riemannian metric on $\mbbP^\prime$.

\begin{definition}\label{def5:metric on mbbP^prime}
  Fix an element $s \in \mscrS$. Let $\mcalG(\mbbP^\prime, s)$ denote the space of all $\Z/2$-invariant Riemannian metrics on $\mbbP^\prime$ that coincide with $g_s^{\mscrV_\HB}$ on $B_{1/4}(\mscrV_\HB)$.
\end{definition}

For any $s \in \mscrS$, $\mcalG(\mbbP^\prime, s)$ is contractible. The following lemma is obvious.

\begin{lemma}
  For each choice of $s \in \mscrS$ and $g_{\mbbP^\prime} \in \mcalG(\mbbP^\prime, s)$, a Clifford action of $T\mbbP^\prime$ on $S_{\mbbP^\prime}$ is determined. This Clifford action is preserved under the isomorphism
  \[
    S_{\mscrV_\HB} \cong S_{\mbbP^\prime}
  \]
  on $B_{1/4}(\mscrV_\HB)$.
\end{lemma}

Before explaining the last auxiliary data, let us consider the families of Hermitian maps
\[
  h_\mbbP,\ h_{\mbbP^\prime},\ h_{\mscrV_\HB},\ h_{\mscrV_\HB}^\prime
\]
on
\[
  S_{\mbbP},\ S_{\mbbP^\prime}, S_{\mscrV_\HB},\ S_{\mscrV_\HB}^\prime.
\]
These correspond to $h_\HB$ constructed in \cref{sssec4:Fred op} and are necessary to construct families of Fredholm operators in combination with the Dirac operators. However, two things should be noted.
\begin{itemize}
  \item Consider constructing a family of Hermitian maps on $S_{\mscrV_\HB}$. A self-Hermitian map at $v_\HB \in \mscrV_\HB$ is not determined by $v_\HB$ alone. In fact, an element of
        \[
          \mscrV = S(E) \times \mscrV_\HB \times V_\R \times \R^+
        \]
        is needed. In other words, $h_{\mscrV_\HB}$ is defined on the vector bundle obtained by pulling back $S_{\mscrV_\HB}$ to the above space. This is similar to the situation where the map
        \[
          \map{\mcalF}{\mcalV}{\mcalW}
        \]
        in \cref{eq4:normalized F} was needed when defining $h_\HB$.
  \item These Hermitian maps depend on the choices of $s \in \mscrS$ and $g_{\mbbP^\prime} \in \mcalG_{\mbbP}$.
\end{itemize}

Fix an element $s \in \mscrS$ and $g_{\mbbP^\prime} \in \mcalG_{\mbbP}$. Then $\mbbP$ and $\mbbP^\prime$ have Riemannian metrics. $h_\mbbP$ and $h_{\mbbP^\prime}$ can be defined by the same formula as $h_\HB$.

The maps $h_{\mscrV_\HB}$ and $h_{\mscrV_\HB}^\prime$ are defined as follows.
Given a point of
\[
  \mscrV = S(E) \times \mscrV_\HB \times V_\R \times \R^+
\], by sending the point of $\mscrV_\HB$ by $\pi$ and then sending the obtained point of
\[
  S(E) \times \mbbP \times V_\R \times \R^+
\]
by $\mcalF$, a Hermitian map
\begin{align}\label{eq5:h base}
  \bar h_{\text{base}}
\end{align}
on $\Lambda_\C^\ast T\mbbP$ is constructed. Also, from the information in the fiber direction of $\mscrV_\HB$, Hermitian maps
\begin{align}\label{eq5:h fiber}
  \bar h_{\text{fiber}},\ \bar h_{\text{fiber}}^\prime
\end{align}
on
\[
  \Lambda_\C^\ast \mcalV_\HB^\prime \otimes_\C \Lambda_\C^\ast \mcalV_\HB^\prime,\ (\Lambda_\R^\ast \mcalV_\HB^\prime) \otimes_\R \C
\]
are constructed. These come from the family of Hermitian maps $h_{\text{pssho}}$ constructed in \cref{sec8:deform ssho} that are translation-invariant for $r > 1$.

Define the Hermitian maps on $S_{\mscrV_\HB}$ and $S_{\mscrV_\HB}^\prime$ by
\begin{align}\label{eq5:h fiber base}
  h_{\text{base}}        & = \varepsilon \otimes \bar h_{\text{base}} \otimes 1 \otimes 1,\ h_{\text{fiber}} = \varepsilon \otimes \varepsilon \otimes \varepsilon \otimes \bar h_{\text{fiber}},               \\
  h_{\text{base}}^\prime & = \varepsilon \otimes \bar h_{\text{base}} \otimes 1 \otimes 1,\ h_{\text{fiber}}^\prime = \varepsilon \otimes \varepsilon \otimes \varepsilon \otimes \bar h_{\text{fiber}}^\prime.
\end{align}
The maps $h_{\mscrV_\HB}$ and $h_{\mscrV_\HB}^\prime$ are defined by
\[
  h_{\mscrV_\HB} = h_{\text{base}} + h_{\text{fiber}},\ h_{\mscrV_\HB}^\prime = h_{\text{base}}^\prime + h_{\text{fiber}}^\prime.
\]

Using the vector bundles
\[
  S_E  = \underline{\C} \otimes TS(E),\
  S_\R = (\R^+ \times V_\R) \times (\Lambda^\ast_\R \R^+ \otimes_\R \C) \otimes (\Lambda^\ast_\R V_\R \otimes_\R \C)
\]
on $S(E)$ and $V_\R \times \R$, define
\begin{align*}
  \mbfS_\mbbP               & = S_E \boxtimes S_\mbbP \boxtimes S_\R,               \\
  \mbfS_{\mbbP^\prime}      & = S_E \boxtimes S_{\mbbP^\prime} \boxtimes S_\R,      \\
  \mbfS_{\mscrV_\HB}        & = S_E \boxtimes S_{\mscrV_\HB} \boxtimes S_\R,        \\
  \mbfS_{\mscrV_\HB}^\prime & = S_E \boxtimes S_{\mscrV_\HB}^\prime \boxtimes S_\R.
\end{align*}
From the
\[
  h_\mbbP,\ h_{\mbbP^\prime},\ h_{\mscrV_\HB},\ h_{\mscrV_\HB}^\prime
\]
defined above, families of Hermitian maps
\begin{align}\label{eq5:hermitian operator}
  \mbfh_\mbbP,\ \mbfh_{\mbbP^\prime},\ \mbfh_\mscrV,\ \mbfh_\mscrV^\prime
\end{align}
are constructed on each of the vector bundles
\begin{align}\label{eq5:"spin bdl"}
  \begin{split}
    \mbfS_\mbbP               & \to S(E) \times \mbbP \times V_\R \times \R^+,        \\
    \mbfS_{\mbbP^\prime}      & \to S(E) \times \mbbP^\prime \times V_\R \times \R^+, \\
    \mbfS_{\mscrV_\HB}        & \to S(E) \times \mscrV_\HB \times V_\R \times \R^+,   \\
    \mbfS_{\mscrV_\HB}^\prime & \to S(E) \times \mscrV_\HB \times V_\R \times \R^+.
  \end{split}
\end{align}
$S_E$ and $S_\R$ are the vector bundles defined in \cref{eq4:S_E} and \cref{eq4:S_R}, and they have families of Hermitian maps $h_E$, $h_\R$ defined in \cref{eq4:h}. Using these,
\[
  \mbfh_\mbbP = h_E \otimes 1 \otimes 1 + \varepsilon \otimes h_\mbbP \otimes 1 + \varepsilon \otimes \varepsilon \otimes h_\R
\]
is defined. The case of $\mbfh_{\mbbP^\prime}$ is similar. For $\mbfh_\mscrV$, it decomposes into the sum of
\[
  \mbfh_{\text{base}} = h_E \otimes 1 \otimes 1 + \varepsilon \otimes h_{\text{base}} \otimes 1 + \varepsilon \otimes \varepsilon \otimes h_\R,\ \mbfh_{\text{fiber}} = \varepsilon \otimes h_{\text{fiber}} \otimes 1.
\]
Also, $\mbfh_\mscrV^\prime$ decomposes into the sum of
\begin{align}\label{eq5:mbfh base fiber}
  \mbfh_{\text{base}}^\prime = h_E \otimes 1 \otimes 1 + \varepsilon \otimes h_{\text{base}}^\prime \otimes 1 + \varepsilon \otimes \varepsilon \otimes h_\R,\ \mbfh_{\text{fiber}}^\prime = \varepsilon \otimes h_{\text{fiber}}^\prime \otimes 1.
\end{align}

Before returning to the auxiliary data, we prove a lemma.

\begin{lemma}\label{lem5:iso of spinors along fiber}
  Fix $s \in \mscrS$. Then there is a $\Z/4$-equivariant isomorphism
  \[
    S_{\mscrV_\HB} \cong S_{\mscrV_\HB}^\prime
  \]
  canonically constructed from $s$. This preserves the Clifford action on $B_{1/2}(\mscrV_\HB)$. Also, under this isomorphism, $\mbfh_{\mscrV_\HB}$ and $\mbfh_{\mscrV_\HB}^\prime$ coincide on $B_{1/2}(\mscrV_\HB)$.
\end{lemma}

\begin{proof}
  From $s$, the isomorphism
  \[
    \Lambda_\C^\ast T\mscrV_\HB \cong \pi^\ast \Lambda_\C^\ast T\mbbP \otimes \Lambda_\C^\ast T_\mbbP \mscrV_\HB
  \]
  is constructed. As a vector bundle over $\mscrV_\HB$, we have $T_\mbbP \mscrV_\HB \cong \mcalV_\HB^\prime$. From \cref{prop7:ssho equality}, we have
  \[
    \Lambda_\C^\ast \mcalV_\HB^\prime \otimes_\C \Lambda_\C^\ast \mcalV_\HB^\prime \cong \Lambda_\R^\ast \mcalV_\HB^\prime \otimes_\R \C.
  \]
  The discussions so far show the existence of the canonical isomorphism between $S_{\mscrV_\HB}$ and $S_{\mscrV_\HB}^\prime$. From the construction of the Riemannian metric on $\mscrV_\HB$ and \cref{prop7:ssho equality}, on $B_{1/2}(\mscrV_\HB)$, they are isomorphic, including the correspondence of Clifford actions and families of Hermitian maps.
\end{proof}

What remains is to specify the auxiliary data that determines the structure of Clifford bundles on the vector bundles in \cref{eq5:spin bdl on H comp}.

Fix $s \in \mscrS$ and $g_{\mbbP^\prime} \in \mcalG(\mbbP^\prime, s)$. At this point, the structures of Clifford bundles
\[
  (S_\mbbP, c_\mbbP, \nabla_\mbbP),\ (S_{\mbbP^\prime}, c_{\mbbP^\prime}, \nabla_{\mbbP^\prime})
\]
on $S_\mbbP$ and $S_\mbbP^\prime$ are determined. (These are discussed in \cref{sssec4:Fred op}.) So it suffices to specify the structures of Clifford bundles on $S_{\mscrV_\HB}$, $S_{\mscrV_\HB}^\prime$.
The Clifford actions have already been constructed, so the only issue is the choice of connections. The last auxiliary data is these connections themselves. Below, we consider two classes of these connections. One only imposes the minimum conditions necessary for obtaining a family of Fredholm operators, and the other imposes stronger technical conditions.

First, we define the class of connections with looser conditions.

\begin{definition}\label{def5:connection}
  Fix $s \in \mscrS$. Let $\mcalC(s)$ (resp. $\mcalC^\prime(s)$) denote the space of all connections $\nabla_{\mscrV_\HB}$ (resp. $\nabla_{\mscrV_\HB}^\prime$) on $S_{\mscrV_\HB}$ (resp. $S_{\mscrV_\HB}^\prime$) that satisfy the following conditions:
  \begin{itemize}
    \item $\nabla_\mscrV$ (resp. $\nabla_{\mscrV_\HB}^\prime$) is translation-invariant over $S(\mscrV_\HB) \times [1, \infty)$.
    \item $(S_{\mscrV_\HB}, c_{\mscrV_\HB}, \nabla_{\mscrV_\HB})$ (resp. $(S_{\mscrV_\HB}^\prime, c_{\mscrV_\HB}^\prime, \nabla_{\mscrV_\HB}^\prime)$) is a Clifford bundle.
    \item The above Clifford bundle is compatible with the $\Z/4$-action $\tau_\mscrV$ (resp. $\tau_\mscrV^\prime$) on $S_{\mscrV_\HB}$ (resp. $S_{\mscrV_\HB}^\prime$) in the sense of \cref{prop4:cmptbl h-c-t}.
  \end{itemize}
\end{definition}

Fix $\nabla_\mscrV \in \mcalC(s)$. We confirm that a family of Fredholm operators is obtained from this. Let $(\mbfS_\mscrV, c_\mscrV, \nabla_\mscrV)$ denote the Clifford bundle obtained as the tensor product of $(S_{\mscrV_\HB}, c_{\mscrV_\HB}, \nabla_{\mscrV_\HB})$ and the Clifford bundles
\[
  (S_E, c_E, \nabla_E),\ (S_\R, c_\R, \nabla_\R)
\]
constructed in \cref{sssec4:Fred op}. Furthermore, let $\mcalD_\mscrV$ denote the Dirac operator obtained from this. Also, consider the same construction for $\mcalC^\prime(s)$. The following lemma is the analog of \cref{prop4:t-delta pt} for $S_\mscrV$.

\begin{lemma}\label{lem5:t-delta prime}
  Fix $s \in \mscrS$ and $\nabla_\mscrV \in \mcalC(s)$. Define a compact subset $K$ of $\mscrV$ by
  \[
    K = S(E) \times B(\mscrV_\HB) \times B(V_\R) \times B(\R^+).
  \]
  Then the set $T$ of positive real numbers $t$ satisfying the following property is nonempty and contractible: there exists $\bar \lambda > 0$ such that for any $s \geq t$,
  \[
    s(\mcalD_\mscrV \mbfh_\mscrV + \mbfh_\mscrV \mcalD_\mscrV) + s^2 \mbfh_\mscrV^2 > \bar \lambda
  \]
  holds outside $K$. Also, when $t \in T$, the family of operators
  \begin{align}
    \mcalD_\mscrV + t\mbfh_{\mscrV}
  \end{align}
  from the $L^2_1$ space to the $L^2$ space is a family of Fredholm operators parametrized by $S(E)$.

  The same statement holds for $\mcalC^\prime(s)$.
\end{lemma}

Finally, we consider a class of connections with stronger conditions. The condition is about the Dirac operator $\mcalD_\mscrV^\prime$ associated with $\nabla_\mscrV^\prime \in \mcalC^\prime(s)$.

We decompose $\mcalD_\mscrV^\prime$ into two differential operators. On the vector bundle
\[
  \Lambda_\R^\ast \mcalV_\HB^\prime \otimes \C \to \mscrV_\HB,
\]
there is a first-order differential operator $\tmcalD_{\text{pssho}}$. This acts on sections of the above vector bundle by applying the $O(\mscrV_\HB)$-equivariant differential operator $\mcalD_{\text{pssho}}$ constructed in \cref{prop8:ker of ppsho} fiberwise on $\mscrV_\HB$. Using this, define the differential operator constructed on
\[
  \mbfS_\mscrV^\prime = S_E \boxtimes S_{\mscrV_\HB}^\prime \boxtimes S_\R
\]
by
\[
  \mcalD_{\text{fiber}}^\prime = \varepsilon \otimes \tmcalD_{\text{pssho}} \otimes 1.
\]
And let
\[
  \mcalD_{\text{base}}^\prime = \mcalD_\mscrV^\prime - \mcalD_{\text{fiber}}^\prime.
\]
The operator $\mcalD_{\text{fiber}}^\prime$ does not depend on the choice of $\nabla_{\mscrV_\HB}^\prime$, but $\mcalD_{\text{base}}^\prime$ depends on it.

\begin{definition}\label{def5:strict connection}
  Fix $s \in \mscrS$. Let $\mcalC_{\text{strict}}^\prime(s)$ denote the space of elements $\nabla_\mscrV^\prime$ of $\mcalC^\prime(s)$ satisfying the following conditions:
  \begin{itemize}
    \item The following anti-commutativity holds:
          \[
            \mcalD_{\text{base}}^\prime \mcalD_{\text{fiber}}^\prime = -\mcalD_{\text{fiber}}^\prime \mcalD_{\text{base}}^\prime,\ \mcalD_{\text{base}}^\prime \mbfh_{\text{fiber}}^\prime = -\mbfh_{\text{fiber}}^\prime \mcalD_{\text{base}}^\prime,
          \]
          where $\mbfh_{\text{fiber}}^\prime$ are the Hermitian maps defined in \cref{eq5:mbfh base fiber}.
    \item Let $\map{\tilde \pi}{\mscrV}{S(E) \times \mbbP \times V_\R \times \R^+}$ be the natural projection. Take any $a \in \Gamma(\mbfS_\mbbP)$. Also, take any $O(V_\HB^\prime)$-equivariant section $\bar b \in \Gamma(\Lambda_\R^\ast V_\HB^\prime \otimes \C)$. Let $b \in \Gamma(\Lambda_\R^\ast \mcalV_\HB^\prime \otimes \C)$ denote the section obtained by placing $\bar b$ along the fibers of $\Lambda_\R^\ast \mcalV_\HB^\prime \otimes \C$. Then
          \[
            \mcalD_{\text{base}}^\prime (\tilde \pi^\ast a \otimes b) = \tilde \pi^\ast (\mcalD_\mscrF a) \otimes b
          \]
          holds, where $\mcalD_\mscrF$ denotes the family of Dirac operators associated with the Clifford bundle on $\mbfS_\mbbP$.
  \end{itemize}
\end{definition}

Note that
\[
  \mcalD_{\text{fiber}}^\prime \mbfh_{\text{base}}^\prime     = -\mbfh_{\text{base}}^\prime \mcalD_{\text{fiber}}^\prime,\ \mbfh_{\text{base}}^\prime \mbfh_{\text{fiber}}^\prime  = -\mbfh_{\text{fiber}}^\prime \mbfh_{\text{base}}^\prime,
\]
always holds for any element of $\mcalC^\prime(s)$.

\begin{lemma}
  The spaces $\mcalC(s)$, $\mcalC^\prime(s)$, and $\mcalC_{\text{strict}}^\prime(s)$ are nonempty and even contractible for any $s \in \mscrS$.
\end{lemma}

\begin{proof}
  We only prove that $\mcalC_{\text{strict}}^\prime(s)$ is nonempty. Let $\{U_\alpha\}$ be an open cover of $S(V_\HB)/U(1)$ consisting of open sets that admit a local trivialization of $\mscrV_\HB$. Let $\{\rho_\alpha^2\}$ be a partition of unity subordinate to $\{U_\alpha\}$. On each $\restr{S_{\mscrV_\HB}^\prime}{U_\alpha}$, we can take a connection $\nabla_\alpha$ that satisfies all conditions except for $\Z/4$-invariance. This can be constructed as follows. Taking a local trivialization
  \[
    \restr{\mscrV_\HB}{U_\alpha} \cong U_\alpha \times V_\HB,
  \]
  we have a natural identification
  \[
    \restr{S_{\mscrV_\HB}^\prime}{U_\alpha} \cong S_\mbbP \boxtimes (V_\HB \times (\Lambda_\R^\ast \mcalV_\HB^\prime \otimes \C)).
  \]
  Using this, it suffices to take the tensor of the connections on both sides as $\nabla_\alpha$.

  Finally, taking the average of $\sum_\alpha \rho_\alpha \nabla_\alpha \rho_\alpha$ under the $\Z/4$-action gives an element of $\mcalC_{\text{strict}}^\prime(s)$.
\end{proof}

Under the above preparation, we prove \cref{prop5:spin mscrV} and \cref{prop5:can iso for vari spin}.

\begin{proof}[Proof of \cref{prop5:spin mscrV}]
  Take $s \in \mscrS$. The two vector bundles $\mbfS_\mscrV$ and $\mbfS_\mscrV^\prime$ on
  \[
    \mscrV = S(E) \times \mscrV_\HB \times V_\R \times \R^+
  \]
  are defined by
  \[
    \mbfS_\mscrV = S_E \boxtimes S_{\mscrV_\HB} \boxtimes S_\R,\ \mbfS_\mscrV^\prime = S_E \boxtimes S_{\mscrV_\HB}^\prime \boxtimes S_\R.
  \]
  On $\mbfS_\mscrV$ and $\mbfS_\mscrV^\prime$, by taking auxiliary data $\nabla_\mscrV \in \mcalC(s)$ and $\nabla_\mscrV^\prime \in \mcalC^\prime(s)$ defined in \cref{def5:connection}, the Clifford bundle structures are constructed. Also, in \cref{eq5:hermitian operator}, families of Hermitian maps $\mbfh_\mscrV$, $\mbfh_\mscrV^\prime$ were constructed on $\mbfS_\mscrV$ and $\mbfS_\mscrV^\prime$, respectively. From \cref{lem5:t-delta prime}, when $t$ is sufficiently large,
  \[
    \mcalD_\mscrV + t \mbfh_\mscrV,\ \mcalD_\mscrV^\prime + t \mbfh_\mscrV^\prime
  \]
  are families of Fredholm operators. Therefore, their determinant line bundles
  \[
    \det(\mcalD_\mscrV + t \mbfh_\mscrV),\ \det(\mcalD_\mscrV^\prime + t \mbfh_\mscrV^\prime)
  \]
  can be constructed. These bundles, together with the anti-linear maps
  \begin{align*}
    \map{\tiota_\mscrV        & }{\det(\mcalD_\mscrV + t \mbfh_\mscrV)}{\det(\mcalD_\mscrV + t \mbfh_\mscrV)},                            \\
    \map{\tiota_\mscrV^\prime & }{\det(\mcalD_\mscrV^\prime + t \mbfh_\mscrV^\prime)}{\det(\mcalD_\mscrV^\prime + t \mbfh_\mscrV^\prime)}
  \end{align*}
  constructed from these and $\tau_\mscrV$, $\tau_\mscrV^\prime$ are triples for spin structure on $E$. Therefore, from \cref{prop3:spin from triple}, spin structures $\mfrakt_\mscrV$, $\mfrakt_\mscrV^\prime$ on $E$ are constructed. Since $\mcalC^\prime(s)$ and $\mscrS$ are contractible, these are independent of the choices of $s$ and $\nabla_\mscrV^\prime$ from \cref{prop3:simp-conn}.
\end{proof}

\begin{proof}[Proof of \cref{prop5:can iso for vari spin}]
  We construct the isomorphism (1): $\mfrakt_\mscrF \cong \mfrakt_\mscrV^\prime$. Take $s \in \mscrS$ and $\nabla_\mscrV^\prime \in \mcalC_{\text{strict}}(s)$. We show that when $t$ is taken sufficiently large, there is a canonical $\Z/4$-equivariant isomorphism between
  \[
    \Ker (\mcalD_\mbbP + t \mbfh_\mbbP),\ \Ker(\mcalD_\mscrV^\prime + t \mbfh_\mscrV^\prime).
  \]
  From the commutativity assumed in \cref{def5:strict connection}, the kernel of $\mcalD_\mscrV^\prime + t \mbfh_\mscrV^\prime$ is the intersection of the kernels of $\mcalD_{\text{base}} + t \mbfh_{\text{base}}$ and $\mcalD_{\text{fiber}} + t \mbfh_{\text{fiber}}$.
  If $t$ is sufficiently large, the kernel of $\mcalD_{\text{fiber}} + t \mbfh_{\text{fiber}}$ takes the form of a tensor product of
  \begin{itemize}
    \item a section $a$ of $\mbfS_\mbbP$,
    \item a section $b_{\text{pssho}} \in \Gamma(\Lambda_\R^\ast V_\HB \otimes_\R \C)$ obtained by placing the standard solution $\bar b_{\text{pssho}}$ of the pseudo-super-symmetric harmonic oscillator on $V_\HB^\prime \times (\Lambda_\R^\ast V_\HB \otimes_\R \C)$ (See \cref{def8:std sol of pssho}.) along the fibers.
  \end{itemize}
  Since $\bar b_{\text{pssho}}$ is $O(V_\HB^\prime)$-invariant, from the assumption in \cref{def5:strict connection},
  \[
    (\mcalD_{\text{base}} + t \mbfh_{\text{base}})(\tilde \pi^\ast a \otimes b_{\text{pssho}}) = \tilde \pi^\ast (\mcalD_\mbbP + \mbfh_\mbbP) a \otimes b_{\text{pssho}}
  \]
  holds. Therefore, if this is 0, then $a$ is an element of $\Ker (\mcalD_\mbbP + t\mbfh_\mbbP)$. This constructs the desired $\Z/4$-equivariant isomorphism. Since a $\Z/4$-equivariant isomorphism can also be constructed for their determinant line bundles, we obtain an isomorphism between $\mfrakt_\mscrF$ and $\mfrakt_\mscrV^\prime$. Since $\mcalC_{\text{strict}}^\prime(s)$ and $\mscrS$ are contractible, it is independent of the choices of $s$ and $\nabla_\mscrV^\prime$.

  We construct the isomorphism (2): $\mfrakt_{\mscrV_\HB}^\prime \cong \mfrakt_{\mscrV_\HB}$. Take $s \in \mscrS$. Take any $\nabla_{\mscrV_\HB}^\prime \in \mcalC_{\text{strict}}^\prime(s)$. Take $\nabla_{\mscrV_\HB} \in \mcalC(s)$ such that it coincides with $\nabla_{\mscrV_\HB}$ on $B_{1/2}(\mscrV_\HB)$ under the isomorphism
  \begin{align}\label{eq5:Cl bdl iso mscrV_HB}
    S_{\mscrV_\HB} \cong S_{\mscrV_\HB}^\prime
  \end{align}
  in \cref{lem5:iso of spinors along fiber}. From the claim of \cref{lem5:str from splitting}(3) about $h$,
  \begin{align*}
     & \mcalD_\mscrV = \mcalD_\mscrV^\prime,\ \mbfh_{\text{base}} = \mbfh_{\text{base}}^\prime,\ \mbfh_{\text{fiber}} = \mbfh_{\text{fiber}}^\prime
  \end{align*}
  holds under \cref{eq5:Cl bdl iso mscrV_HB} on
  \[
    S(E) \times B_{1/2}(\mscrV_\HB) \times V_\R \times \R^+.
  \]

  Using two positive real numbers $t$ and $u$, we consider families of differential operators represented by
  \[
    \mcalD_\mscrV + t \mbfh_{\text{base}} + u \mbfh_{\text{fiber}},\ \mcalD_\mscrV^\prime + t \mbfh_{\text{base}}^\prime + u \mbfh_{\text{fiber}}^\prime.
  \]
  If $t$ and $u$ are taken sufficiently large, these are families of $\Z/4$-equivariant Fredholm operators parametrized by $S(E)$. The spin structures constructed from these are canonically isomorphic to $\mfrakt_\mscrV$ and $\mfrakt_\mscrV^\prime$. Therefore, it suffices to construct a canonical isomorphism based on these operators. The strategy is to use the Witten deformation technique. (For a foundational formulation of the Witten deformation, see, for example, Zhang\cite{Zhang-Witten-deformation-2001}. In our situation, a more advanced and detailed analysis is essential. The readers are also referred to Furuta\cite{Furuta2007Index-Theorem} and Miyazawa\cite{miyazawa2021localization}.) Roughly speaking, when $t$ and $u$ are sufficiently large, the vector spaces consisting of eigenvectors corresponding to eigenvalues below a certain positive real number $\lambda$ of each of
  \[
    (\mcalD_\mscrV + t \mbfh_{\text{base}} + u \mbfh_{\text{fiber}})^2,\ (\mcalD_\mscrV^\prime + t \mbfh_{\text{base}}^\prime + u \mbfh_{\text{fiber}}^\prime)^2
  \]
  are isomorphic. However, there are several technical hurdles, such as that $\lambda$ must not be an eigenvalue. Due to that difficulty, we argue as follows.

  Fix $\bar \lambda > 0$. From \cref{prop4:t-delta pt}, for all sufficiently large $t$, outside of $S(E) \times \mbbP \times B(V_\R) \times B(\R^+)$,
  \[
    t(\mcalD_{\mbbP} \mbfh_{\mbbP} + \mbfh_{\mbbP}\mcalD_{\mbbP}) + t^2 \mbfh_{\mbbP}^2 > \bar \lambda
  \]
  holds. Fix such a $t$, and let
  \[
    \mcalD_{\mscrV_\HB, t} = \mcalD_{\mscrV_\HB} + t \mbfh_{\text{base}},\ \mcalD_{\mscrV_\HB, t}^\prime = \mcalD_{\mscrV_\HB}^\prime + t \mbfh_{\text{base}}^\prime.
  \]

  If $u$ is taken sufficiently large compared to $t$, outside of $S(E) \times B(\mscrV_{\HB}) \times V_\R \times \R^+$,
  \begin{align*}
    u(\mcalD_{\mscrV_\HB, t} \mbfh_{\text{fiber}} + \mbfh_{\text{fiber}} \mcalD_{\mscrV_\HB, t}) + u^2 \mbfh_{\text{fiber}}^2 > \bar \lambda, \\
    u(\mcalD_{\mscrV_\HB, t} \mbfh_{\text{fiber}}^\prime + \mbfh_{\text{fiber}}^\prime \mcalD_{\mscrV_\HB, t}) + u^2 \mbfh_{\text{fiber}}^{\prime 2} > \bar \lambda
  \end{align*}
  hold. Also, if $u$ is sufficiently large, the minimum nonzero eigenvalue of
  \[
    (\mcalD_{\text{pssho}} + u h_{\text{pssho}})^2
  \]
  is greater than $\bar \lambda$. Here, $\mcalD_{\text{pssho}}$ and $h_{\text{pssho}}$ are the operators explained in \cref{prop8:ker of ppsho}. Then, it can be seen that the spectrum of
  \[
    (\mcalD_\mscrV^\prime + t \mbfh_{\text{base}}^\prime + u \mbfh_{\text{fiber}}^\prime)^2
  \]
  between $[0, \bar \lambda]$ is discrete and does not depend on $u$.

  Fix any $e \in S(E)$. Take $\lambda \in (0, \bar \lambda)$ so that it is not an eigenvalue of
  \[
    (\mcalD_{\mbbP} + t\mbfh_{\mbbP})^2
  \]
  for the restriction to $e$. If $u$ is sufficiently large, $\lambda$ is not an eigenvalue of
  \[
    (\mcalD_\mscrV^\prime + t \mbfh_{\text{base}}^\prime + u \mbfh_{\text{fiber}}^\prime)^2
  \]
  either. Then, by using the Witten deformation technique, it is shown that there exists a $\Z/2$-invariant open neighborhood $U_{\lambda, e}$ of $e \in S(E)$ such that for sufficiently large $u$, $\lambda$ is not the eigenvalue of $(\mcalD_\mscrV + t \mbfh_{\text{base}} + u \mbfh_{\text{fiber}})^2$ and a $\Z/4$-equivariant isomorphism
  \[
    \map{\Phi_{\lambda, e, u}}{\restr{E_{\leq \lambda}((\mcalD_\mscrV^\prime + t \mbfh_{\text{base}}^\prime + u \mbfh_{\text{fiber}}^\prime)^2)}{U_{\lambda, e}}}{\restr{E_{\leq \lambda}((\mcalD_\mscrV + t \mbfh_{\text{base}} + u \mbfh_{\text{fiber}})^2)}{U_{\lambda, e}}}
  \]
  is constructed, where
  \[
    \restr{E_{\leq \lambda}((\mcalD_\mscrV^\prime + t \mbfh_{\text{base}}^\prime + u \mbfh_{\text{fiber}}^\prime)^2)}{U_{\lambda, e}}
  \]
  denotes the vector bundle consisting of sections with eigenvalues less than or equal to $\bar \lambda$. This isomorphism is constructed by composing the multiplication by some cut-off function $\rho$ and the $L^2$-projection. Precisely, $\map{\rho}{S_{\mscrV_{\HB}}}{[0, 1]}$ is a $\Z/2$-invariant smooth function which is identically 1 on $B_{1/8}(S_{\mscrV_{\HB}})$ and 0 outside $B_{1/4}(S_{\mscrV_{\HB}})$. By taking the determinant, a $\Z/4$-equivariant isomorphism
  \[
    \map{\det \Phi_{\lambda, e, u}}{\restr{\det(\mcalD_\mscrV^\prime + t \mbfh_{\text{base}}^\prime + u \mbfh_{\text{fiber}}^\prime)}{U_{\lambda, e}}}{\restr{\det (\mcalD_\mscrV + t \mbfh_{\text{base}} + u \mbfh_{\text{fiber}})}{U_{\lambda, e}}}
  \]
  is obtained for $u$ sufficiently large. The map $\det \Phi_{\lambda, e, u}$ is not a global isomorphism over $S(E)$. Therefore, we glue the isomorphisms with different $\lambda$'s, $e$'s and $u$'s by using $\Z/2$-invariant partition of unity. The following lemma is key to prove that the resulting map
  \[
    \det(\mcalD_\mscrV^\prime + t \mbfh_{\text{base}}^\prime + u \mbfh_{\text{fiber}}^\prime) \to \det (\mcalD_\mscrV + t \mbfh_{\text{base}} + u \mbfh_{\text{fiber}})
  \]
  is a $\Z/4$-invariant isomorphism for $u$ sufficiently large, and gives the desired isomorphism between $\mfrakt_{\mscrV}^\prime$ and $\mfrakt_{\mscrV}$. The proof of this lemma is an application of the Witten deformation technique.

  \begin{lemma}
    Let $\lambda^\prime \in (0, \bar \lambda)$, $e^\prime \in S(E)$ and $U_{\lambda^\prime, e^\prime}$ be other choices of the above data, and assume $U_{\lambda, e} \cap U_{\lambda^\prime, e^\prime}$ is nonempty. Take $e^\pprime \in U_{\lambda, e} \cap U_{\lambda^\prime, e^\prime}$ arbitrarily. Then for any $\varepsilon > 0$, there exists $u_0 > 0$ and an open neighborhood $V \subset U_{\lambda, e} \cap U_{\lambda^\prime, e^\prime}$ of $e^\pprime$ such that for any $u \geq u_0$
    \[
      \abs{1 - \det \Phi_{\lambda, e, u}^{-1} \det \Phi_{\lambda^\prime, e^\prime, u}} < \varepsilon
    \]
    on $V$. Here we regard
    \[
      \map{\det \Phi_{\lambda, e, u}^{-1} \det \Phi_{\lambda^\prime, e^\prime, u}}{\det(\mcalD_\mscrV^\prime + t \mbfh_{\text{base}}^\prime + u \mbfh_{\text{fiber}}^\prime)}{\det(\mcalD_\mscrV^\prime + t \mbfh_{\text{base}}^\prime + u \mbfh_{\text{fiber}}^\prime)}
    \]
    as a function from $U_{\lambda, e} \cap U_{\lambda^\prime, e^\prime}$ to $\C$.
  \end{lemma}

  \begin{proof}
    Without loss of generality we can assume $\lambda \leq \lambda^\prime$. In the following, we drop the restriction to $U_{\lambda, e} \cap U_{\lambda^\prime, e^\prime}$ from the notation. Let
    \[
      E_\lambda^{\lambda^\prime}((\mcalD_\mscrV^\prime + t \mbfh_{\text{base}}^\prime + u \mbfh_{\text{fiber}}^\prime)^2),\ E_\lambda^{\lambda^\prime}((\mcalD_\mscrV + t \mbfh_{\text{base}} + u \mbfh_{\text{fiber}})^2)
    \]
    be the subbundle which is spanned by eigenvectors with eigenvalue in $[\lambda, \lambda^\prime]$. Then we have an orthogonal decomposition
    \begin{align*}
        & E_{\leq \lambda^\prime}((\mcalD_\mscrV^\prime + t \mbfh_{\text{base}}^\prime + u \mbfh_{\text{fiber}}^\prime)^2)                                                                                                                     \\
      = & E_{\leq \lambda}((\mcalD_\mscrV^\prime + t \mbfh_{\text{base}}^\prime + u \mbfh_{\text{fiber}}^\prime)^2) \oplus E_\lambda^{\lambda^\prime}((\mcalD_\mscrV^\prime + t \mbfh_{\text{base}}^\prime + u \mbfh_{\text{fiber}}^\prime)^2)
    \end{align*}
    and
    \begin{align*}
        & E_{\leq \lambda^\prime}((\mcalD_\mscrV + t \mbfh_{\text{base}} + u \mbfh_{\text{fiber}})^2)                                                                                                \\
      = & E_{\leq \lambda}((\mcalD_\mscrV + t \mbfh_{\text{base}} + u \mbfh_{\text{fiber}})^2) \oplus E_\lambda^{\lambda^\prime}((\mcalD_\mscrV + t \mbfh_{\text{base}} + u \mbfh_{\text{fiber}})^2)
    \end{align*}
    as well for $u$ sufficiently large.

    Define
    \[
      \map{\Phi_{\lambda, u}^{\lambda^\prime}}{E_{\leq \lambda^\prime}((\mcalD_\mscrV^\prime + t \mbfh_{\text{base}}^\prime + u \mbfh_{\text{fiber}}^\prime)^2)}{E_{\leq \lambda^\prime}((\mcalD_\mscrV^\prime + t \mbfh_{\text{base}}^\prime + u \mbfh_{\text{fiber}}^\prime)^2)}
    \]
    as the composition of multiplication by $\rho$ and the $L^2$ projection to $E_\lambda^{\lambda^\prime}((\mcalD_\mscrV + t \mbfh_{\text{base}} + u \mbfh_{\text{fiber}})^2)$, and set
    \[
      \Psi_{\lambda, u}^{\lambda^\prime} = \Phi_{\lambda, e, u} \oplus \Phi_{\lambda, u}^{\lambda^\prime}.
    \]
    For any $u$, we have the following estimates for the norms of the operators with respect to the $L^2$ norms:
    \[
      \norm{\Phi_{\lambda, e, u}} \leq 1,\ \norm{\Phi_{\lambda^\prime, e^\prime, u}} \leq 1,\ \norm{\Psi_{\lambda, u}^{\lambda^\prime}} \leq 1.
    \]
    By the Witten deformation technique, we also have the following estimates: for every $e^\pprime \in U_{\lambda, e} \cap U_{\lambda^\prime, e^\prime}$ and every $\varepsilon^\prime > 0$, there exists an open neighborhood $V$ of $e^\pprime$ and $u_0 > 0$ such that
    \begin{align*}
      \norm{\Phi_{\lambda, e, u}^{-1}} < 1 + \varepsilon^\prime,\ \norm{\Phi_{\lambda^\prime, e^\prime, u}^{-1}} < 1 + \varepsilon^\prime,\ \norm{(\Psi_{\lambda, u}^{\lambda^\prime})^{-1}} < 1 + \varepsilon^\prime,\ \norm{\Phi_{\lambda^\prime, e^\prime, u} - \Psi_{\lambda, u}^{\lambda^\prime}} < \varepsilon^\prime, \\
      \norm{\Phi_{\lambda, u}^{\lambda^\prime} \circ (\mcalD_\mscrV^\prime + t \mbfh_{\text{base}}^\prime + u \mbfh_{\text{fiber}}^\prime) - (\mcalD_\mscrV + t \mbfh_{\text{base}} + u \mbfh_{\text{fiber}}) \circ \Phi_{\lambda, u}^{\lambda^\prime}} < \varepsilon^\prime
    \end{align*}
    on $V$ for any $u \geq u_0$. From these estimates it can be shown that
    \begin{align*}
      \abs{1 - \det \Phi_{\lambda^\prime, e^\prime, u}^{-1} \det \Psi_{\lambda, u}^{\lambda^\prime}} & < C(\varepsilon^\prime), \\
      \abs{1 - \det \Phi_{\lambda, e, u}^{-1} \det \Psi_{\lambda, u}^{\lambda^\prime}}               & < C(\varepsilon^\prime)
    \end{align*}
    on $V$ for any $u \geq u_0$, where $C(\varepsilon^\prime)$ converges to 0 as $\varepsilon^\prime$ goes to 0. This finishes the proof.
  \end{proof}

  We construct the isomorphism (3): $\mfrakt_{\mscrV} \cong \mfrakt_{\mscrF^\prime}$. Fix $s \in \mscrS$ and $g^\prime \in \mcalG(\mbbP^\prime, s)$. Recall that by the embedding $\phi$ defined in \cref{eq5:embedding},
  \[
    B_{1/2}(\mscrV_\HB)
  \]
  is identified with an open subset of $\mbbP^\prime$. Thus the restrictions of the two vector bundles
  \[
    S_{\mbbP^\prime} \to \mbbP^\prime,\ S_{\mscrV} \to \mscrV_\HB
  \]
  to $B_{1/4}(\mscrV_\HB)$ are canonically isomorphic including the Clifford action. Note that the family of Hermitian maps $h_{\mbbP^\prime}$ defined on
  \[
    S_{\mbbP^\prime} = (\underline{\C} \oplus \mcalO(1)) \otimes \Lambda_\C^\ast T\mbbP^\prime \otimes \Lambda_\C^\ast \mcalW_\HB \otimes \Lambda_\C^\ast \mcalV_\HB^\prime
  \]
  has the following decomposition:
  \begin{align*}
    h_{\mbbP^\prime} & = h_{\mbbP^\prime}^{\text{base}} + h_{\mbbP^\prime}^{\text{fiber}}.
  \end{align*}
  Here, $h_{\mbbP^\prime}^{\text{base}}$ is the term coming from the family of Hermitian maps on $(\underline{\C} \oplus \mcalO(1)) \otimes \Lambda_\C^\ast T\mbbP^\prime \otimes \Lambda_\C^\ast \mcalW_\HB$, and $h_{\mbbP^\prime}^{\text{fiber}}$ is the term coming from the family of Hermitian maps on $\Lambda_\C^\ast \mcalV_\HB^\prime$. As they are, $h_{\mbbP^\prime}^{\text{base}}$ and $h_{\mscrV_\HB}^{\text{base}}$ are not equal on $B_{1/4}(\mscrV_\HB)$. However, by slightly modifying $h_{\mbbP^\prime}^{\text{base}}$, we can make them equal on $B_{1/4}(\mscrV_\HB)$. Let us denote the modified one by $\tilde h_{\mbbP^\prime}^{\text{base}}$. Similarly, by suitably taking a slight modification $\tilde h_{\mbbP^\prime}^{\text{fiber}}$ of $h_{\mbbP^\prime}^{\text{fiber}}$, we can make it equal to $h_{\mscrV_\HB}^{\text{fiber}}$ on $B_{1/4}(\mscrV_\HB)$. Furthermore, the spin structure on $E$ constructed using $\tilde h_{\mbbP^\prime}^{\text{base}} + \tilde h_{\mscrV_\HB}^{\text{fiber}}$ instead of $h_{\mbbP^\prime}$ is canonically isomorphic to $\mfrakt_{\mscrF^\prime}$. Therefore, it suffices to construct
  \[
    \mfrakt_{\mscrV_\HB} \cong \mfrakt_{\mscrF^\prime}
  \]
  using these modified families of Hermitian maps.

  Take $\nabla_{\mscrV_\HB} \in \mcalC(s)$ in the way specified in (2). Also, take a $\Z/4$-invariant connection $\nabla_{\mbbP^\prime}$ on $S_{\mbbP^\prime}$ such that
  \begin{itemize}
    \item it coincides with $\nabla_{\mscrV_\HB}$ on $B_{1/4}(\mscrV_\HB)$, and
    \item $(S_{\mbbP^\prime}, c_{\mbbP^\prime}, \nabla_{\mbbP^\prime})$ is a Clifford bundle.
  \end{itemize}
  Let $\mcalD$, $\mcalD_{\mbbP^\prime}$ denote the families of Dirac operators on $\mbfS_{\mscrV_\HB}$, $\mbfS_{\mbbP^\prime}$ constructed from each. Using positive real numbers $t$ and $u$, consider the two families of operators represented by
  \begin{align*}
     & \mcalD + t \mbfh_{\mscrV_\HB}^{\text{base}} + u \mbfh_{\mscrV_\HB}^{\text{fiber}},                                \\
     & \mcalD_{\mbbP^\prime} + t \tilde \mbfh_{\mbbP^\prime}^{\text{base}} + u \tilde \mbfh_{\mscrV_\HB}^{\text{fiber}}.
  \end{align*}
  By the same argument as in the proof of (2), when $t$ and $u$ are sufficiently large, both are families of Fredholm operators and moreover, a $\Z/4$-equivariant isomorphism between their determinant line bundles is constructed. Therefore, we have a canonical isomorphism between $\mfrakt_{\mscrV_\HB}$ and $\mfrakt_{\mscrF^\prime}$.
\end{proof}

\subsection{The case of taking direct sums of both types of vector bundles}\label{ssec5:both vbs}

In \cref{ssec5:real vb} and \cref{ssec5:quat vb}, we only considered the cases of taking direct sums of real vector bundles and quaternionic vector bundles separately. In this subsection, we consider the case of taking direct sums of these simultaneously. The following theorem can be proved by carefully observing the constructions made in this section.

\begin{theorem}\label{thm5:stab in order}
  Let $U$ be a topological space which is locally simply-connected and homeomorphic to an open set of some paracompact Hausdorff space. Let $E \to U$ be a rank 3 vector bundle over $U$ with an orientation and a metric. Let $\mscrF$ be a model of FDA for families of h$K3$. Let $V_\R^\prime$ be a $Pin(2)$-equivariant real vector bundle with a metric, and $V_\HB^\prime$ be a $Pin(2)$-equivariant quaternionic vector bundle with a metric. Let $\mscrF^\prime$ be a model of FDA for families of h$K3$ obtained by taking the direct sum of $\mscrF$ with $V_\R^\prime$ and $V_\HB^\prime$. The isomorphism
  \[
    \map{\Phi_{\mscrF^\prime \mscrF}}{\mfrakt_\mscrF}{\mfrakt_{\mscrF^\prime}}
  \]
  between the spin structures $\mfrakt_{\mscrF}$ and $\mfrakt_{\mscrF^\prime}$ on $E$ made from $\mscrF$ and $\mscrF^\prime$ does not depend on the order of taking the direct sums of $V_\R^\prime$ and $V_\HB^\prime$. Also, when $\mscrF^\pprime$ is a model of FDA for families of h$K3$ obtained by taking the direct sum of $\mscrF^\prime$ with a $Pin(2)$-equivariant real vector bundle $V_\R^\pprime$ with a metric and a $Pin(2)$-equivariant quaternionic vector bundle $V_\HB^\pprime$ with a metric,
  \[
    \Phi_{\mscrF^\pprime\mscrF^\prime} \circ \Phi_{\mscrF^\prime\mscrF} = \Phi_{\mscrF^\pprime\mscrF}
  \]
  holds.
\end{theorem}

Finally, we prove \cref{thm5:can iso}.

\begin{proof}[Proof of \cref{thm5:can iso}]
  Let $(R_1, U_1, \varepsilon_1, \lambda_1),\ (R_2, U_2, \varepsilon_2, \lambda_2) \in \tilde \mcalA$. (For the definition of $\tilde \mcalA$, see \cref{def4:proj}.) If we choose
  \[
    p_1 \in \mcalP(R_1, U_1, \varepsilon_1, \lambda_1),\ p_2 \in (R_2, U_2, \varepsilon_2, \lambda_2),
  \]
  then models $\mscrF_1$ and $\mscrF_2$ of FDA for families of h$K3$
  for $E$ are constructed from each of them. Let $\mfrakt_1$ and $\mfrakt_2$ denote the spin structures on $E$ constructed from these. We will prove that a canonical isomorphism is constructed between their restrictions to $U_1 \cap U_2$.

  Without loss of generality, we can assume that $R_1 \leq R_2$. From \cref{lem4:cov of B}(2), the following claim is verified.

  \begin{claim}\label{clm5:op cov}
    Define the set $\mcalB$ by
    \[
      \mcalB = \set{(U, \varepsilon, \lambda)}{U \subset U_1 \cap U_2\ \text{and}\ (R^\pprime, U, \varepsilon, \lambda) \in \tilde \mcalA\ \text{for all $R^\pprime \in [R_1, R_2]$}}.
    \]
    Then,
    \[
      \set{U \subset U_1 \cap U_2}{\exists \varepsilon > 0\ \exists \lambda > 0\ (U, \varepsilon, \lambda) \in \mcalB}
    \]
    is an open cover of $U_1 \cap U_2$.
  \end{claim}

  Take $(U, \varepsilon, \lambda) \in \mcalB$ with $\lambda$ larger than $\lambda_1$ and $\lambda_2$. First, we show that an isomorphism
  \[
    \restr{\mfrakt_1}{U} \to \restr{\mfrakt_2}{U}
  \]
  of spin structures on $\restr{E}{U}$ is constructed from $(U, \varepsilon, \lambda)$.

  For different choices of $R^\pprime$, there are canonical isomorphisms between the spin structures on $\restr{E}{U}$ constructed from $(R^\pprime, U, \varepsilon, \lambda)$, and these isomorphisms are compatible. Let $\tmfrakt_{(U, \varepsilon, \lambda)}$ denote that spin structure. We will construct two isomorphisms
  \[
    \restr{\mfrakt_1}{U} \cong \tmfrakt_{(U, \varepsilon, \lambda)},\ \restr{\mfrakt_2}{U} \cong \tmfrakt_{(U, \varepsilon, \lambda)}.
  \]

  Consider the case of $R^\pprime = R$. Let
  \[
    \mscrV^{\lambda_1\lambda},\ \mscrW^{\lambda_1\lambda}
  \]
  denote the vector bundles over $U$ consisting of all eigenvectors of $D^\ast D$, $D D^\ast$ whose eigenvalues lie in $[\lambda_1, \lambda]$. Take $p_1 \in \mcalP(R_1, U, \varepsilon_1, \lambda_1)$. Then we have
  \[
    p_1 + p_{\mscrW^{\lambda_1\lambda}} \in \mcalP(R_1, U, \varepsilon, \lambda),
  \]
  where $\map{p_{\mscrW^{\lambda_1\lambda}}}{\mscrW}{\mscrW^{\lambda_1\lambda}}$ denotes the $L^2$-orthogonal projection. The model of FDA for families of h$K3$ constructed from $(R_1, U, \varepsilon_1, \lambda_1)$ and $p_1 + p_{\mscrW^{\lambda_1\lambda}}$ coincides with the one obtained by taking the direct sum of the restriction to $U$ of the model of FDA for families of h$K3$ constructed from $(R_1, U, \varepsilon_1, \lambda_1)$ and $p_1$ with $\id_{V^{\lambda_1\lambda}}$. (We identified $V^{\lambda_1\lambda}$ and $W^{\lambda_1\lambda}$ through $D$.) Therefore, from \cref{thm5:stab in order}, an isomorphism is constructed between $\restr{\mfrakt_1}{U}$ and $\tmfrakt_{(U, \varepsilon, \lambda)}$. From the contractibility of $\mcalP(R_1, U, \varepsilon_1, \lambda_1)$, this isomorphism does not depend on the choice of $p_1$.

  Similarly, by considering the case of $R^\pprime = R_2$, an isomorphism between $\restr{\mfrakt_2}{U}$ and $\tmfrakt_{(U, \varepsilon, \lambda)}$ is constructed. By composing these isomorphisms, an isomorphism $\restr{\mfrakt_1}{U} \to \restr{\mfrakt_2}{U}$ is constructed.

  Let $(U^\prime, \varepsilon^\prime, \lambda^\prime)$ be another element of $\mcalB$. We prove that the isomorphisms between
  \[
    \restr{\mfrakt_1}{U \cap U^\prime},\ \restr{\mfrakt_2}{U \cap U^\prime}
  \]
  constructed from each coincide. If this is done, then from \cref{clm5:op cov}, a canonical isomorphism between $\restr{\mfrakt_1}{U_1 \cap U_2}$ and $\restr{\mfrakt_2}{U_1 \cap U_2}$ will have been constructed.

  Below, we discuss spin structures on $\restr{E}{U \cap U^\prime}$ without explicitly writing the restriction symbols. Let $\tmfrakt_{(U^\prime, \varepsilon^\prime, \lambda^\prime)}$ denote the spin structure constructed from $(U^\prime, \varepsilon^\prime, \lambda^\prime)$. From \cref{thm5:stab in order}, there is an isomorphism between $\tmfrakt_{(U, \varepsilon, \lambda)}$ and $\tmfrakt_{(U^\prime, \varepsilon^\prime, \lambda^\prime)}$. This isomorphism does not depend on the choice of $R^\pprime \in [R_1, R_2]$. Also, from the claim about compatibility in \cref{thm5:stab in order}, the three isomorphisms between
  \[
    \mfrakt_1,\ \tmfrakt_{(U, \varepsilon, \lambda)},\ \tmfrakt_{(U^\prime, \varepsilon^\prime, \lambda^\prime)}
  \]
  are compatible. Similarly, the three isomorphisms between
  \[
    \mfrakt_2,\ \tmfrakt_{(U, \varepsilon, \lambda)},\ \tmfrakt_{(U^\prime, \varepsilon^\prime, \lambda^\prime)}
  \]
  are also compatible. Therefore, the isomorphism between $\mfrakt_1$ and $\mfrakt_2$ is the same whether going through $\tmfrakt_{(U, \varepsilon, \lambda)}$ or $\tmfrakt_{(U^\prime, \varepsilon^\prime, \lambda^\prime)}$.

  Finally, the compatibility of isomorphisms between spin structures when taking a third element of $\tilde \mcalA$ follows from the compatibility stated in \cref{thm5:stab in order}.
\end{proof}

\section{Proof of the main theorems}\label{sec6:prf main thm}
In this section, we prove the main theorem stated in \cref{sec2:main thms}.

\begin{proof}[Proof of \cref{thm1:spin on H^+}]
  We show (1). Given a lift of the principal $\Diffplus$-bundle $\mcalE$ to a principal $\Diffspin$-bundle $\tmcalE \to B$, the goal is to construct a spin structure on $\hplus{\mbbX} \to B$. By specifying an element $(R, U, \varepsilon, \lambda)$ of $\tilde \mcalA$ defined in \cref{def4:proj} and an element $p$ of $\mcalP(R, U, \varepsilon, \lambda)$, a model of FDA for families of a h$K3$ on $\restr{\mbbX}{U} \to U$ is constructed by \cref{eg4:fda from sw}. By \cref{thm4:spin from fda}(1), from this model, a spin structure on
  \[
    \hplus{\restr{\mbbX}{U}} \to U
  \]
  is constructed. By \cref{thm4:indep proj}, this spin structure is determined independently of the choice of $p$. (More precisely, a canonical isomorphism can be taken between spin structures obtained from two different choices of $p$, and for three choices of $p$, these isomorphisms are compatible.)
  By gluing the spin structures constructed by the above method, we obtain a spin structure on the whole of $\restr{\hplus{\mbbX}}{U}$.

  We glue the spin structures constructed from each $(R, U, \varepsilon, \lambda)$. Let $(R, U, \varepsilon, \lambda)$ and $(R^\prime, U^\prime, \varepsilon^\prime, \lambda^\prime)$ be different elements of $\tilde \mcalA$. From each one, spin structures on
  \[
    \hplus{\restr{\mbbX}{U}},\ \hplus{\restr{\mbbX}{U^\prime}}
  \]
  are constructed. By \cref{thm5:can iso}, a canonical isomorphism is constructed between these spin structures on $U \cap U^\prime$, and for three choices of elements of $\mcalA$, that isomorphism is compatible. Therefore, the spin structures can be glued together to construct a spin structure of $\hplus{\mbbX}$ over the whole of $B$. This finishes the proof of (1).

  We prove (2). The goal is to show the functoriality of constructing a spin structure of $\hplus{\mbbX}$ from a lift of the principal $\Diffplus$-bundle $\mcalE$ to a principal $\Diffplus$-bundle. It suffices to consider the correspondence of morphisms. Assume that a morphism
  \[
    \map{(f, \Phi, \tilde \Phi)}{(B_0, \mcalE_0, \mbbX_0, g_0, \tmcalE_0)}{(B_1, \mcalE_1, \mbbX_1, g_1, \tmcalE_1)}
  \]
  of the category $\mcalC$ described in \cref{rem1:functor} is given.

  Let $\tilde \mcalA_0,\ \tilde \mcalA_1$ be $\tilde \mcalA$ of \cref{def4:proj} for
  \[
    (B_0, \mcalE_0, \mbbX_0, g_0, \tmcalE_0),\ (B_1, \mcalE_1, \mbbX_1, g_1, \tmcalE_1)
  \]
  respectively. Then for $(R, U, \varepsilon, \lambda) \in \tilde \mcalA_1$,
  \[
    (R, f^{-1}(U), \varepsilon, \lambda) \in \tilde \mcalA_0
  \]
  holds. The triple $(f, \Phi, \tilde \Phi)$ induces
  \begin{itemize}
    \item a map from the model of FDA made from $(R, f^{-1}(U), \varepsilon, \lambda)$ to the model of FDA made from $(R, U, \varepsilon, \lambda)$,
    \item a map between families of Fredholm operators constructed in \cref{prop4:t-delta fam} from the models of FDA,
    \item a map between determinant line bundles $\det (\mcalD + t\mbfh)$ of families of Fredholm operators constructed in \cref{ssec4:base not pt}.
  \end{itemize}
  Therefore, by using $(f, \Phi, \tilde \Phi)$, a $Spin(3)$-equivariant map from the spin structure on
  \[
    \hplus{\restr{\mbbX_0}{{f^{-1}(U)}}}
  \]
  to the spin structure on
  \[
    \hplus{\restr{\mbbX_1}{U}}
  \]
  is constructed. This map does not depend on the choice of $(R, U, \varepsilon, \lambda)$. It is clear from the construction that this correspondence is functorial.

  We prove (3). Assume that a lift $\tmcalE$ of $\mcalE$ to a principal $\Diffspin$-bundle is given. Then a spin structure $\mfrakt$ on $\hplus{\mbbX}$ is constructed. The goal is to show that the automorphism $\pm 1$ of $\tmcalE$ corresponds exactly to the automorphism $\pm 1$ of $\mfrakt$ through the functor constructed in (2). (The meaning of $\pm 1$ is explained in \cref{rem1:pm 1}.) Since $\mfrakt$ is constructed as the gluing of spin structures $\mfrakt_{(R, U, \varepsilon, \lambda)}$ on $\hplus{\restr{\mbbX}{U}}$ determined by specifying
  \[
    (R, U, \varepsilon, \lambda) \in \tilde \mcalA,
  \]
  it suffices to show that $\pm 1$ corresponds exactly between $\restr{\tmcalE}{U}$ and $\mfrakt_{(R, U, \varepsilon, \lambda)}$.

  The $\pm 1$ of $\restr{\tmcalE}{U}$ corresponds exactly to the $\pm 1$ of a model $\mscrF_{(R, U, \varepsilon, \lambda)}$ of FDA constructed from $(R, U, \varepsilon, \lambda)$. (see \cref{eg4:pm 1 for fda}.) By \cref{thm4:spin from fda}(3), the $\pm 1$ of $\mscrF_{(R, U, \varepsilon, \lambda)}$ corresponds exactly to the $\pm 1$ of $\mfrakt_{(R, U, \varepsilon, \lambda)}$.
\end{proof}

\begin{proof}[Proof of \cref{thm1-0:spin on TX+H^+}]
  Take an open cover $\{U_i\}$ of $\BDiffplus$ such that each
  \[
    \restr{\EDiffplus}{U_i}
  \]
  has a lift to the principal $\Diffspin$-bundle. For each $i$, fix a lift of $\restr{\EDiffplus}{U_i}$ to the principal $\Diffspin$-bundle
  \[
    \tmcalE_i \to U_i.
  \]
  By \cref{thm1:spin on H^+}(1), a spin structure on $\hplus{\restr{\Xunivplus}{U_i}}$ is constructed using $\tmcalE_i$. Denote its restriction to $U_i$ by $\mfrakt_i$. Also, $\tmcalE_i$ defines a spin structure $\mfraku_i$ on the tangent bundle along the fiber $T_{U_i} \mbbX_i$ of $\mbbX_i = \restr{\Xunivplus}{U_i}$. By taking the direct sum of $\mfraku_i$ and $\pi_{\mbbX_i}^\ast \mfrakt_i$, a spin structure
  \[
    \mfraku_i \oplus \pi_{\mbbX_i}^\ast \mfrakt_i
  \]
  on
  \[
    \restr{(T_{\BDiffplus} \Xunivplus \oplus \pi_{\mbbX_i}^\ast \hplus{\Xunivplus})}{U_i}
  \]
  is determined. We prove that for any $i$ and $j$, a canonical isomorphism is constructed between the two spin structures
  \[
    \restr{(\mfraku_i \oplus \pi_{\mbbX_i}^\ast \mfrakt_i)}{U_i \cap U_j},\
    \restr{(\mfraku_i \oplus \pi_{\mbbX_i}^\ast \mfrakt_i)}{U_i \cap U_j}
  \]
  on
  \[
    \restr{(T_{\BDiffplus} \Xunivplus \oplus \pi_{\mbbX_i}^\ast \hplus{\Xunivplus})}{U_i \cap U_j}.
  \]
  First, fix $b \in U_i \cap U_j$ and construct an isomorphism between the fibers at $b$. Fix a $\Diffspin$-equivariant map
  \[
    \map{\varphi_b^{ji}}{(\tmcalE_i)_b}{(\tmcalE_j)_b}
  \]
  that lifts the identity on $\EDiffplus_b$. The map $\varphi_b^{ji}$ induces an isomorphism
  \[
    \map{\chi_b^{ji}}{(\mfraku_i)_b}{(\mfraku_j)_b}
  \]
  of spin structures on $(T_{\BDiffplus} \Xunivplus)_b \to (\Xunivplus)_b$. Also, by \cref{thm1:spin on H^+}(2), $\varphi_b^{ji}$ induces an isomorphism
  \[
    \map{\psi_b^{ji}}{(\mfrakt_i)_b}{(\mfrakt_j)_b}
  \]
  of spin structures on $\hplus{\Xunivplus}_b \to \{b\}$. Hence, we have an isomorphism
  \[
    \map{\chi_b^{ji} \oplus \psi_b^{ji}}{(\mfraku_i \oplus \pi_{\mbbX_i}^\ast \mfrakt_i)_b}{(\mfraku_i \oplus \pi_{\mbbX_i}^\ast \mfrakt_i)_b}.
  \]
  By \cref{thm1:spin on H^+}(3), this is independent of the choice of $\varphi_b^{ji}$. Therefore, a map
  \[
    \map{\chi^{ji} \oplus \psi^{ji}}{\restr{(\mfraku_i \oplus \pi_{\mbbX_i}^\ast \mfrakt_i)}{U_i \cap U_j}}{\restr{(\mfraku_i \oplus \pi_{\mbbX_i}^\ast \mfrakt_i)}{U_i \cap U_j}}
  \]
  giving an isomorphism of spin structures for each fiber is canonically constructed. By taking a sufficiently small open neighborhood $V$ of each $b \in U_i \cap U_j$, $\varphi^{ji}$ can be taken continuously on $V$, so this is continuous. Also, by \cref{thm1:spin on H^+}(3),
  \[
    \{\chi^{ji} \oplus \psi^{ji}\}
  \]
  satisfies the cocycle condition. Therefore, these spin structures are glued together to give a spin structure on
  \[
    T_{\BDiffplus} \Xunivplus \oplus \pi_{\Xunivplus}^\ast \hplus{\Xunivplus}.
  \]
  This finishes the proof.
\end{proof}

\begin{proof}[Proof of \cref{thm1:alpha = w_2}]
  By \cref{thm1-0:spin on TX+H^+}, we have
  \[
    w_2(T_{\BDiffplus} \Xunivplus) = \pi_{\Xunivplus}^\ast w_2(\hplus{\Xunivplus}).
  \]
  This is equivalent to
  \[
    \alpha(\Xunivplus, \mfraks) = w_2(\hplus{\Xunivplus})
  \]
  by \cref{rem1:char of alpha}.
\end{proof}

\begin{remark}\label{rem6:gerbe}
  We recapture the proof of \cref{thm1:spin on TX+H^+} using the concept of $O(1)$-gerbes. The goal is to show that \cref{thm1:spin on TX+H^+} can be formulated as the existence of a canonical isomorphism between two $O(1)$-gerbes. Here, we use the formulation of gerbes by Hitchin\cite{Hitchin-Lag-submfd-2001} and Chatterjee\cite{Chatterjee-gerbe}. An $O(1)$-gerbe on a topological space $B$ consists of
  \begin{itemize}
    \item an open cover $\{U_i\}_{i \in I}$ of $B$,
    \item a principal $O(1)$-bundle $P_{ij}$ on $U_{ij} = U_i \cap U_j$,
    \item a section $s_{ijk}$ of $(\delta P)_{ijk} = P_{jk} \otimes P_{ik}^{-1} \otimes P_{ij}$ on $U_{ijk} = U_i \cap U_j \cap U_k$
  \end{itemize}
  satisfying
  \[
    (\delta s)_{ijkl} = s_{jkl} \otimes s_{ikl}^{-1} \otimes s_{ijl} \otimes s_{ijk}^{-1} = 1
  \]
  on $U_{ijkl} = U_i \cap U_j \cap U_k \cap U_l$. For example, given a central extension
  \[
    1 \to \Z/2 \to \tilde G \to G \to 1
  \]
  of topological groups and a principal $G$-bundle $R \to B$, an $O(1)$-gerbe can be constructed as follows: for simplicity, assume that $B$ has a good cover $\{U_i\}$ and that $R$ has a lift to a principal $\tilde G$-bundle over each $U_i$. For each $U_i$, fix a lift $\tilde R_i$ of $\restr{R}{U_i}$ to a principal $\tilde G$-bundle. Define a principal $O(1)$-bundle $P_{ij}$ by
  \[
    P_{ij} = \coprod_{b \in U_{ij}} \Iso_{\tilde G}(\tilde R_{i, b}, \tilde R_{j, b}).
  \]
  Furthermore, fix $u_{ij} \in \Gamma(P_{ij})$ and define $s_{ijk} \in \Gamma(P_{ijk})$ by
  \[
    s_{ijk} = u_{jk} \otimes u_{ik}^{-1} \otimes u_{ij}.
  \]
  These data define an $O(1)$-gerbe.

  In the above proof, an $O(1)$-gerbe
  \[
    \mcalG_{\mcalE} = (\{U_i\}, \{P_{ij}\}, \{s_{ijk}\})
  \]
  is constructed from the central extension
  \[
    1 \to \Z/2 \to \Diffspin \to \Diffplus \to 1
  \]
  and the principal $\Diffplus$-bundle
  \[
    \EDiffplus \to \BDiffplus.
  \]
  Similarly, we can construct another $O(1)$-gerbe $\mcalG_{\hplus{\mbbX}}$ for $\hplus{\mbbX}$. However, using the auxiliary data employed in the construction of $\mcalG_\mcalE$, a special construction can be chosen. \cref{thm1:spin on H^+}(1) stated that a spin structure on $\hplus{\restr{\Xunivplus}{U_i}}$ can be canonically constructed from a lift of $\restr{\EDiffplus}{U_i}$ to a principal $\Diffspin$-bundle. Also, the functoriality mentioned in \cref{thm1:spin on H^+}(2) implies that a map between two lifts of $\restr{\EDiffplus}{U_{ij}}$ to principal $\Diffspin$-bundles induces a map between the two spin structures constructed on $\hplus{\restr{\Xunivplus}{U_{ij}}}$. The $O(1)$-gerbe $\mcalG_{\hplus{\mbbX}}$ is constructed from these data.

  For each $i$ and $j$, define a map $P_{ij} \to Q_{ij}$ using the functoriality in \cref{thm1:spin on H^+}(2). \cref{thm1:spin on H^+}(3) asserts that this map is an isomorphism of principal $O(1)$-bundles. Furthermore, $s_{ijk}$ and $t_{ijk}$ correspond exactly. This gives a canonical isomorphism between the two $O(1)$-gerbes $\mcalG_{\mcalE}$ and $\mcalG_{\hplus{\mbbX}}$.
\end{remark}

\section{Appendix: Overview of the Witten deformation}\label{sec6:Witten deformation}

In this section, we give an overview of Witten deformation. The results of this section are used in the proofs of \cref{prop5:ker pssho} and \cref{prop5:can iso for vari spin}(2)(3). We refer to Furuta\cite{Furuta2007Index-Theorem} and Miyazawa\cite{miyazawa2021localization} as references.

In this paper, we frequently use the Witten deformation, especially in the proof of \cref{prop5:can iso for vari spin}(2)(3). It is used in the following way. In \cref{prop5:can iso for vari spin}(2)(3), we carry out an argument corresponding to the proof of the excision theorem for the index of a family. Therefore, we need a result corresponding to the fact that the families index of Dirac operators constructed on a family of manifolds coincides with the families index of Dirac operators obtained by restricting it to a family on open subsets. (In fact, we need a slightly stronger result than a mere coincidence of the index of the family.) In this paper, we apply the Witten deformation to a family of Dirac-type operators. However, in this section, we explain the Witten deformation for the case where the base space is a single point. The generalization to the case of a family is straightforward.

First, let us state the setup. Let us consider a tuple which consists of
\begin{itemize}
  \item a Riemannian manifold $M$,
  \item a vector bundle $S$ over $M$ with a Hermitian metric,
  \item a Clifford action $c$ on $S$ by $TM$,
  \item a Dirac-type operator $D$ on $S$,
  \item a section $h \in \Gamma(\Herm(S))$,
  \item a closed subset $F$ of $M$,
  \item positive real numbers $\bar \lambda$ and $T$.
\end{itemize}
Here we call $D$ a Dirac-type operator if it is a formally self-adjoint first-order differential operator on $S$, and its symbol is $ic$. In the following definition, we impose conditions on this tuple so that it fits the Witten deformation. As a preparation for this, we introduce some notation. First, we define the $L^2_1$-norm and $L^2_2$-norm on $\Gamma_c(S)$ by
\begin{align*}
  \norm{\phi}_{L^2_1}^2 & = \norm{D \phi}_{L^2}^2 + \norm{\phi}_{L^2}^2,  \\
  \norm{\phi}_{L^2_2}^2 & = \norm{D^2 \phi}_{L^2}^2 + \norm{\phi}_{L^2}^2
\end{align*}
and denote by $L^2_1(S)$ and $L^2_2(S)$ the completions with respect to these norms. Next, for $\lambda > 0$ and $t > 0$, we denote by
\[
  E_\lambda((D + th)^2)
\]
the vector space consisting of all eigenspinors of $(D + th)^2$ with eigenvalue $\lambda$ in the weak sense in $L^2(S)$. Also, we set
\[
  E_{\leq \lambda}((D + th)^2) = \bigoplus_{\mu \in [0, \lambda]} E_\mu((D + th)^2).
\]

\begin{definition}\label{def6:tame tuple}
  We say that a tuple $(M, S, c, D, h, F, \bar \lambda, T)$ is tame if it satisfies the following conditions:
  \begin{enumarabicp}
    \item $c$ is $L^\infty$-bounded.
    \item The zeroth-order operator $\{D, h\} = Dh + hD$ is $L^\infty$-bounded.
    \item At each point of $M \setminus F$, $h$ is an isomorphism, and $h^{-1}$ is $L^\infty$-bounded on $M \setminus F$.
    \item For any $t \geq T$,
    \[
      E_{\leq \bar \lambda}((D + th)^2)
    \]
    is finite-dimensional and contained in $L^2_2(S)$. Also, if we denote by
    \[
      E_{\leq \bar \lambda}((D + th)^2)^\perp
    \]
    the orthogonal complement of $E_{\leq \bar \lambda}((D + th)^2)$ with respect to the $L^2$-norm, then for any $\psi \in E_{\leq \bar \lambda}((D + th)^2)^\perp \cap L^2_2(S)$,
    \[
      \norm{(D + th) \psi}_{L^2}^2 \geq \bar \lambda \norm{\psi}_{L^2}^2
    \]
    holds. Moreover, $E_{\leq \bar \lambda}((D + th)^2)$ and $E_{\leq \bar \lambda}((D + th)^2)^\perp \cap L^2_2(S)$ are orthogonal with respect to the quadratic form
    \[
      \phi \mapsto \norm{(D + th) \phi}_{L^2}^2.
    \]
    \item For $t \geq T$,
    \[
      t(hD + Dh) + t^2 h^2 \geq \bar \lambda
    \]
    holds on $M \setminus F$.
    \item Let $\lambda \in (0,\bar \lambda)$ be arbitrary. If there exists a finite-dimensional subspace $E$ of $L^2_2(S)$ such that
    \begin{align*}
      \norm{(D + th) \phi}_{L^2}^2 \leq \lambda \norm{\phi}_{L^2}^2, &  &  & \phi \in E
    \end{align*}
    holds, then
    \[
      E_{\leq \lambda}((D + th)^2) \geq \dim E.
    \]
    Moreover, the largest eigenvalue of $(D + th)^2$ below $\lambda$ is at most
    \[
      \max_{\phi \in E,\ \norm{\phi}_{L^2} = 1} \norm{(D + th) \phi}_{L^2}^2.
    \]
  \end{enumarabicp}
\end{definition}

We state the proposition that is the goal of this section. First, let us describe the setup.
Let
\[
  (M, S, c, D, h, F, \bar \lambda, T),\ (M^\prime, S^\prime, c^\prime, D^\prime, h^\prime, F^\prime, \bar \lambda^\prime, T^\prime)
\]
be two tame tuples. Take an open subset $U$ of $M$ containing $F$ and an open subset $U^\prime$ of $M^\prime$ containing $F^\prime$. Assume that there are given a diffeomorphism
\[
  \map{\varphi}{U}{U^\prime}
\]
preserving the Riemannian metric and a metric-preserving isomorphism of vector bundles
\[
  \map{\tilde \varphi}{\restr{S}{U}}{\restr{S^\prime}{U^\prime}}
\]
covering $\varphi$ such that
\[
  \tilde \varphi_\ast c = c^\prime,\ \tilde \varphi_\ast D = D^\prime,\ \tilde \varphi_\ast h = h^\prime,\ \varphi^{-1}(F^\prime) = F.
\]
When a section of $S$ has support in $U$, we can identify it with a section of $S^\prime$ with support in $U^\prime$ through $\varphi$ and $\tilde \varphi$. In the following, we use this identification without explicit mention.

Take a smooth function
\[
  \map{\rho}{M}{[0, 1]}
\]
that is $L^\infty_1$-bounded, identically 1 on a neighborhood of $F$, and has support contained in $U$. For a positive real number $\lambda < \bar \lambda$ and $t \geq T$, define
\[
  \map{\Phi_{\lambda,t}}{E_{\leq \lambda}((D + th)^2)}{E_{\leq \lambda}((D^\prime + th^\prime)^2)}
\]
by
\[
  \Phi_{\lambda, t} (\phi) = \Pi_{\lambda, t} \rho \phi.
\]
Here, $\Pi_{\lambda, t}$ denotes the orthogonal projection from $L^2(S^\prime)$ to $E_{\leq \lambda}((D^\prime + th^\prime)^2)$.

\begin{proposition}\label{prop6:Witten deformation}
  In the above setup, assume that a positive real number $\lambda < \min\{\bar \lambda, \bar \lambda^\prime\}$ satisfies the following property: for any $t \geq T$,
  \[
    \lambda \notin \sigma((D + th)^2)
  \]
  holds, and furthermore,
  \[
    \sigma((D + th)^2) \cap [0, \lambda]
  \]
  is independent of $t \geq T$. Then, there exists $T_0 \geq \max\{T, T^\prime\}$ such that for any $t \geq T_0$,
  \[
    \lambda \notin \sigma((D^\prime + th^\prime)^2)
  \]
  holds, and furthermore, $\Phi_{\lambda, t}$ is an isomorphism.
\end{proposition}

\begin{remark}
  In \cref{prop6:Witten deformation}, it is not necessary to impose a condition on the spectrum for $(D^\prime + th^\prime)^2$. That is, the assumption is asymmetric for $(D + th)^2$ and $(D^\prime + th^\prime)^2$.
\end{remark}

In the proof of \cref{prop5:can iso for vari spin}, a more precise statement is needed. To prepare for that assertion, we define a distance on the Grassmannian of a Hilbert space.

\begin{definition}
  Let $H$ be a Hilbert space. For a positive integer $r$, let $\Gr(r, H)$ denote the set of all $r$-dimensional subspaces of $H$. We define a distance $d_H$ on $\Gr(r, H)$ as follows: For $E, E^\prime \in \Gr(r, H)$, $d_H(E, E^\prime)$ is the infimum of $d > 0$ such that the followings hold:
  \begin{itemize}
    \item For any orthonormal basis $e_1,\ \cdots,\ e_r$ of $E$, there exists an orthonormal basis $e^\prime_1,\ \cdots,\ e^\prime_r$ of $E^\prime$ satisfying
          \[
            \sqrt{\sum_{i = 1}^r \norm{e_i - e^\prime_i}^2} \leq d.
          \]
    \item For any orthonormal basis $e^\prime_1,\ \cdots,\ e^\prime_r$ of $E^\prime$, there exists an orthonormal basis $e_1,\ \cdots,\ e_r$ of $E$ satisfying
          \[
            \sqrt{\sum_{i = 1}^r \norm{e_i - e^\prime_i}^2} \leq d.
          \]
  \end{itemize}
\end{definition}

\begin{proposition}\label{prop6:distance is small}
  Let two tame tuples
  \[
    (M, S, c, D, h, F, \bar \lambda, T),\ (M^\prime, S^\prime, c^\prime, D^\prime, h^\prime, F^\prime, \bar \lambda^\prime, T^\prime)
  \]
  and
  \[
    \map{\varphi}{U}{U^\prime},\ \map{\tilde \varphi}{\restr{S}{U}}{\restr{S^\prime}{U^\prime}},\ \map{\rho}{M}{[0, 1]},\ \lambda
  \]
  satisfy the assumptions of \cref{prop6:Witten deformation}. Then, for any $\varepsilon > 0$, there exists $T_0 \geq \max\{T, T^\prime\}$ such that for $t \geq T_0$,
  \[
    d_{L^2(S^\prime)}(\rho \cdot E_{\leq \lambda}((D + th)^2), E_{\leq \lambda}((D^\prime + th^\prime)^2)) < \varepsilon.
  \]
\end{proposition}

\cref{prop6:Witten deformation} follows immidiately from \cref{prop6:distance is small}.

We now proceed to the proof of \cref{prop6:distance is small}. The following lemma is the key point.

\begin{lemma}\label{lem6:distance is small}
  Let two tame tuples
  \[
    (M, S, c, D, h, F, \bar \lambda, T),\ (M^\prime, S^\prime, c^\prime, D^\prime, h^\prime, F^\prime, \bar \lambda^\prime, T^\prime)
  \]
  and
  \[
    \map{\varphi}{U}{U^\prime},\ \map{\tilde \varphi}{\restr{S}{U}}{\restr{S^\prime}{U^\prime}},\ \map{\rho}{M}{[0, 1]},\ \lambda
  \]
  satisfy the assumptions of \cref{prop6:Witten deformation}. Then, for any $\varepsilon > 0$, there exists $T_1 \geq \max\{T, T^\prime\}$ such that for any $t \geq T_1$,
  \[
    \dim E_{\leq \lambda}((D + th)^2) = \dim E_{\leq \lambda}((D^\prime + th^\prime)^2)
  \]
  and
  \begin{align*}
    d_{L^2(S)}(E_{\leq \lambda}((D + th)^2), \rho \cdot E_{\leq \lambda}(D + th)^2)                                    & < \varepsilon, \\
    d_{L^2(S^\prime)}(E_{\leq \lambda}((D^\prime + th^\prime)^2), \rho \cdot E_{\leq \lambda}(D^\prime + th^\prime)^2) & < \varepsilon, \\
    d_{L^2(S)}(E_{\leq \lambda}((D + th)^2), \rho \cdot E_{\leq \lambda}((D^\prime + th^\prime)^2))                    & < \varepsilon
  \end{align*}
\end{lemma}

First, we give the proof of \cref{prop6:distance is small} assuming \cref{lem6:distance is small}.

\begin{proof}[Proof of \cref{prop6:distance is small}]
  In this proof, we use the abbreviations
  \[
    E_{\lambda, t} = E_{\leq \lambda}((D + th)^2),\ E_{\lambda, t}^\prime = E_{\leq \lambda}((D^\prime + th^\prime)^2).
  \]
  Let $\varepsilon > 0$ be arbitrary. Take $T_0 \geq \max\{T, T^\prime\}$ such that for any $t \geq T_0$,
  \begin{align*}
    d_{L^2(S)}(E_{\lambda, t}, \rho \cdot E_{\lambda, t})                      & < \varepsilon / 3, \\
    d_{L^2(S^\prime)}(E_{\lambda, t}^\prime, \rho \cdot E_{\lambda, t}^\prime) & < \varepsilon / 3, \\
    d_{L^2(S)}(E_{\lambda, t}, \rho \cdot E_{\lambda, t}^\prime)               & < \varepsilon / 3.
  \end{align*}
  We obtain the estimate
  \begin{align*}
    d_{L^2(S^\prime)}(\rho \cdot E_{\lambda, t}, E_{\lambda, t}^\prime) & \leq d_{L^2(S^\prime)}(\rho \cdot E_{\lambda, t}, \rho \cdot E_{\lambda, t}^\prime) + d_{L^2(S^\prime)}(\rho \cdot E_{\lambda, t}^\prime, E_{\lambda, t}^\prime).
  \end{align*}
  Here, we have
  \[
    d_{L^2(S^\prime)}(\rho \cdot E_{\lambda, t}, \rho \cdot E_{\lambda, t}^\prime) = d_{L^2(S)}(\rho \cdot E_{\lambda, t}, \rho \cdot E_{\lambda, t}^\prime),
  \]
  and so for $t \geq T_0$,
  \begin{align*}
         & d_{L^2(S^\prime)}(\rho \cdot E_{\lambda, t}, E_{\lambda, t}^\prime)                                                                                                                               \\
    \leq & d_{L^2(S)}(\rho \cdot E_{\lambda, t}, E_{\lambda, t}) + d_{L^2(S)}(E_{\lambda, t}, \rho \cdot E_{\lambda, t}^\prime) + d_{L^2(S^\prime)}(\rho \cdot E_{\lambda, t}^\prime, E_{\lambda, t}^\prime) \\
    <    & \varepsilon /3 + \varepsilon / 3 + \varepsilon / 3 = \varepsilon,
  \end{align*}
  which is the desired estimate.
\end{proof}

We now proceed to the proof of \cref{lem6:distance is small}. First, we prove a lemma corresponding to Miyazawa\cite[Lemma 6.1]{miyazawa2021localization}.

\begin{lemma}\label{lem6:L^2 localization}
  Let $(M, S, c, D, h, F, \bar \lambda, T)$ be a tame tuple. Take a smooth function
  \[
    \map{\rho}{M}{[0, 1]}
  \]
  which is identically 1 on a neighborhood of $F$. Then, there exist functions
  \[
    \map{A}{[T, \infty)}{(0, \infty)},\ \map{B}{[T, \infty)}{[0, 1]}
  \]
  satisfying the following properties:
  \begin{itemize}
    \item For any $t \geq T$ and $\phi \in L^2_2(S)$ such that
          \[
            \norm{(D + th) \phi}_{L^2(M, S)}^2 \leq \bar \lambda \norm{\phi}_{L^2(M, S)}^2,
          \]
          we have
          \begin{align}
            \norm{\phi}_{L^2(M \setminus F, S)} & \leq A(t) \norm{\phi}_{L^2(M, S)}, \label{eq6:A(t)}                               \\
            B(t) \norm{\phi}_{L^2(M, S)}        & \leq \norm{\rho \phi}_{L^2(M, S)} \leq \norm{\phi}_{L^2(M, S)} \label{eq6: B(t)}.
          \end{align}
          Here, for an open subset $V$ of $M$, we set
          \[
            \norm{\phi}_{L^2(V, S)}^2 = \int_{V} \abs{\phi}^2.
          \]
    \item As $t \to \infty$, we have
          \[
            A(t) \to 0,\ B(t) \to 1.
          \]
  \end{itemize}
\end{lemma}

\begin{proof}
  The proof is identical to that of Miyazawa\cite[Lemma 6.1]{miyazawa2021localization}. We repeat the proof here for the reader's convenience. Assume that $\phi \in L^2_2(S)$ satisfies
  \[
    \norm{(D + th) \phi}_{L^2(M, S)}^2 \leq \bar \lambda \norm{\phi}_{L^2(M, S)}^2
  \]
  Then, the following estimates hold:
  \begin{align*}
    \bar \lambda \norm{\phi}_{L^2(M, S)}^2 & \geq \norm{(D + th) \phi}_{L^2(M, S)}^2                                                                                                                             \\
                                           & = \norm{D \phi}_{L^2(M, S)}^2 + t^2 \norm{h \phi}_{L^2(M, S)}^2 + 2 t \abracket{\{D, h\}\phi, \phi}_{L^2(M, S)}                                                     \\
                                           & \geq t^2 \norm{h \phi}_{L^2(M \setminus F)}^2 - 2t \norm{\{D, h\}}_{L^\infty} \norm{\phi}_{L^2(M, S)}^2                                                             \\
                                           & \geq t^2 \frac{\norm{\phi}_{L^2(M \setminus F, S)}^2}{\norm{h^{-1}}_{L^\infty(M \setminus F, S)}^2} - 2t \norm{\{D, h\}}_{L^\infty(M,S)} \norm{\phi}_{L^2(M, S)}^2.
  \end{align*}
  Rearranging this inequality, we obtain
  \[
    \norm{\phi}_{L^2(M \setminus F, S)}^2 \leq \frac{(\bar \lambda + 2t \norm{\{D, h\}}_{L^\infty(M,S)}) \norm{h^{-1}}_{L^\infty(M \setminus F, S)}^2}{t^2} \norm{\phi}_{L^2(M, S)}^2.
  \]
  Therefore, we can set
  \[
    A(t)^2 = \frac{(\bar \lambda + 2t \norm{\{D, h\}}_{L^\infty}) \norm{h^{-1}}_{L^\infty(M \setminus F, S)}^2}{t^2}.
  \]
  Moreover, we have
  \begin{align*}
    \norm{\rho \phi}_{L^2(M, S)} & \geq \norm{\phi}_{L^2(M, S)} - \norm{(1 - \rho) \phi}_{L^2(M, S)}  \\
                                 & \geq \norm{\phi}_{L^2(M, S)} - \norm{\phi}_{L^2(M \setminus F, S)} \\
                                 & \geq \norm{\phi}_{L^2(M ,S)} - A(t) \norm{\phi}_{L^2(M, S)}        \\
                                 & = (1 - A(t)) \norm{\phi}_{L^2(M, S)}.
  \end{align*}
  Hence, we can set
  \[
    B(t) = \max\{0, 1 - A(t)\}
  \]
  and this finishes the proof.
\end{proof}

Next, we prove a refinement of the inequality shown in the proof of Miyazawa\cite[Lemma 6.2]{miyazawa2021localization}.

\begin{lemma}\label{lem6:D+th localization}
  Let a tame tuple
  \[
    (M, S, c, D, h, F, \bar \lambda, T)
  \]
  be given. Take a smooth function
  \[
    \map{\rho}{M}{[0, 1]}
  \]
  which is identically 1 on a neighborhood of $F$ and $L^\infty_1$-bounded. Then, for any $\delta > 0$, there exists $T_2 \geq \max\{T, T^\prime\}$ such that for any $t \geq T_2$ and $\phi \in E_{\bar \lambda}((D + th^2))$ with $\norm{\phi}_{L^2} = 1$, we have
  \[
    -\delta < \norm{(D + th)(\rho\phi)}_{L^2}^2 - \norm {(D + th) \phi}_{L^2}^2 < \delta
  \]
\end{lemma}

\begin{proof}
  The following equality holds:
  \begin{align}\label{eq6:to be estimated}
    \begin{split}
        & \norm{(D + th) \rho \phi}_{L^2}^2                                                                                                  \\
      = & \norm{[D, \rho] \phi}_{L^2}^2 + 2 \RePart \abracket{\rho (D + th) \phi, [D, \rho] \phi}_{L^2} + \norm{\rho (D + th) \phi}_{L^2}^2.
    \end{split}
  \end{align}
  Since $\phi \in E_{\leq \bar \lambda}((D + th)^2)$, we have
  \[
    \norm{(D + th) \phi}_{L^2}^2 \leq \bar \lambda \norm{\phi}_{L^2}^2 \leq \bar \lambda
  \]
  Also, since $\rho$ is identically 1 on a neighborhood of $F$, we have
  \begin{align*}
    \norm{[D, \rho] \phi}_{L^2} & \leq \norm{[D, \rho] \phi}_{L^2(M \setminus F)}                                   \\
                                & = C \norm{c}_{L^\infty} \norm{d \rho}_{L^\infty} \norm{\phi}_{L^2(M \setminus F)} \\
                                & \leq C \norm{c}_{L^\infty} \norm{d \rho}_{L^\infty} A(t) \norm{\phi}_{L^2}        \\
                                & = C \norm{c}_{L^\infty} \norm{d \rho}_{L^\infty} A(t).
  \end{align*}
  Here, $C$ is a constant depending only on the dimension of $M$. Set
  \[
    A^\prime(t) = C \norm{c}_{L^\infty} \norm{d \rho}_{L^\infty} A(t).
  \]
  From the above estimates, we have
  \[
    \abs{\norm{[D, \rho] \phi}_{L^2}^2 + 2 \RePart \abracket{\rho (D + th) \phi, [D, \rho] \phi}_{L^2}} \leq A^\prime(t)^2 + 2 A^\prime(t) \sqrt{\bar \lambda}.
  \]
  Also, since $(D + th) \phi \in E_{\leq \lambda}((D + th)^2)$, by \cref{lem6:L^2 localization}, we have
  \[
    B(t) \norm{(D + th) \phi}_{L^2} \leq \norm{\rho (D + th) \phi}_{L^2} \leq \norm{(D + th) \phi}_{L^2}.
  \]
  Combining these, we obtain
  \begin{align*}
    \norm{(D + th) \phi}_{L^2}^2 & - (1 - B(t)^2) \bar \lambda - A^\prime(t)^2 - 2 A^\prime(t) \sqrt{\bar \lambda}                                               \\
                                 & \leq \norm{(D + th) \rho \phi}_{L^2}^2 \leq \norm{(D + th) \phi}_{L^2}^2 + A^\prime(t)^2 + 2 A^\prime(t) \sqrt{\bar \lambda}.
  \end{align*}
  Since $A(t) \to 0$ and $B(t) \to 1$ as $t \to \infty$, we obtain the desired inequality.
\end{proof}

Before proving \cref{lem6:distance is small}, we prove a general lemma about non-negative self-adjoint unbounded operators on Hilbert spaces. Let us introduce some notation to state the lemma. Let $H$ be a Hilbert space. Fix a non-negative real number $r$ and positive real numbers $\lambda_0 < \lambda_1$. Let $\mscrS_{r, \lambda_0, \lambda_1}$ be the set of non-negative self-adjoint operators $A$ on $H$ satisfying:
\begin{itemize}
  \item The spectrum of $A$ up to $\lambda_0$ is discrete and the dimension of the direct sum of the eigenspaces up to $\lambda_0$ is $r$. Furthermore, the spectrum of $A$ and $[\lambda_0, \lambda_1]$ is empty.
  \item For any $v \in E_{\leq \lambda}(A)^\perp \cap D(A)$, we have
        \[
          \abracket{Av, v} \geq \lambda_1 \norm{v}^2.
        \]
        Moreover, $E_{\leq \lambda_1}(A)$ and $E_{\leq \lambda_1}(A)^\perp \cap D(A)$ are orthogonal with respect to the quadratic form
        \[
          v \mapsto \abracket{Av, v}.
        \]
\end{itemize}
Take any $A \in \mscrS_{r, \lambda_0, \lambda_1}$ and an $r$-dimensional subspace $E$ of $D(A)$. Let
\[
  \lambda_1(A, E) \leq \cdots \leq \lambda_r(A, E)
\]
be the eigenvalues (with multiplicity) of the restriction of the quadratic form
\[
  v \mapsto \abracket{Av, v}
\]
to $E$, arranged in increasing order. When $E = E_{\leq \lambda_0}(A) = E_{\leq \lambda_1}(A)$, these coincide with the eigenvalues of $A$.

The following lemma is essentially the same as Furuta\cite[Theorem 5.35]{Furuta2007Index-Theorem}.

\begin{lemma}\label{lem6:Hilb distance}
  Let $H$ be a Hilbert space. Fix a non-negative real number $r$ and  positive real numbers $\lambda_0, \lambda_1$. Then, for any $d > 0$, there exists $\delta > 0$ such that the following holds: For any $A \in \mscrS_{r, \lambda_0, \lambda_1}$ and an $r$-dimensional subspace $E$ of $D(A)$, if we have
  \begin{align*}
    \abs{\lambda_k(A, E) - \lambda_k(A, E_{\leq \lambda_0}(A))} < \delta, &  &  & 1 \leq k \leq r,
  \end{align*}
  then we have
  \[
    d_H(E, E_{\leq \lambda_0}(A)) < d.
  \]
\end{lemma}

\begin{proof}
  Take any $A \in \mscrS_{r, \lambda_0, \lambda_1}$ and an $r$-dimensional subspace $E$ of $D(A)$. Further, assume that for all $1 \leq k \leq r$, we have
  \[
    \abs{\lambda_k(A, E) - \lambda_k(A, E_{\leq \lambda_0}(A))} < \delta.
  \]
  We may take $\delta$ to be smaller than $(\lambda_1 - \lambda_0)/2$. From the fact
  \begin{align*}
    \abracket{Av, v} \geq \lambda_1 \norm{v}^2, &  &  & v \in E_{\leq \lambda_0}(A)^\perp \cap D(A),
  \end{align*}
  it follows that $E$ can be described as the graph of a linear map
  \[
    \map{f}{E_{\leq \lambda_0}(A)}{E_{\leq \lambda_0}(A)^\perp \cap D(A)}.
  \]
  We prove the statement by induction on $r$ in this situation. Let $v_r \in E_{\leq \lambda_0}(A)$ be a unit eigenvector of $A$ corresponding to the eigenvalue $\lambda_r(A, E_{\leq \lambda_0}(A))$. Then we have
  \begin{align*}
    \lambda_r(A, E) & \geq \frac{\abracket{Av_r, v_r} + \abracket{A f(v_r), f(v_r)}}{\norm{v_r}^2 + \norm{f(v_r)}^2}                             \\
                    & \geq  \frac{\lambda_r(A, E_{\leq \lambda_0}(A)) \norm{v_r}^2 + \lambda_1 \norm{f(v_r)}^2}{\norm{v_r}^2 + \norm{f(v_r)}^2}.
  \end{align*}
  Rearranging this, we obtain
  \[
    \norm{f(v_r)}^2 \leq \frac{\lambda_r(A, E) - \lambda_r(A, E_{\leq \lambda_0}(A))}{\lambda_1 - \lambda_r(A, E)} \leq \frac{2 \delta}{\lambda_1 - \lambda_0}.
  \]
  Let $E^\prime$ be the subspace of $E$ spanned by $v_r$ and vectors orthogonal to $v_r + f(v_r)$. The above inequality implies that if $\delta$ is sufficiently small, we have
  \[
    d(E, E^\prime) \leq d/r.
  \]
  The spaces $E^\prime$ and $E_{\leq \lambda}(A)$ share the common vector $v_r$. So it suffices to show that the distance between
  \[
    E \cap v_r^\perp,\ E_{\leq \lambda_0} \cap v_r^\perp
  \]
  is small. It can be seen that the eigenvalues of these $(r-1)$-dimensional subspaces are almost equal, so the problem reduces to the $(r-1)$-dimensional case. In the above argument, $\delta$ does not depend on the choice of $A$ and $E$.
\end{proof}

\begin{proof}[Proof of \cref{lem6:distance is small}]
  Let $\varepsilon > 0$ be arbitrary. For sufficiently large $t$, \cref{lem6:L^2 localization} implies that
  \begin{align*}
    d_{L^2(S)}(E_{\leq \lambda}((D + th)^2), \rho \cdot E_{\leq \lambda}(D + th)^2)                                    & < \varepsilon, \\
    d_{L^2(S^\prime)}(E_{\leq \lambda}((D^\prime + th^\prime)^2), \rho \cdot E_{\leq \lambda}(D^\prime + th^\prime)^2) & < \varepsilon.
  \end{align*}
  Let
  \[
    \lambda_1 \leq \cdots \leq \lambda_r,\ \lambda_1^\prime \leq \cdots \leq \lambda^\prime_s
  \]
  be the eigenvalues (counted with multiplicity) up to $\lambda$ of $(D + th)^2$ and $(D^\prime + th^\prime)^2$, respectively. Further, let
  \[
    \lambda_1(\rho) \leq \cdots \leq \lambda_r(\rho),\ \lambda_1^\prime(\rho) \leq \cdots \leq \lambda^\prime_s(\rho)
  \]
  be the eigenvalues of the restrictions of $(D + th)^2$ and $(D^\prime + th^\prime)^2$ to
  \[
    \rho \cdot E_{\leq \lambda}((D + th)^2),\ \rho \cdot E_{\leq \lambda}((D^\prime + th^\prime)^2)
  \]
  respectively.
  From \cref{lem6:L^2 localization} and \cref{lem6:D+th localization}, for sufficiently large $t$, for all $1 \leq k \leq r$, $1 \leq l \leq s$, we have
  \begin{align}\label{eq6:similar eigenvalue}
    \lambda_k \approx \lambda_k(\rho),\ \lambda^\prime_l \approx \lambda_l^\prime(\rho).
  \end{align}
  Therefore, from \cref{def6:tame tuple}(6), we see that $r = s$. Furthermore, \cref{def6:tame tuple}(6) implies that for all $1 \leq k \leq r$, we have
  \[
    \lambda_k \leq \lambda_k^\prime(\rho),\ \lambda^\prime_k \leq \lambda_k(\rho).
  \]
  Hence, together with \cref{eq6:similar eigenvalue}, we see that
  \[
    \lambda_k \approx \lambda_k^\prime(\rho).
  \]
  Applying \cref{lem6:Hilb distance} with
  \[
    H = L^2(S),\ A = (D + th)^2,\ E = \rho \cdot E_{\leq \lambda}((D^\prime + th^\prime)^2),
  \]
  we see that
  \[
    d_{L^2(S)}(E_{\leq \lambda}((D + th)^2), \rho \cdot E_{\leq \lambda}((D^\prime + th^\prime)^2))                     < \varepsilon.
  \]
  This finishes the proof.
\end{proof}

\section{Appendix: Irreducible Clifford representations on quaternionic vector space}\label{sec7:ssho for H}

This section is preparatory for \cref{sec8:deform ssho}. We present two constructions of the irreducible Clifford representations on a quaternionic vector space and show that there is a canonical isomorphism between them.

Let $V_\HB$ be a quaternionic vector space with a $Pin(2)$-invariant metric. The representation we consider is that of $Cl(V_\HB^- \oplus V_\HB)$, where $V_\HB^-$ is the same quaternionic vector space as $V_\HB$, but with the metric multiplied by $-1$ (i.e., it is negative definite).

The goal of this section is to describe the representation of $Cl(V_\HB^- \oplus V_\HB)$ in two ways, taking into account the action of $j$ on $V_\HB$. The first construction uses the complex exterior algebra of $V_\HB$. The second construction uses the real exterior algebra of $V_\HB$.

First, we describe the representation of $Cl(V_\HB^- \oplus V_\HB)$ using two copies of the complex exterior algebra of $V_\HB$. Let
\[
  S_0^{V_\HB} = (\Lambda_\C^\ast V_\HB) \otimes_\C (\Lambda_\C^\ast V_\HB).
\]
We define
\[
  \map{c_0^{V_\HB}}{V_\HB \times S_0^{V_\HB}}{S_0^{V_\HB}},\ \map{h_0^{V_\HB}}{V_\HB}{\Herm(S_0^{V_\HB})},\ \map{\tau_0^{V_\HB}}{S_0^{V_\HB}}{S_0^{V_\HB}}
\]
by
\begin{align}\label{eq7:tau c h1}
  c_0^{V_\HB}(v, \omega_0 \otimes \omega_{V_\HB}) & = (\varepsilon \omega_0) \otimes ((v^\wedge - v^{\lrcorner}) \omega_{V_\HB}), \\
  h_0^{V_\HB}(v, \omega_0 \otimes \omega_{V_\HB}) & = i ((v^\wedge - v^{\lrcorner})\omega_0) \otimes \omega_{V_\HB},              \\
  \tau_0^{V_\HB}(\omega_0 \otimes \omega_{V_\HB}) & = (\omega_0 \cdot j) \otimes (\omega_{V_\HB} \cdot j)
\end{align}
where $\varepsilon$ is the operator representing the $\Z/2$-grading of $\Lambda_\C^\ast V_\HB$. It can be verified that $c_0^{V_\HB}$ and $h_0^{V_\HB}$ give a representation of $Cl(V_\HB^- \oplus V_\HB)$. That is, the first realization of the representation of $Cl(V_\HB^- \oplus V_\HB)$ is the triple
\begin{align}\label{eq7:ssho1}
  (S_0^{V_\HB}, c_0^{V_\HB}, h_0^{V_\HB}).
\end{align}

Next, we describe the representation of $Cl(V_\HB^- \oplus V_\HB)$ using the real exterior algebra of $V_\HB$. Let
\[
  S_1^{V_\HB} = (\Lambda_\R^\ast V_\HB) \otimes_\R \C.
\]
We define
\[
  \map{c_1^{V_\HB}}{V_\HB \times S_1^{V_\HB}}{S_1^{V_\HB}},\ \map{h_1^{V_\HB}}{V_\HB}{\Herm(S_1^{V_\HB})},\ \map{\tau_1^{V_\HB}}{S_1^{V_\HB}}{S_1^{V_\HB}}
\]
by
\begin{align}\label{eq7:tau c h2}
  \begin{split}
    c_1^{V_\HB}(v, \omega \otimes z) & = (v^\wedge - v^{\lrcorner}) \omega \otimes z, \\
    h_1^{V_\HB}(v, \omega \otimes z) & = (v^\wedge + v^{\lrcorner})\omega \otimes z,  \\
    \tau_1^{V_\HB}(\omega \otimes z) & = (\omega \cdot j) \otimes \bar z.
  \end{split}
\end{align}
The second realization of the representation of $Cl(V_\HB^- \oplus V_\HB)$ is the tuple
\begin{align}\label{eq7:ssho2}
  (S_1^{V_\HB}, c_1^{V_\HB}, h_1^{V_\HB}).
\end{align}

The following lemmas are the direct consequences of the constructions.

\begin{lemma}\label{lem7:commutativity}
  For $k = 0, 1$, the following diagrams commute.
  \[
    \begin{tikzcd}[column sep=1cm]
      V_\HB \times S_k^{V_\HB} \arrow[r, "c_k^{V_\HB}"] \arrow[d, "(\cdot j)\times \tau_k^{V_\HB}"] & S_k^{V_\HB} \arrow[d, "\tau_k^{V_\HB}"] \\
      V_\HB \times S_k^{V_\HB} \arrow[r, "c_k^{V_\HB}"] & S_k^{V_\HB}
    \end{tikzcd}
    \quad
    \begin{tikzcd}[column sep=1cm]
      V_\HB \times S_k^{V_\HB} \arrow[r, "h_k^{V_\HB}"] \arrow[d, "(\cdot j) \times \tau_k^{V_\HB}"] & S_k^{V_\HB} \arrow[d, "\tau_k^{V_\HB}"] \\
      V_\HB \times S_k^{V_\HB} \arrow[r, "h_k^{V_\HB}"] & S_k^{V_\HB}
    \end{tikzcd}
  \]
\end{lemma}

\begin{lemma}\label{lem7:dir sum}
  Let $V_\HB$ and $V_\HB^\prime$ be quaternionic vector spaces with $Pin(2)$-invariant metrics. For $k = 0, 1$, there are canonical $\Z/2$-grading-preserving isomorphisms
  \[
    S_k^{V_\HB \oplus V_\HB^\prime} \cong S_k^{V_\HB} \otimes S_k^{V_\HB^\prime}
  \]
  that send
  \[
    \tau_k^{V_\HB \oplus V_\HB^\prime},\ c_k^{V_\HB \oplus V_\HB^\prime},\ h_k^{V_\HB \oplus V_\HB^\prime}
  \]
  to
  \[
    \tau_k^{V_\HB} \otimes \tau_k^{V_\HB^\prime},\ c_k^{V_\HB} \otimes 1 + \varepsilon \otimes c_k^{V_\HB^\prime},\ h_k^{V_\HB} \otimes 1 + \varepsilon \otimes h_k^{V_\HB^\prime}
  \]
  respectively.
\end{lemma}

\begin{proposition}\label{prop7:ssho vect sp}
  Let $V_\HB$ be a quaternionic vector space with a $Pin(2)$-invariant metric. There exists a unique $\Z/2$-grading-preserving complex vector space isomorphism
  \[
    \map{F^{V_\HB}}{S_0^{V_\HB}}{S_1^{V_\HB}}
  \]
  satisfying the following properties:
  \begin{enumarabicp}
    \item $F^{V_\HB}$ intertwines $c_0^{V_\HB}$ with $c_1^{V_\HB}$, $h_0^{V_\HB}$ with $h_1^{V_\HB}$, and $\tau_0^{V_\HB}$ with $\tau_1^{V_\HB}$.
    \item For any orthonormal quaternionic basis $e_1, \cdots, e_n$ of $V_\HB$, the element $1 \otimes 1 \in S_0^{V_\HB}$ corresponds to
    \[
      \bigwedge_{l = 1}^n (e_l \otimes 1 + (e_l i) \otimes i) \wedge (e_l j \otimes 1 + (e_l ji) \otimes i) \in S_1^{V_\HB}.
    \]
  \end{enumarabicp}
  Furthermore, if $V_\HB^\prime$ is another quaternionic vector space with a $Pin(2)$-invariant metric, then
  \[
    \map{F^{V_\HB} \otimes F^{V_\HB^\prime}}{S_0^{V_\HB} \otimes S_0^{V_\HB^\prime}}{S_1^{V_\HB} \otimes S_1^{V_\HB^\prime}}
  \]
  and
  \[
    \map{F^{V_\HB \oplus V_\HB^\prime}}{S_0^{V_\HB \oplus V_\HB^\prime}}{S_1^{V_\HB \oplus V_\HB^\prime}}
  \]
  coincide via the canonical isomorphisms in \cref{lem7:dir sum}.
\end{proposition}

\begin{proof}
  The uniqueness follows from the fact that $S_0^{V_\HB}$ is generated by elements obtained by applying $c_0^{V_\HB}$ and $h_0^{V_\HB}$ to elements of $V_\HB$. The existence is easily seen by explicitly constructing the isomorphism with a fixed orthonormal quaternionic basis.

  When $V_\HB^\prime$ is another quaternionic vector space with a $Pin(2)$-invariant metric, the fact that $F^{V_\HB} \otimes F^{V_\HB^\prime}$ satisfies (1) and (2) and the uniqueness of $F^{V_\HB \oplus V_\HB^\prime}$ imply the stated compatibility.
\end{proof}

In the proof of \cref{lem5:iso of spinors along fiber}, we use \cref{prop7:ssho vect sp} in the following form.

\begin{proposition}\label{prop7:ssho equality}
  Let $V_\HB$ and $V_\HB^\prime$ be quaternionic vector spaces with $Pin(2)$-equivariant metrics. For $k = 0, 1$, let
  \[
    \tilde S_k^{V_\HB^\prime} = (S(V_\HB) \times V_\HB^\prime \times S_k^{V_\HB^\prime}) / U(1).
  \]
  Define
  \[
    \tilde \tau_k^{V_\HB^\prime},\ \tilde c_k^{V_\HB^\prime},\ \tilde h_k^{V_\HB^\prime},
  \]
  by pulling back $\tau_k^{V_\HB^\prime}$, $c_k^{V_\HB^\prime}$, $h_k^{V_\HB^\prime}$ to $\tilde S_k^{V_\HB^\prime}$ and taking the $U(1)$-quotient. Then, there exists a canonical vector bundle isomorphism
  \[
    \map{\tilde F^{V_\HB^\prime}}{\tilde S_0^{V_\HB^\prime}}{\tilde S_1^{V_\HB^\prime}}
  \]
  that intertwines $\tilde \tau_k^{V_\HB^\prime}$, $\tilde c_k^{V_\HB^\prime}$, $\tilde h_k^{V_\HB^\prime}$. Furthermore, if $V_\HB^\pprime$ is another quaternionic vector bundle with a $Pin(2)$-equivariant metric, then
  \[
    \map{\tilde F^{V_\HB^\prime \oplus V_\HB^\pprime}}{\tilde S_0^{V_\HB^\prime \oplus V_\HB^\pprime}}{\tilde S_1^{V_\HB^\prime \oplus V_\HB^\pprime}}
  \]
  and
  \[
    \map{\tilde F^{V_\HB^\prime} \boxtimes \tilde F^{V_\HB^\pprime}}{\tilde S_0^{V_\HB^\prime} \boxtimes \tilde S_0^{V_\HB^\pprime}}{\tilde S_1^{V_\HB^\prime} \boxtimes \tilde S_1^{V_\HB^\pprime}}
  \]
  coincide via the canonical isomorphisms in \cref{lem7:dir sum}.
\end{proposition}
\begin{proof}
  This follows from \cref{prop7:ssho vect sp}.
\end{proof}

\section{Appendix: Deformation of the super-symmetric harmonic oscillator}\label{sec8:deform ssho}

This section discusses the deformation of the super-symmetric harmonic oscillator on the quaternionic vector space $V_\HB$ used in \cref{ssec5:quat vb}. The original super-symmetric harmonic oscillator is described by the representation \cref{eq7:ssho2} of $Cl(V_\HB^- \oplus V_\HB)$ together with the Levi-Civita connection. For technical reasons, we need to deform this to data with a cylindrical end (which is necessary for $\mcalD_{\mscrV}^\prime + t \mbfh_{\mscrV}^\prime$ to become a family of Fredholm operators, using the notation from \cref{ssec5:quat vb}). Before and after the deformation of the super-symmetric harmonic oscillator, we still provide a Clifford action $c$ of $TV_\HB$ and a section
\[
  h \in \Herm(S_1^{V_\HB})
\]
on the vector bundle $V_\HB \times S_1^{V_\HB}$ over $V_\HB$. However, the way the Riemannian metric on $TV_\HB$ is put and the metric on the vector bundle $V_\HB \times S_{V_\HB}$ is put are different, and accordingly, the definitions of $c$ and $h$ change.

Let $V_\HB$ be a quaternionic vector space with a $Pin(2)$-invariant metric. We use the same notation as in \cref{sec7:ssho for H}. First, we introduce notation for the pre-deformation super-symmetric harmonic oscillator. Let $g$ be the Riemannian metric on $V_\HB$ induced from the metric as a vector space. Note that the vector bundle over $V_\HB$
\begin{align}\label{eq8:V_HB times S_HB prime}
  V_\HB \times S_1^{V_\HB}
\end{align}
is canonically isomorphic to
\[
  \Lambda_\R^\ast TV_\HB \otimes_\R \C.
\]
We denote by $\mcalS_1^{V_\HB}$ the vector bundle \cref{eq8:V_HB times S_HB prime} equipped with the metric induced from this isomorphism and $g$. Let $\nabla_1^{V_\HB}$ be the Levi-Civita connection on $\mcalS_1^{V_\HB}$. The pre-deformation super-symmetric harmonic oscillator is the tuple
\begin{align}\label{eq8:before deform}
  (g, \mcalS_1^{V_\HB}, c_1^{V_\HB}, \nabla_1^{V_\HB}, h_1^{V_\HB}, \tau_1^{V_\HB}).
\end{align}
Here, $c_1^{V_\HB}$ is regarded as the Clifford action of $TV_\HB$ on $\mcalS_1^{V_\HB}$, and $h_1^{V_\HB}$ is regarded as a section of the vector bundle $\Herm(\mcalS_1^{V_\HB})$.

We now describe the deformed super-symmetric harmonic oscillator. We view $V_\HB$ as a manifold
\[
  V_\HB = B(V_\HB) \cup S(V_\HB) \times [1, \infty).
\]
We want to put a translation-invariant Riemannian metric on $V_\HB$. We do this by the following construction. Let $g^\prime$ be a Riemannian metric on $V_\HB \setminus \{0\}$ induced from the polar coordinate map
\[
  V_\HB \setminus \{0\} \to S(V_\HB) \times (0, \infty).
\]
Here, the Riemannian metric on $S(V_\HB) \times (0, \infty)$ is defined using the restriction of $g$ to $S(V_\HB)$ and the standard Riemannian metric on $(0, \infty)$.
Take a smooth function $\map{\rho}{[0, \infty)}{[0, 1]}$ which is constantly 1 on $[0, 1/2]$ and supported in $[0, 1]$. Then, define a Riemannian metric $\bar g$ on $V_\HB$ by
\begin{align}\label{eq8:bar g}
  \bar g = (\rho \circ r) g + (1- \rho \circ r) g^\prime
\end{align}
where $\map{r}{V_\HB}{[0, \infty)}$ is the distance from the origin function.

We construct the deformation of \cref{eq8:before deform}. First, we use the Riemannian metric $\bar g$ on $V_\HB$. On the vector bundle
\[
  V_\HB \times S_1^{V_\HB} \cong V_\HB \times (\Lambda_\R^\ast T V_\HB \otimes_\R \C),
\]
we put a metric induced from $\bar g$. We denote this vector bundle with a metric by $\bar \mcalS_1^{V_\HB}$.
We also define
\[
  \map{c_{\text{pssho}}}{TV_\HB \times_{V_\HB} \bar \mcalS_1^{V_\HB}}{\bar \mcalS_1^{V_\HB}},\ h_{\text{pssho}} \in \Gamma(\Herm(\bar \mcalS_1^{V_\HB})), \map{\tau_{\text{pssho}}}{\bar \mcalS_1^{V_\HB}}{\bar \mcalS_1^{V_\HB}}
\]
by
\begin{align*}
  c_{\text{pssho}}(v, \tilde v, \omega \otimes z) & = (v, (\tilde v^\wedge - \tilde v^{\lrcorner_{\bar g}}) \omega \otimes z),                                                                                          \\
  h_{\text{pssho}}(v, \omega \otimes z)           & = (v, (v^\wedge + v^{\lrcorner_{\bar g}})\omega \otimes z),                                                                                                         \\
  \tau_{\text{pssho}}(v, \omega \otimes z)        & = (vj, \omega j \otimes \bar z),                                           &  & v \in V_\HB,\ \tilde v \in T_v V_\HB,\ \omega \in \Lambda_\R^\ast V_\HB,\ z \in \C.
\end{align*}
Here, $\lrcorner_{\bar g}$ denotes the interior product using $\bar g$. Let $\nabla_{\text{pssho}}$ be the Levi-Civita connection on $\bar \mcalS_1^{V_\HB}$. Now we have constructed the \textit{pseudo-super-symmetric harmonic oscillator}
\begin{align}\label{eq8:pseudo2}
  (\bar g, \bar \mcalS_1^{V_\HB}, c_{\text{pssho}}, \nabla_{\text{pssho}}, h_{\text{pssho}}, \tau_{\text{pssho}}).
\end{align}

Since $\bar g$ is translation-invariant, the pseudo-super-symmetric harmonic oscillator is also translation-invariant. The following lemma is clear from the construction.
\begin{lemma}\label{lem8:coincidence on B_1/2}
  The identity map on the total spaces of the vector bundles
  \[
    S_1^{V_\HB},\ \bar S_1^{V_\HB}
  \]
  is a metric-preserving isomorphism over $B_{1/2}(V_\HB)$ and intertwines $c$, $\nabla$, $h$, $\tau$.
\end{lemma}

The key property of the super-symmetric harmonic oscillator is that the kernel of
\[
  \mcalD + t h_1^{V_\HB}
\]
is one-dimensional for all $t > 0$. Here $\mcalD$ is the Dirac operator constructed from \cref{eq8:before deform}. This property is inherited by the pseudo-super-symmetric harmonic oscillator in the following sense.

\begin{proposition}\label{prop8:ker of ppsho}
  Let $\mcalD$ be the Dirac operator constructed from \cref{eq8:before deform}, and let $\mcalD_{\text{pssho}}$ be the Dirac operator constructed from \cref{eq8:pseudo2}. For sufficiently large positive real number $t$, there exists a $\tau$-commuting isomorphism
  \begin{align}\label{eq8:iso ssho}
    \Ker(\mcalD + t h_1^{V_\HB}) \cong \Ker(\mcalD_{\text{pssho}} + t h_{\text{pssho}}).
  \end{align}
  (The kernels are considered in the respective $L^2$ spaces.) This isomorphism is unique up to multiplication by positive real numbers.
\end{proposition}

The isomorphism is constructed by multiplying a cut-off function followed by the $L^2$-projection.

This proposition is used in the proof of \cref{prop5:can iso for vari spin}(1). For ease of reference in the proof of \cref{prop5:can iso for vari spin}(1), we give a name to the standard solution in $\Ker(\mcalD_{\text{pssho}} + t h_{\text{pssho}})$ obtained from \cref{eq8:iso ssho}.

\begin{definition}\label{def8:std sol of pssho}
  Let $t$ be sufficiently large so that \cref{prop8:ker of ppsho} holds. Denote by
  \[
    \bar b_{\text{pssho}}
  \]
  the element of $\Ker(\mcalD_{\text{pssho}} + t h_{\text{pssho}})$ with $L^2$ norm 1 corresponding to the standard solution
  \begin{align*}
    e^{-\frac{1}{2}t^2\abs{v^2}}, &  & v \in V_\HB
  \end{align*}
  in $\Ker(\mcalD + t h_1^{V_\HB})$.
\end{definition}

\section{Declaration of generative AI and AI-assisted technologies in the writing process}

During the preparation of this work the author used Claude 3 Opus, Sonnet and ChatGPT with version GPT3-5 in order to translate the manuscript originally written in Japanese by the author into English. After using this tool/service, the author reviewed and edited the content as needed and take full responsibility for the content of the published article.


\begin{thebibliography}{10}

  \bibitem{Baraglia-Obstructions-to-smooth-actions-2019}
  D.~Baraglia.
  \newblock Obstructions to smooth group actions on 4-manifolds from families {S}eiberg-{W}itten theory.
  \newblock {\em Adv. Math.}, 354:106730, 32, 2019.

  \bibitem{Baraglia-Constraints-on-families-2021}
  D.~Baraglia.
  \newblock Constraints on families of smooth 4-manifolds from {B}auer-{F}uruta invariants.
  \newblock {\em Algebr. Geom. Topol.}, 21(1):317--349, 2021.

  \bibitem{baraglia-mod-2023-arxiv}
  D.~Baraglia.
  \newblock The mod 2 {S}eiberg-{W}itten invariants of spin structures and spin families, 2023.

  \bibitem{Baraglia--Konno-gluing2020}
  D.~Baraglia and H.~Konno.
  \newblock A gluing formula for families {S}eiberg-{W}itten invariants.
  \newblock {\em Geom. Topol.}, 24(3):1381--1456, 2020.

  \bibitem{Baraglia--Konno-2022}
  D.~Baraglia and H.~Konno.
  \newblock On the {B}auer-{F}uruta and {S}eiberg-{W}itten invariants of families of 4-manifolds.
  \newblock {\em J. Topol.}, 15(2):505--586, 2022.

  \bibitem{Baraglia-Konno-Nielsen-K3}
  D.~Baraglia and H.~Konno.
  \newblock A note on the {N}ielsen realization problem for {$K3$} surfaces.
  \newblock {\em Proc. Amer. Math. Soc.}, 151(9):4079--4087, 2023.

  \bibitem{Chatterjee-gerbe}
  D.~S. Chatterjee.
  \newblock On gerbs.
  \newblock {\em PhD thesis, Trinity College, Cambridge}, 1998.

  \bibitem{Furuta-preprint}
  M.~Furuta.
  \newblock Stable homotopy version of {S}eiberg--{W}itten invariant.
  \newblock {\em preprint}.

  \bibitem{Furuta-11/8-inequality2001}
  M.~Furuta.
  \newblock Monopole equation and the {$\frac{11}8$}-conjecture.
  \newblock {\em Math. Res. Lett.}, 8(3):279--291, 2001.

  \bibitem{Furuta2007Index-Theorem}
  M.~Furuta.
  \newblock {\em Index theorem. 1}, volume 235 of {\em Translations of Mathematical Monographs}.
  \newblock American Mathematical Society, Providence, RI, 2007.
  \newblock Translated from the 1999 Japanese original by Kaoru Ono, Iwanami Series in Modern Mathematics.

  \bibitem{Hitchin-Lag-submfd-2001}
  N.~Hitchin.
  \newblock Lectures on special {L}agrangian submanifolds.
  \newblock In {\em Winter {S}chool on {M}irror {S}ymmetry, {V}ector {B}undles and {L}agrangian {S}ubmanifolds ({C}ambridge, {MA}, 1999)}, volume~23 of {\em AMS/IP Stud. Adv. Math.}, pages 151--182. Amer. Math. Soc., Providence, RI, 2001.

  \bibitem{Kato--Konno--Nakamura-Rigidity-mod2-family-SW-2021}
  T.~Kato, H.~Konno, and N.~Nakamura.
  \newblock Rigidity of the mod 2 families {S}eiberg-{W}itten invariants and topology of families of spin 4-manifolds.
  \newblock {\em Compos. Math.}, 157(4):770--808, 2021.

  \bibitem{Konno-char-class2021}
  H.~Konno.
  \newblock Characteristic classes via 4-dimensional gauge theory.
  \newblock {\em Geom. Topol.}, 25(2):711--773, 2021.

  \bibitem{Konno--Lin-2022-arXiv}
  H.~Konno and J.~Lin.
  \newblock Homological instability for moduli spaces of smooth 4-manifolds.
  \newblock {\em arXiv}, 2022.

  \bibitem{Konno--Nakamura-constraints2023}
  H.~Konno and N.~Nakamura.
  \newblock Constraints on families of smooth 4-manifolds from {${\rm Pin}^{-}$}(2)-monopole.
  \newblock {\em Algebr. Geom. Topol.}, 23(1):419--438, 2023.

  \bibitem{konno--taniguchi2021groups}
  H.~Konno and M.~Taniguchi.
  \newblock The groups of diffeomorphisms and homeomorphisms of 4-manifolds with boundary.
  \newblock {\em Adv. Math.}, 409:Paper No. 108627, 58, 2022.

  \bibitem{Kronheimer--Mrowka-Dehn-twist-K3-2020}
  P.~B. Kronheimer and T.~S. Mrowka.
  \newblock The {D}ehn twist on a sum of two {$K3$} surfaces.
  \newblock {\em Math. Res. Lett.}, 27(6):1767--1783, 2020.

  \bibitem{Li--Liu-family-wall-crossing2001}
  T.-J. Li and A.-K. Liu.
  \newblock Family {S}eiberg-{W}itten invariants and wall crossing formulas.
  \newblock {\em Comm. Anal. Geom.}, 9(4):777--823, 2001.

  \bibitem{Lin-nonsymplectic-loops-2022-arXiv}
  J.~Lin.
  \newblock The {F}amily {S}eiberg-{W}itten {I}nvariant and nonsymplectic loops of diffeomorphisms.
  \newblock {\em arXiv}, 2022.

  \bibitem{Liu-family-blow-up2000}
  A.-K. Liu.
  \newblock Family blowup formula, admissible graphs and the enumeration of singular curves. {I}.
  \newblock {\em J. Differential Geom.}, 56(3):381--579, 2000.

  \bibitem{miyazawa2021localization}
  J.~Miyazawa.
  \newblock Localization of a ${KO}^{\ast}(\text{pt})$-valued index and the orientability of the ${P}in^-(2)$ monopole moduli space, 2021.

  \bibitem{Morgan--Szabo-homotopy-K3-1997}
  J.~W. Morgan and Z.~Szab\'{o}.
  \newblock Homotopy {$K3$} surfaces and mod {$2$} {S}eiberg-{W}itten invariants.
  \newblock {\em Math. Res. Lett.}, 4(1):17--21, 1997.

  \bibitem{Nakamura-family-SW2003}
  N.~Nakamura.
  \newblock The {S}eiberg-{W}itten equations for families and diffeomorphisms of 4-manifolds.
  \newblock {\em Asian J. Math.}, 7(1):133--138, 2003.

  \bibitem{Ruberman-an-obstruction1998}
  D.~Ruberman.
  \newblock An obstruction to smooth isotopy in dimension {$4$}.
  \newblock {\em Math. Res. Lett.}, 5(6):743--758, 1998.

  \bibitem{Ruberman-polynomial-invarint-of-diffeo-1999}
  D.~Ruberman.
  \newblock A polynomial invariant of diffeomorphisms of 4-manifolds.
  \newblock In {\em Proceedings of the {K}irbyfest ({B}erkeley, {CA}, 1998)}, volume~2 of {\em Geom. Topol. Monogr.}, pages 473--488. Geom. Topol. Publ., Coventry, 1999.

  \bibitem{Ruberman-positive-scalar-curvature-2001}
  D.~Ruberman.
  \newblock Positive scalar curvature, diffeomorphisms and the {S}eiberg-{W}itten invariants.
  \newblock {\em Geom. Topol.}, 5:895--924, 2001.

  \bibitem{Szymik-Characteristic-cohomotopy-classes-family-2010}
  M.~Szymik.
  \newblock Characteristic cohomotopy classes for families of 4-manifolds.
  \newblock {\em Forum Math.}, 22(3):509--523, 2010.

  \bibitem{Witten-monopoles-1994}
  E.~Witten.
  \newblock Monopoles and four-manifolds.
  \newblock {\em Math. Res. Lett.}, 1(6):769--796, 1994.

  \bibitem{Zhang-Witten-deformation-2001}
  W.~Zhang.
  \newblock {\em Lectures on {C}hern-{W}eil theory and {W}itten deformations}, volume~4 of {\em Nankai Tracts in Mathematics}.
  \newblock World Scientific Publishing Co., Inc., River Edge, NJ, 2001.

\end{thebibliography}
\end{document}